\documentclass{amsart}

\usepackage{amssymb, graphicx}

\newcommand{\ncmnd}{\newcommand}
\ncmnd{\nthm}{\newtheorem}
\ncmnd{\R}{\mathbf{R}}

\theoremstyle{plain}

\nthm{thm}{Theorem}[section]
\nthm{cor}[thm]{Corollary}
\nthm{prop}[thm]{Proposition}
\nthm{lem}[thm]{Lemma}

\theoremstyle{definition}
\nthm{note}[thm]{Note}
\nthm{rem}[thm]{Remark}
\nthm{defn}[thm]{Definition}
\nthm{examp}[thm]{Example}
\nthm{prob}[thm]{Problem}

\DeclareMathOperator{\grad}{grad}

\makeindex

\begin{document}

\title{Riemannian geometry}

\author{Richard L. Bishop}

\maketitle

\begin{center}
{\bf \Large Preface}
\end{center}
These lecture notes are based on the course in Riemannian geometry
at the University of Illinois over a period of many years. The material
derives from the course at MIT developed by Professors Warren
Ambrose and I M Singer and then reformulated in the book by Richard
J. Crittenden and me, ``Geometry of Manifolds'', Academic Press, 1964.
That book was reprinted in 2000 in the AMS Chelsea series. These notes
omit the parts on differentiable  manifolds, Lie groups, and isometric
imbeddings. The notes are not intended to be for individual self study;
instead they depend heavily on an instructor's guidance and the use of
numerous problems with a wide range of difficulty.

The geometric structure in this treatment emphasizes the use of the
bundles of bases and frames and avoids the arbitrary coordinate expressions
as much as possible. However, I have added some material of historical
interest, the Taylor expansion of the metric in normal coordinates which
goes back to Riemann. The curvature tensor was probably discovered by
its appearance in this expansion.

There are some differences in names which I believe are a substantial
improvement over the fashionable ones. What is usually called a ``distribution''
is called a ``tangent subbundle'', or ``subbundle'' \index{subbundle} for short. The name ``solder
form'' \index{solder form} never made much sense to me and is now labeled the descriptive term
``universal cobasis''\index{universal cobasis}. The terms ``first Bianchi identity'' and ``second Bianchi\index{cobasis!universal}
identity'' are historically inaccurate and are replaced by ``cyclic curvature identity''
and ``Bianchi identity'' -- Bianchi was too young to have anything to do with the
first, and even labeling the second with his name is questionable since it appears
in a book by him but was discovered by someone else (Ricci?).

The original proof of my Volume Theorem used Jacobi field comparisons and
is not reproduced. Another informative approach is to use comparison theory
for matrix Riccati equations and a discussion of how that works is included
and used to prove the Rauch Comparison Theorem.

In July, 2013, I went through the whole file, correcting many typos, making
minor additions, and, most importantly, redoing the index using the Latex option for that purpose.

\vspace{\medskipamount}
\noindent Richard L. Bishop

\noindent University of Illinois at Urbana-Champaign

\noindent July, 2013.

\tableofcontents

\section{Riemannian metrics}

Riemannian geometry considers manifolds with the additional structure of
a Riemannian metric, a type $(0,2)$ positive definite symmetric tensor
field. To a first order approximation this means that a Riemannian
manifold is a Euclidean space: we can measure lengths of vectors and
angles between them. Immediately we can define the length of curves by
the usual integral, and then the distance between points comes from the
glb of lengths of curves. Thus, Riemannian manifolds become metric
spaces in the topological sense.

\vspace{\medskipamount}
Riemannian geometry is a subject of current mathematical research in
itself. We will try to give some of the flavor of the questions being
considered now and then in these notes. A Riemannian structure
is also frequently used as a tool for the study of other properties of
manifolds. For example, if we are given a second order linear 
elliptic partial
differential operator, then the second-order coefficients of that
operator form a Riemannian metric and it is natural to ask what the
general properties of that metric tell us about the partial differential
equations involving that operator. Conversely, on a Riemannian manifold
there is singled out naturally such an operator (the Laplace-Beltrami \index{Laplace-Beltrami}
operator),
so that it makes sense,
for example, to talk about solving the heat equation\index{heat equation} on a Riemannian
manifold. The Riemannian metrics have nice properties not shared by just
any topological metrics, so that in topological studies they are also
used as a tool for the study of manifolds.

\vspace{\medskipamount}
The generalization of Riemannian geometry to the case where the metric
is not assumed to be positive definite, but merely nondegenerate, forms
the basis for general relativity theory\index{general relativity}. We will not go very far in that
direction, but some of the major theorems and concepts are identical in
the generalization. We will be careful to point out which theorems we
can prove in this more general setting. For a deeper study there is a
fine book: O'Neill\index{O'Neill, B.}, Semi-Riemannian geometry, Academic Press, 1983. I
recommend this book also for its concise summary of the theory of
manifolds, tensors, and Riemannian geometry itself.

The first substantial question we take up is the existence of Riemannian
metrics. It is interesting that we can immediately use Riemannian
metrics as a tool to shed some light on the existence of semi-Riemannian
metrics of nontrivial index.

\vspace{\medskipamount}
\begin{thm}[Existence of Riemannian metrics] \index{Riemannian metrics} On any smooth
manifold there exists a Riemannian metric.
\end{thm}

\vspace{\medskipamount}
The key idea of the proof is that locally we always have Riemannian
metrics carried over from the standard one on Cartesian space by
coordinate mappings, and we can glue them together smoothly with a
partition of unity\index{partition of unity}. In the gluing process the property of being positive
definite is preserved due to the convexity of the set of positive
definite symmetric matrices.
What happens for indefinite metrics? The set of nonsingular symmetric
matrices of a given index is {\em not} convex, so that the existence
proof breaks down. In fact, there is a condition on the manifold, which
can be reduced to topological invariants, in order that a
semi-Riemannian metric of index $\nu$ exist: there must be a subbundle
of the tangent bundle of rank $\nu$. When $\nu = 1$ the structure is
called a {\em Lorentz } structure\index{Lorentz structure}; that is the case of interest in
general relativity theory; the topological condition for a compact manifold to
have a Lorentz structure is easily understood: the Euler characteristic\index{Euler characteristic}
must be $0$.

The proof in the Lorentz case can be done by using the fact that for any
simple curve there is a diffeomorphism which is the identity outside any
given neighborhood of the curve and which moves one end of the curve to
the other. In the compact case start with a vector field having discrete
singularities. Then by choosing disjoint simple curves from these singularities
to inside a fixed ball, we can obtain a diffeomorphism which moves all of
them inside that ball. If the Euler characteristic is $0$, then by the Hopf
index theorem\index{Hopf index theorem}, the index of the vector field on the boundary of the ball
is $0$, so the vector field can be extended to a nonsingular vector field
inside the ball.

Conversely, by the following Theorem \ref{thm:metricsexist}, a Lorentz metric would give a
nonsingular rank $1$ subbundle. If that field is nonorientable, pass to the
double covering for which the lift of it is orientable. Then there is a
nonsingular vector field which is a global basis of the line field, so the
Euler characteristic is $0$.\index{Euler characteristic!Lorentz manifold}
\index{orientable}

In the noncompact case, take a countable exhaustion of the manifold
by an increasing family of compact sets. Then the singularities of a
vector field can be pushed outside each of the compact sets sequentially,
leaving a nonsingular vector field on the whole in the limit. Thus, every
noncompact (separable) manifold has a Lorentz structure.

\vspace{\medskipamount}
\begin{thm}[Existence of semi-Riemannian metrics]\index{semi-Riemannian metric}
\label{thm:metricsexist}
A smooth manifold has a semi-Riemannian metric of index
$\nu$ if and only if there is a subbundle of the tangent bundle of
rank $\nu$.
\end{thm}

\vspace{\medskipamount}

The idea of the proof is: the subbundle will be the directions in which
the semi-Riemannian metrics will be negative definite. If we change the
sign on the subbundle and leave it unchanged on the orthogonal
complement, we will get a Riemannian metric. The construction goes both
ways.

\vspace{\medskipamount}

Although the idea for the proof of Theorem \ref{thm:metricsexist} is correct, there are some nontrivial
technical difficulties to entertain us. One direction is relatively
easy.

\vspace{\smallskipamount}

{\em Proof of ``if'' part} If $M$ has a smooth tangent subbundle $V$ of rank 
$\nu$, then $M$ has a semi-Riemannian metric of index $\nu$.

\begin{defn} A {\em frame} at a point $p$ of a semi-Riemannian manifold
is a basis of the tangent space $M_p$ with respect to which the
component matrix of the metric tensor is diagonal with $-1$'s followed
by 1's on the diagonal. A {\em local frame field} is a local basis of
vector fields which is a frame at each point of its domain.
\end{defn}

\vspace{\medskipamount}

\begin{lem}[Technical Lemma 1] Local frames exist in a neighborhood of 
every point.\index{frame field!local}
\end{lem}

TL 1 is important for other purposes than the proof of the theorem at
hand. For the proof of TL 1 one modifies the Gramm-Schmidt procedure,
starting with a smooth local basis and shrinking the domain at each
step if necessary to divide by the length for normalization.

\vspace{\medskipamount}
\begin{rem} If we write $g = g_{ij}\omega^i\omega^j$, then a local
frame is exactly one for which the coframe of $1$-forms $(\epsilon^i)$
satisfies\index{coframe}

$$g=-(\epsilon^1)^2-\cdots-(\epsilon^\nu)^2+\cdots+(\epsilon^n)^2.$$

The Gramm-Schmidt \index{Gramm-Schmidt procedure}procedure amounts to iterated completion of squares,
viewing $g$ as a homogeneous quadratic polynomial in the $\omega^i$.
The modifications needed to handle the negative signs are probably
easier in this form.
\end{rem}

\vspace{\medskipamount}

\begin{lem}[Technical Lemma 2] If $V$ is a smooth tangent subbundle of 
rank $\nu$ and $g$ is a Riemannian metric, then the $(1,1)$ tensor
field $A$ and the semi-Riemannian metric $h$ given as follows are
smooth:

\begin{equation}A=
\begin{cases} -1&\text{ on } V\\
1&\text{ on } V^\perp.
\end{cases}
\end{equation}
$$h(v,w)=g(Av,w).$$
\end{lem}

(Their expressions in terms of smooth local frames {\em adapted to V}
for $g$ are {\em constant}, hence smooth.)

Now the converse.

{\em Proof of``only if'' part} If there is a semi-Riemannian metric $h$ of 
index $\nu$, then there is a tangent subbundle $V$ of rank $\nu$.

The outline of the proof goes as follows. Take a Riemannian metric
$g$. Then $h$ and $g$ are related by a $(1,1)$ tensor field $A$ as
above. We know that $A$ is symmetric with respect to $g$-frames, and
has $\nu$ negative eigenvalues (counting multiplicities) at each
point. Thus, the subspace spanned by the eigenvectors of these
negative eigenvalues at each point is $\nu$-dimensional. The claim is
that those subspaces form a smooth subbundle, even though it may be
impossible to choose the individual eigenvectors to form smooth vector
fields.
\begin{lem}[Technical Lemma 3] If $A:V\to V$ is a symmetric linear 
operator, smoothly dependent on coordinates $x^1,\ldots,x^n$, and
$\lambda_0$ is a simple eigenvalue at the origin, then $\lambda_0$
extends to a smooth simple eigenvalue function in a neighborhood of
the origin having a smooth eigenvector field.
\end{lem}

Let $P(X) = \det (XI-A)$. Use the implicit function theorem to solve
$P(X)=0$, getting $X=\lambda (x^1,\ldots,x^n)$. Then we can write
$P(X) = (X-\lambda )Q(X)$, and any nonzero column of $Q(A)$ is an
eigenvector, by the Cayley-Hamilton theorem.\index{Cayley-Hamilton theorem}

\begin{lem}[Technical Lemma 4] Suppose that
$W:\R^n \to \bigwedge^\nu V$ is a smooth function with
decomposable, nonzero values. Then locally there are smooth
vector fields having wedge product equal to $W$.
\end{lem}
For $\omega \in \bigwedge^{\nu-1}V^*$, the interior product
$i(\omega)W$ is always in the subspace carried by $W$. Choose $\nu$ of
these $\omega$'s which give linearly independent interior products
with $W$ at one point; then they do so locally.

\begin{lem}[Technical Lemma 5] If $A:V\to V$ is a symmetric linear 
operator of index $\nu$, smoothly dependent on coordinates $x^i$, then
the extension of $A$ to a {\em derivation} of the Grassmann algebra\index{Grassmann algebra}
$\bigwedge^*V\to\bigwedge^*V$ has a unique minimum simple eigenvalue
$\lambda_1+\cdots+\lambda_\nu$ on $\bigwedge^\nu V$. The (smooth!)
eigenvectors $W:\R^n\to \bigwedge^\nu V$ are decomposable.
\end{lem}

\vspace{\medskipamount}
\begin{prob}\label{prob1} Generalize the result of TL's 3, 4, 5: If there 
is a group of $\nu$ eigenvalues $\lambda_1,\ldots,\lambda_\nu$ of
$A:V\to V$ which always satisfy $a < \lambda_i < b$, then the subspace
spanned by their eigenvectors is smooth. \end{prob}

\vspace{\medskipamount}
(Consider $B=(aI-A)(bI-A)$. Can $a,b$ be continuous functions of 
$x^1,\ldots,x^n$ too?)

\vspace{\medskipamount}
\begin{prob}\label{prob2} Now order {\em all} of the eigenvalues of 
symmetric smooth $A:V\to V,\ \lambda_1\le\lambda_2\le\cdots\le\lambda_n$, 
defining uniquely $n$
functions $\lambda_i$ of $x^1,\ldots,x^n$. Prove that the $\lambda_i$
are continuous; on the subset where
$\lambda_{i-1}<\lambda_i<\lambda_{i+1}$, $\lambda_i$ is smooth and has
locally smooth eigenvector fields.\end{prob}

\vspace{\medskipamount}
\begin{prob} \label{prob3}Construct an example 
$$A = \begin{pmatrix}f(x,y)&g(x,y)\\ g(x,y)&h(x,y)
\end{pmatrix}$$
for which $\lambda_1(0,0)=\lambda_2(0,0)=1$ and neither $\lambda_1, \lambda_2$
nor their eigenvector fields are smooth in a neighborhood of
$(0,0)$.\end{prob}

\vspace{\medskipamount}
\begin{rem} I have had to referee and reject two papers because the
proofs were based on the assumption that eigenvector fields could be
chosen smoothly. Take care!
\end{rem}

\noindent \textbf{Hard Problem.}
In Problem \ref{prob3}, can it be arranged so that there 
is no smooth eigenvector field on the set where $\lambda_1<\lambda_2$?
If that set is simply connected, then there is a smooth vector field;
and in any case there is always a smooth subbundle of rank 1.
However, in the non-simply-connected case, the subbundle may be
disoriented in passing around some loop.

\subsection{Finsler Metrics} \index{Finsler metric}Let $L:TM\to \R$ be as continuous function which is
positive on nonzero vectors and is {\em positive homogeneous of degree 1}:
$${\rm for\ } \alpha \in \R, v \in TM, \hbox{we have } L(\alpha v) = 
|\alpha|
L(v).$$
Eventually we should also assume that the {\em unit balls}, that is , the 
subsets
of the tangent spaces on which $L\le 1$, are convex, but that property is not
required to give the initial facts we want to look at here. In fact, it is 
usually assumed that $L^2$ is smooth and that its restriction to each tangent
space has positive definite second derivative matrix (the {\em Hessian} of 
$L^2$)
with respect to linear coordinates on that tangent space. This implies that the
unit balls are strictly convex and smooth. 
 
A function $L$ satisfying these conditions is called a {\em Finsler metric} on M.
If $g$ is a Riemannian metric on M, then there is a corresponding Finsler 
metric,
given by the norm with respect to $g$: $L(v) = \sqrt {g(v,v)}$. We use the 
letter
$L$ because it is treated like a Lagrangian in mechanics.

Finsler metrics were systematically studied by P. Finsler starting about 1918.\index{Finsler, P.}

\vspace{\medskipamount}
\subsection{Length of Curves} For a piecewise $C^1$ curve $\gamma:[a,b]\to M$ we 
define\index{length}
the {\em length} of $\gamma$ to be
$$|\gamma| = \int_a^b L(\gamma '(t))\/dt.$$
The reason for assuming that $L$ is positive homogeneous is that it makes the 
length of a curve independent of its parametrization. This follows easily
from the change of variable formula for integrals.

Length is clearly additive with respect to chaining of curves end-to-end.
Not every Finsler metric comes from a Riemannian metric. The condition for that
to be true is the {\em parallelogram law}\index{parallelogram law}, well-known to functional analysts:
\begin{thm}[Characterization of Riemannian Finsler Metrics]
A Finsler metric $L$
is the Finsler metric of a Riemannian metric g if and only if it satisfies the
parallelogram law:
$$L^2(v+w) + L^2(v-w) = 2L^2(v) + 2L^2(w),$$
for all tangent vectors $v,w$ at all points of $M$. 
When this law is satisfied, the Riemannian metric can be recovered from $L$ by the {\em polarization identity}\/:\index{polarization}
$$2g(v,w) = L^2(v+w) - L^2(v) - L^2(w).$$
\end{thm}

\subsection{Distance}\index{distance} When we have a notion of lengths of curves satisfying the 
additivity with respect to chaining curves end-to-end, as we do when we have
a Finsler metric, then we can define the {\em intrinsic metric} (the word 
metric is used here as it is in topology, a distance function) by specifying
the {\em distance from p to q} to be\index{metric!intrinsic}
$$d(p,q) = \inf \{|\gamma|: \gamma \hbox{is a curve from $p$ to $q$}\}.$$
Generally this function is only a semi-metric, in that we could have
$d(p,q) = 0$ but not $p=q.$ The symmetry and the triangle inequality
are rather easy consequences of the definition, but the nondegeneracy
in the case of Finsler lengths of curves is nontrivial.
\begin{thm}[Topological Equivalence Theorem] If $d$ is the intrinsic
metric coming from a Finsler metric, then $d$ is a topological metric
on $M$ and the topology given by $d$ is the same as the manifold
topology.\index{topology of Finsler manifold}
\end{thm}

If $M$ is not connected, then the definition gives $d(p,q) = \infty$
when $p$ and $q$ are not in the same connected components. We
simply allow such a value for $d;$ it is a reasonable extension of
the notion of a topological metric.
\begin{lem}[Nondegeneracy Lemma] If $\varphi:U\to \R^n$ is a
coordinate map, where $U$ is a {\em compact} subset of $M$, then
there are positive constant $A$, $B$ such that for every curve
$\gamma$ in $U$,
$$|\gamma| \ge A|\varphi \circ \gamma|, \qquad |\varphi \circ
\gamma|\ge B|\gamma|.$$
Here $|\varphi \circ \gamma|$ is the standard Euclidean length of
$\varphi \circ\gamma.$
\end{lem}

The key step in the proof is to observe that the union $EU$ of the
unit (with respect to the Euclidean coordinate metric) spheres at
points of $U$ forms a {\em compact} subset of the tangent bundle $TM$.
Since $L$ is positive on nonzero vectors and continuous on $TM,$ on
$EU$ we have a positive minimum $A$ and a finite maximum $1/B$
for $L$ on $EU$. 

\vspace{\medskipamount}
The result breaks down on infinite-dimensional manifolds modeled
on a Banach space,\index{Banach space} because there the set of unit vectors $EU$ will
not be compact. So in that case the inequalities $L(v) \ge A||v||,$
and $||v|| \ge BL(v)$ must be taken as a local hypothesis on $L$.

\vspace{\medskipamount}
The nondegeneracy of $d$ uses the lemma in an obvious way, although
there is a subtlety that could be overlooked: the nondegeneracy of the
intrinsic metric on $\R^n$ defined by the standard Euclidean
metric must be proved independently.

\vspace{\medskipamount}
Aside from the Hausdorff separation axiom,\index{Hausdorff separation axiom} the topology of a manifold
is determined as a local property by the coordinate maps on compact
subsets $U$. Since the nondegeneracy lemma tells us that there is
nesting of d-balls and coordinate-balls, the topologies must coincide.

\begin{prob}\label{prob4} Prove that the intrinsic metric on $\R^n$
defined by the standard Euclidean metric is nondegenerate, and, in
fact, coincides with the usual distance formula.\end{prob}

Hint: For a curve $\gamma$ from $p$ to $q$, split the tangent vector
$\gamma'$ into components parallel to $\vec {pq}$ and perpendicular
to $\vec {pq}$. The integral of the parallel component is always at
least the usual distance.

\vspace{\medskipamount}
\subsection{Length of a curve in a metric space}\index{length} If we have a topological
metric space $M$ with distance function $d$ and a curve $\gamma
:[a,b]\to M$, then the {\em length} of $\gamma$ is 
$$|\gamma| = \sup \sum d(\gamma(t_i),\gamma(t_{i+1})),$$
where the supremum is taken over all partitions of the interval
$[a,b]$. Due to the triangle inequality, the insertion of another point
into the partition does not decrease the sum, so that in particular
the length of a curve from $p$ to $q$ is at least $d(p,q)$. If the
length is finite, we call the curve {\em rectifiable}. It is obvious
from the definition that the length of a curve is independent of its
parametrization and satisfies the additivity property with respect to
chaining curves.
 Following the definition of the length of a curve, we can then define the
{\em intrinsic metric generated by $d$}, as we defined the intrinsic metric
for a Finsler space. Clearly the intrinsic metric is at least as great as
the metric $d$ we started with.

\section{Minimizers}

A curve whose length equals the intrinsic distance
between its endpoints is called a {\em minimizer}\index{minimizer} or a {\em shortest path}. In Riemannian 
manifolds they are often called {\em geodesics},\index{geodesic} but we will avoid that
term for a while because there is another definition for geodesics. One
of our major tasks will be to establish the relation between minimizers
and geodesics. They will turn out to be not quite the same:
geodesics turn out to be only {\em locally} minimizing, and furthermore,
there are technical reasons why we should require that geodesics
have a special parametrization, a constant-speed parametrization.

\subsection{Existence of Minimizers} We can establish the existence of
minimizers for a Finsler metric within a connected compact set
by using convergence techniques developed by nineteenth century
mathematicians to show the existence of solutions of ordinary
differential equations by the Euler method.\index{Euler method} 
The main tool is the
Arzel\`a-Ascoli Theorem.\index{Arzel\`a-Ascoli Theorem}
Cesare Arzel\`a (1847-1912)\index{Arzel\`a, C.} was a
professor at Bologna, who established the case where the domain is a
closed interval, and Giulio Ascoli (1843-1896)\index{Ascoli, G.} was a professor at
Milan, who formulated the definition of uniform equicontinuity.\index{uniform equicontinuity}

\begin{defn} A collection of maps $\mathcal{F} = \{f:M\to N\}$ from a
metric space $M$ to a metric space $N$ is {\em uniformly equicontinuous}
if for every $\epsilon > 0$, there is $\delta > 0$ such that for every
$f \in \mathcal{F}$ and every $x, y \in M$ such that $d(x,y) < \delta$ we have 
$d(f(x), f(y)) < \epsilon$.
\end{defn}

The word ``uniform'' refers to the quantification over all $x, y$, just as
in ``uniform continuity'', while ``equicontinuous''  refers to the
quantification over all members of the family. The definition is only
significant for infinite families of functions, since a finite family of
uniformly continuous functions is always also equicontinuous.
Moreover, the application is usually to the case where $M$ is compact,
so that uniform continuity, but not uniform equicontinuity,
is automatic. 
If $M$ and $N$ are subsets of Euclidean spaces and the members of
the family have a uniform bound on their gradient vector lengths,
then the family is uniformly equicontinuous. Even if those gradients
exist only piecewise, this still works, which explains why the
theorem below can be used to get existence of solutions by the Euler
method.

\begin{thm}[Arzel\`a-Ascoli Theorem] Let $K,K'$ be compact metric
space and assume that $K$ has a countable dense subset. Let $\mathcal{F}$
be a collection of continuous functions $K\to K'$. Then the following
properties are equivalent:

(a) $\mathcal{F}$ is uniformly equicontinuous.

(b) Every sequence in $\mathcal{F}$ has a subsequence which is
uniformly convergent on $K$.
\end{thm}

\vspace{\medskipamount}
There is a proof for the case where $K$ and $K'$ are subsets of
Euclidean spaces given in R. G.~Bartle,\index{Bartle, R.G.} ``The Elements of Real Analysis'',
2nd Edition, Wiley, page 189. No essential changes are needed for this
abstraction to metric spaces.

\begin{thm}[Local Existence of minimizers in Finsler Spaces] 
\index{minimizer!existence}If
$M$ is a Finsler manifold, each $p\in M$ has nested neighborhoods
$U\subset V$ such that every pair of points in $U$ can be joined by a
minimizer which is contained in $V$.
\end{thm}

To start the proof one takes $V$ to be a compact coordinate ball
about $p$, and $U$ a smaller coordinate ball so that, using the
curve-length estimates from the Nondegeneracy Lemma, any curve
which starts and ends in $U$ cannot go outside $V$ unless it has
greater length than the longest coordinate straight line in $U$. 
Now for two points $q$ and $r$ in $U$ we define a family of curves
$\mathcal{F}$ parametrized on the unit interval $[0,1] = I$ (so that $K$ of
the theorem will be $I$ with the usual distance on the line). We
require that the length of each curve to be no more than the Finsler
length of the coordinate straight line from $q$ to $r$. We parametrize
each curve so that it has constant ``speed'', which together with
the uniform bound on length, gives the uniform equicontinuity of
$\mathcal{F}$. We take $K'$ to be the outer compact ball $V$.

\vspace{\smallskipamount}
By the definition of intrinsic distance $d(q,r)$ there exists a
sequence of curves from $q$ to $r$ for which the lengths converge to
the infimum of lengths. If the coordinate straight line is already
minimizing, then we have our desired minimizer. Otherwise the
lengths will be eventually less than that of the straight line, and
from there on the sequence will be in $\mathcal{F}$. By the AA Theorem
\index{AA -- Arzel\`a-Ascoli}
there must be a uniformly convergent subsequence. It is fairly easy to
prove that the limit is a minimizer.

\begin{rem} The minimizers do not have to be unique, even 
locally. For an example consider the $l_1$ or $l_\infty$ norm on $R^2$
to define the Finsler metric (the ``taxicab metric'').\index{taxicab metric} Generally the
local uniqueness of minimizers can only be obtained by assuming
that the unit tangent balls of the Finsler metric are strictly convex.
We won't do that part of the theory in such great generality, but we
will obtain the local uniqueness in Riemannian manifolds by using
differentiability and the calculus of variations.
\end{rem}
 
The result on existence of minimizers locally can be abstracted a
little more.  Instead of using coordinate balls we can just assume
that the space is locally compact. Thus, in a locally compact space
with intrinsic metric there are always minimizers locally. The
proof is essentially the same.

\vspace{\medskipamount}

\subsection{Products} If we have two metric spaces $M$ and $N$, then the
Cartesian product has a metric $d_{M\times N}$, whose square is
$d_M^2 + d_N^2$. \index{product mteric}We use the same idea to get a Finsler metric or a
Riemannian metric on the product when we are given those structures on
the factors. When we pass to the intrinsic distance function of a
Finsler product there is no surprise, the result is the product
distance function.

\vspace{\medskipamount}
\begin{prob}\label{prob5} Prove that the projection into each factor of a
minimizer is a minimizer. \index{minimizer!product space}Conversely, if we handle the
parametrizations correctly, then the product of two minimizers is
again a minimizer.\end{prob}

\section{Connections}\index{connection}
 
We define an additional structure to a manifold
structure, called a connection. It can be given either in terms of
covariant derivative operators $D_X$ or in terms of a horizontal
tangent subbundle on the bundle of bases. Both ways are important, so
we will give both and establish the way of going back and forth.
Eventually our goal is to show that on a semi-Riemannian manifold
(including the Riemannian case) there is a canonical connection called
the Levi-Civita connection.
\index{Levi-Civita, T.}

\subsection{Covariant Derivatives}\index{covariant derivative} For a connection in terms of
covariant derivatives we give axioms for the operation $(X,Y) \to D_XY$,
associating a vector field $D_XY$ to a pair of vector fields $(X,Y)$.

1. If $X$ and $Y$ are $C^k$, then $D_XY$ is $C^{k-1}$.

2. $D_XY$ is $\mathcal{F}$-linear in $X$, for $\mathcal{F} =$ real-valued
functions on $M$. That is, $D_{fX}Y = fD_XY$ and $D_{X+X'}Y = D_XY +
D_{X'}Y$.

3. $D_X$ is a derivation with respect to multiplication by
elements of $\mathcal{F}$. That is, $D_X(fY) = (Xf)Y + fD_XY$ and $D_X(Y+Y') = D_XY + D_XY'$.

\vspace{\medskipamount}

We say that $D_XY$ is the {\em covariant derivative of Y with respect
to X}. It should be thought of as an extension of the defining
operation $f \to Xf$ of vector fields so as to operate on vector fields
$Y$ as well as functions $f$.

\subsection{Pointwise in $X$ Property} Axiom 2 conveys the information that for
a vector $x \in M_p$, we can define $D_xY \in M_p$. We extend $x$ to a
vector field $X$ and prove, using 2, that $(D_XY)(p)$ is the same for
all such extensions $X$.

\subsection{Localization} If $Y$ and $Y'$ coincide on an open set, then
$D_XY$ and $D_XY'$ coincide on that open set.

\subsection{Basis Calculations} If $(X_1,\ldots,X_n) = B$ is a local basis of
vector fields, we define $n^2$ $1$-forms $\varphi_j^i$ locally by
$$D_XX_j = \sum_{i=1}^n \varphi_j^i(X)X_i.$$
These $\varphi_j^i$ are called the {\em connection $1$-forms with respect
\index{connection!$1$-forms}
to the local basis B}. If we let $\omega^1,\ldots,\omega^n$ be the dual
basis to $B$ of $1$-forms, and arrange them in a column $\omega$, and let
$\varphi = (\varphi_j^i)$, then we can specify the connection locally
in terms of $\varphi$ by
$$D_XY = B(X + \varphi (X))\omega(Y).$$

\subsection{Parallel Translation} If $D_xY = 0$, we say that $Y$ is 
\index{parallel translation}{\em infinitesimally parallel} in the direction $x$. 
If $\gamma:[a,b] \to M$
is a curve and $D_{\gamma '(t)}Y = 0$ for all $t\in [a,b]$, we say that
$Y$ is {\em parallel along} $\gamma$ and that $Y(\gamma(b))$ is the
{\em parallel translate} of $Y(\gamma(a))$ along $\gamma$.

\begin{thm}[Existence, Uniqueness, and Linearity of Parallel Translation]
For a given curve $\gamma:[a,b]\to M$ and a vector $y \in
M_{\gamma(a)}$, there is a unique parallel translate of $y$ along
$\gamma$ to $\gamma(b)$.  This operation of parallel translation along
$\gamma, M_{\gamma(a)}\to M_{\gamma(b)}$, is a linear transformation.
\end{thm}

If $\gamma$ lies in a local basis neighborhood of $B$, then the
coefficients $\omega(Y\circ\gamma)$ of $Y$ along $\gamma$, when $Y$ is
parallel, satisfy a system of linear homogeneous differential
equations:
$$\frac{d\omega(Y\circ\gamma)}{dt} = -\varphi(\gamma'(t))
\omega(Y\circ\gamma).$$
Globally we chain together the local parallel translations to span all
of $\gamma$ in finitely many steps, using the fact that $\gamma([a,b])$
is compact.

\subsection{Existence of Connections} 
\index{connection!existence} It is trivial to check that for a local
basis $B$ we can set $\varphi = 0$ and obtain a connection in the local
basis neighborhood. Then if $\{U_\alpha\}$ covers $M$, $D^\alpha$ is a
connection on $U_\alpha$, and $\{f_\alpha\}$ is a subordinate partition
of unity \index{partition of unity}(we do not even have to require that $0\le f_\alpha\le 1$,
only that the sum $\sum f_\alpha = 1$ be locally finite), then the definition
$$D_XY = \sum_\alpha f_\alpha D_X^\alpha Y$$
defines a connection globally on $M$.

\subsection{An Affine Space} \index{affine space}If $D$ and $E$ are connections, then for any
$f\in \mathcal{F}$ we have that $fD +(1-f)E$ is again a connection. As $f$
runs through constants this gives a straight line in the collection of
all connections on $M$. Moreover, $S_XY = D_XY - E_XY$ is
$\mathcal{F}$-linear in $Y$, and so defines a (1,2) tensor field $S$ such that $E = D + S$. Conversely, for any (1,2) tensor field $S,\ D+S$ is a
connection. We interpret this to say:

\vspace{\smallskipamount}

\textsl{The set of all connections on $M$ is an affine space for which the
associated vector space can be identified with the space of all (1,2)
tensor fields on $M$.}

\subsection{The Induced Connection on a Curve} \index{connection!on curve}
If $\gamma$ is a curve in a
manifold $M$ which has a connection $D$, then we can define what we
call the induced connection on the pullback of the tangent bundle to
the curve. This is a means of differentiating vector fields $Y$ along
the curve with respect to the velocity of the curve. Thus, for each
parameter value $t$, $Y(t) \in M_{\gamma(t)}$, and $D_{\gamma'}Y$ will
be another field, like $Y$, along $\gamma$. There are various ways of
formulating the definition, and there is even a general theory of
pulling back connections along maps (see Bishop \& Goldberg,
pp.220--224), \index{Goldberg, S.I.}\index{connection!on map}
but there is a simple method for curves as follows. 
Take a tangent space basis at some point of $\gamma$ and parallel translate
this basis along $\gamma$ to get a parallel basis field
$(E_1,\ldots,E_n)$ for vector fields along $\gamma$. Then we can
express $Y$ uniquely in terms of this basis field, $Y = \sum f^iE_i$,
where the coefficient functions $f^i$ are smooth functions of $t$. We
define $D_{\gamma}Y = \sum {f^i}'E_i$. From the viewpoint of the theory
of pulling back connections it would be more appropriate to write
$D_{d/dt}Y$ instead of $D_{\gamma'}$. In fact, the Leibnitz rule for
this connection on $\gamma$ is
$$D_{\gamma'}fY = \frac{df}{dt}Y + fD_{\gamma'}Y,$$
and we also have the strange result that even if $\gamma'(t) = 0$ it
is possible to have $(D_{\gamma'}Y)(t) \ne 0$. Even a field on a
constant curve (which is just a curve in the tangent space of the
constant value of the curve) can have nonzero covariant derivative
along the curve.

\subsection{Parallelizations} \index{parallelization}
A manifold $M$ is called {\em parallelizable}
if the tangent bundle $TM$ is trivial as a vector bundle: $TM \cong
M\times\R^n$.  For a given trivialization $(\cong)$ the vector
fields $X_1,\ldots,X_n$ which correspond to the standard unit vectors
in $\R^n$ are called the corresponding {\em parallelization} of
$M$. Conversely, a global basis of vector fields gives a
trivialization of $TM$. If we set $\varphi = 0$, we get the {\em
connection of the parallelization}.\index{connection!of parallelation} 
For this connection parallel
translation depends only on the ends of the curve. The parallel fields
are $\sum a^iX_i,\ a^i$ constants.

\vspace{\medskipamount}

An even-dimensional sphere is {\em not} parallelizable.

\vspace{\medskipamount}

Lie groups are parallelized by a basis of the left-invariant vector
fields.

\vspace{\medskipamount}

We shall see that for any manifold $M$ its bundle of bases $BM$ and
various frame bundles $FM$ are parallelizable.
\index{bundle of bases}

\subsection{Torsion} \index{torsion}If $D$ is a connection, then
$$T(X,Y) = D_XY - D_YX - [X,Y]$$
defines an $\mathcal{F}$-linear, skew-symmetric function of pairs of vector
fields, called the {\em torsion} of $D$. Hence $T$ is a $(1,2)$ tensor
field. For a local basis $B$ with dual $\omega$ and connection form
$\varphi$, the torsion $T(X,Y)$ has $B$-components
$$\Omega(X,Y) = d\omega(X,Y) + \varphi\wedge\omega(X,Y).$$
The $\R^n$-valued $2$-form $\Omega$ is called the {\em local torsion
$2$-form} and its defining equation\index{torsion!$2$-form}
$$\Omega = d\omega + \varphi\wedge\omega$$
is called the {\em first structural equation}.

\vspace{\medskipamount}

A connection with $T=0$ is said to be {\em symmetric}.\index{connection!symmetric}

\vspace{\medskipamount}

For any connection $D$, the connection $E = D - \frac{1}{2}T$ is always
symmetric.

\subsection{Curvature} \index{curvature}For vector fields $X,Y$ we define the 
{\em curvature
operator} $R_{XY}:Z\to R_{XY}Z$, mapping a third vector field $Z$ to a
fourth one $R_{XY}Z$ by\index{curvature!operator}
$$R_{XY} = D_{[X,Y]} - D_XD_Y + D_YD_X.$$
(Some authors define this to be $-R_{XY}$.) As a function of three
vector fields $R_{XY}Z$ is $\mathcal{F}$-linear and so defines a $(1,3)$
tensor field. However, the interpretation as a $2$-form whose values are
linear operators on the tangent space is the meaningful viewpoint. For
a local basis we have {\em local curvature $2$-forms}, with values which
are $n\times n$ matrices:
$$\Phi = d\varphi + \varphi\wedge\varphi,$$
which is the {\em second structural equation}. \index{second structural equation}\index{curvature!2-form}
The sign has been
switched, so that the matrix of $R_{XY}$ with respect to the basis $B$
is $-\Phi(X,Y)$. The wedge product is a combination of matrix
multiplication and wedge product of the matrix entries, just as it was
in the first structural equation, so that $(\varphi\wedge\varphi)(X,Y)
= \varphi(X)\varphi(Y) - \varphi(Y)\varphi(X)$.

\vspace{\medskipamount}
\begin{prob}\label{prob6} Calculate the torsion tensor for the connection of
a parallelization, relating it to the brackets of the basis
fields.\end{prob}\index{connection!torsion of parallelization}

\vspace{\smallskipamount}
\begin{prob}\label{prob7} Calculate the law of change of a connection: 
\index{connection!law of change} that
is, if we have a local basis $B$ and its connection form $\varphi$ and
another local basis $B'= BF$ and its connection form $\varphi'$, find
the expression for $\varphi'$ on the overlapping part of the local
basis neighborhoods in terms of $\varphi$ and $F$. In the case of
coordinate local bases the matrix of functions $F$ is a Jacobian
matrix.\end{prob}

\vspace{\smallskipamount}

\begin{prob}\label{prob8} Check that the axioms for a connection are
satisfied for the connection specified by the partition of unity and
local connections, in the proof of existence of connections.\index{partition of unity}

\vspace{\smallskipamount}\end{prob}

\begin{prob}\label{prob9} Verify that the definition of $T(X,Y)$ leads to
the local expression for $\Omega$ given by the first structural
equation; that is, $T$ and $\Omega$ are assumed to be related by $T =
B\Omega$.\end{prob}

\subsection{The Bundle of Bases} We let \index{bundle of bases}
$$BM =\{(p,x_1,\ldots,x_n):p\in M,\ (x_1,\ldots,x_n)\  \hbox{\rm 
is a basis of } M_p\}.$$ 
\index{bundle of bases} This is called the {\em bundle
of bases of M}, and we make it into a manifold of dimension $n+n^2$ as
follows. Locally it will be a product manifold of a neighborhood $U$ of a
local basis $(X_1, \ldots, X_n)$ with the general linear group
$Gl(n,\R)$ consisting of all $n\times n$ nonsingular matrices.
\index{general linear group}
Since the condition of nonsingularity is given by requiring the
continuous function {\em determinant} to be nonzero, $Gl(n,\R)$
can be viewed as an open set in $\R^{n^2}$, so that it gets a
manifold structure from the single coordinate map. Then if $p\in U$ and
$g=(g_i^j)\in Gl(n,\R)$, we make the element $(p,g)$ of the
product correspond to the basis $(p, \sum_j X_j(p)g_1^j,\ldots,\sum_j
X_j(p)g_n^j)$. It is routine to prove that if $M$ has a $C^k$ structure,
then the structure defined on $BM$ by using $C^{k-1}$ local bases is a
$C^{k-1}$ manifold structure on $BM$.

\vspace{\smallskipamount}

The {\em projection map} $\pi:BM\to M$ given by $(p,x_1,\ldots,x_n)\to
p$ is given locally by the product projection, so is a smooth map. The
{\em fiber over p} is $\pi^{-1}(p)$, a submanifold diffeomorphic to
$Gl(n,\R)$.

\vspace{\smallskipamount}

Each local basis $B = (X_1,\ldots,X_n)$ can be thought of as a smooth
map $B:U\to BM$, called a {\em cross-section of BM over U}. The
composition with $\pi, \ \pi\circ B$ is the identity map on $U$.
\index{cross-section}

\subsection{The Right Action of the General Linear Group} 
\index{right action} Each matrix $g\in
Gl(n,\R)$ can be used as a change of basis matrix on every basis
of every tangent space. This simultaneous change of all bases in the
same way is a map $R_g:BM\to BM$, given by $(p, x_1,\ldots,x_n)\to (p,
\sum_j x_jg_1^j,\ldots,\sum_j x_jg_n^j)$. For $b\in BM$ we will also
write $R_gb = bg$. It is called the {\em right action of g on BM}. For
two elements $g,h\in Gl(n,\R)$ we clearly have $(bg)h = b(gh)$,
that is, $R_h\circ R_g = R_{gh}$.

\subsection{The Universal Dual $1$-Forms} \index{cobasis!universal}
\index{universal cobasis}There is a column of $1$-forms on $BM$,
existing purely due to the nature of $BM$ itself, an embodiment of the
idea of a dual basis of the basis of a vector space. These $1$-forms are
called the {\em universal dual $1$-forms}, are denoted $\omega =
(\omega^i)$, and are defined by the equation:
$$\pi_*(v) = \sum_i\omega^i(v)x_i,$$
where $v$ is a tangent vector to $BM$ at the point $(p,
x_1,\ldots,x_n)=b$. We can also use multiplication of the row of basis
vectors by the columns of values of the $1$-forms to write the definition
as $\pi_*(v) = b\omega(v)$. Thus, the projection of a tangent vector
$v$ to $BM$ is referred to the basis at which $v$ lives, and the
coefficients are the values of these canonical $1$-forms on $v$. (The
canonical $1$-forms are usually called the {\em solder forms} of $BM$.)

\vspace{\smallskipamount}

It is a simple consequence of the definition of $\omega$ that if $B$
is a local basis on $M$, then the pullback $B^*\omega$ is the local
dual basis of $1$-forms. This justifies the name for $\omega$.

\vspace{\smallskipamount}
We clearly have that $\pi\circ R_g = \pi$, so that 
$$\pi_*\circ {R_g}_*(v) = bg\omega({R_g}_*(v)) = b\omega(v),$$
where $v$ is a tangent at $b\in BM$. Rubbing out the ``$b$'' and
``$(v)$'' on both sides of the equation leaves us with an equation for
the action of $R_g$ on $\omega$:
$$g\cdot {R_g}^*\omega = \omega, \qquad {\rm or}\qquad  
{R_g}^*\omega = g^{-1}\omega.$$

\subsection{The Vertical Subbundle} \index{subbundle!vertical}
The tangent vectors to the fibers of
$BM$, that is, the tangent vectors in the kernel of $\pi_*$, form a
subbundle of TBM of rank $n^2$. This is called the {\em vertical}
subbundle of $TBM$, and is denoted by $V$.

\subsection{Connections} \index{connection}
A {\em connection on BM} is a specification of a
complementary subbundle $H$ \index{subbundle!horizontal}
to $V$ which is smooth and invariant under
all right action maps $R_g$. The idea is that moving in the direction
of $H$ on $BM$ represents a motion of a basis along a curve in $M$
which will be defined to be parallel translation of that basis along
the curve. \index{parallel translation}
We can then parallel translate any tangent vector along the
curve in $M$ by requiring that its coefficients with respect to the
parallel basis field be constant. The invariance of $H$ under the right
action is needed to make parallel translation of vectors be independent
of the choice of initial basis.

\subsection{Horizontal Lifts} \index{horizontal lift}
Since $H$ is complementary to the kernel of
$\pi_*$ at each point of $BM$, the restriction of $\pi_*$ to $H(b)$ is
a vector space isomorphism to $M_{\pi b}$. Hence we can apply the
inverse to vectors and vector fields on $M$ to obtain {\em horizontal
lifts} of those vectors and vector fields. We usually lift single
vectors to single horizontal vectors, but for a smooth vector field $X$
on $M$ we take all the horizontal lifts of all the values of $X$, thus
obtaining a {\em smooth} vector field $\bar X$. These vector field
lifts are compatible with $\pi$ and all $R_g$, so that if $Y$ is
another vector field on $M$, then $[\bar X,\bar Y]$ is right invariant
and can be projected to $[X,Y]$. However, we have not assumed that $H$
is involutive, so that $[\bar X,\bar Y]$ is not generally the
horizontal lift of $[X,Y]$.

\vspace{\smallskipamount}

The construction of the pull-back bundle and its horizontal subbundle
is the bundle of bases version of the induced connection on the curve
$\gamma$. More generally, the induced connection on a map has a bundle
of bases version defined in just the same way.\index{connection!pullback to curve}\index{connection!induced}

\vspace{\medskipamount}
We can also get horizontal lifts of smooth curves in $M$. This is
equivalent to getting the parallel translates of bases along the curve.
If the curve is the integral curve of a vector field $X$, then a
horizontal lift of the curve is just an integral curve of $\bar X$.
However, curves can have points where the velocity is $0$, making it
impossible to realize the curve even locally as the integral curve of a
smooth vector field. Thus, we need to generalize the bundle
construction a little to obtain horizontal lifts of arbitrary smooth
curves $\gamma:[a,b]\to M$. We define the {\em pull-back bundle}
$\gamma^*BM$ to be the collection of all bases of all tangent spaces at
points $\gamma(t)$, and give it a manifold with boundary structure,
diffeomorphic to the product $[a,b]\times Gl(n,\R)$, just as we
did for $BM$, along with a smooth map into $BM$. The horizontal
subbundle $H$ can also be pulled back to a horizontal subbundle of rank
$1$. Then we have a vector field $d/dt$ on $[a,b]$ whose horizontal lift
has integral curves representing the desired horizontal lift of
$\gamma$. This structure on $\gamma^*BM$ corresponds to the connection
on $\gamma$ given above.

\vspace{\smallskipamount}
When we relate all of this to the other version of connections in terms
of covariant derivative operators, we see that the differential
equations problem for getting parallel translations has turned into the
familiar problem of getting integral curves of a vector field on a
different space.

\subsection{The Fundamental Vector Fields} \index{fundamental vector fields}
Corresponding to the
left-invariant vector fields on $Gl(n,\R)$ we have some
canonically defined vertical vector fields on $BM$; each fiber is a
copy of $Gl(n,\R)$ and these canonical fields are carried over as
copies of the left invariant fields. Generally on a Lie group the
left-invariant vector fields are identified with the tangent space at
the identity: in one direction we simply evaluate the vector field at
the identity; in the other direction, if we are given a vector at the
identity as the velocity $\gamma'(0)$ of a curve, then we can get the
value at any other point $g$ of the group as the velocity of the curve
$g\cdot\gamma$ at time $0$. Note that whereas $g$ multiplies the curve on
the left, which is what makes the vector field {\em left invariant},
what we see nearby $g$ is the result of multiplying $g$ on the {\em
right} by $\gamma$. It is this process of multiplying on the right by a
curve through the identity that we can imitate in the case of $BM$,
since we have the right action of $Gl(n,\R)$ on $BM$. 
It is natural to view the tangent space of $Gl(n,\R)$ at the identity
matrix $I$ as being the set of all $n\times n$ matrices, which we
denote by $gl(n,\R)$. If we let $E_j^i$ be the matrix with $1$ in
the $ij$ position and $0$'s elsewhere, we get a standard basis of the Lie
algebra. For a curve with velocity $E_j^i$ we can take simply
$\gamma(t) = I + tE_j^i$.
The {\em fundamental vector fields} on $BM$ are the vector fields
$E_j^i$ 
defined by
$$E_j^i(b) = \gamma'(0),\ {where} \ \gamma(t) = b\cdot (I+tE_j^i).$$
We sometimes also call the constant linear combinations of these basis
fields {\em fundamental}.

\subsection{The Connection Forms}\index{connection!$1$-forms} If we are given a connection $H$ on $BM$,
we define a matrix of $1$-forms $\varphi = (\varphi_j^i)$ to be the forms
which are dual to the vector fields $E_j^i$ on the vertical subbundle
$V$ of $TBM$ and are $0$ on the connection subbundle $H$. This means that
if $\gamma(t)= b\cdot(I+ta)$, where $a$ is an $n\times n$ matrix, then
$\varphi(\gamma'(0))= a$. 

Clearly the connection forms completely determine $H$ as the subbundle
they annihilate. Thus, in order to give a connection it is adequate to
specify the connection form $\varphi$. In order to say what matrices of
1-forms on $BM$ determine a connection, besides the property that the
restriction to the vertical $V$ gives forms dual to the fundamental
fields $E_i^j$, we have to spell out the condition that $H$ is right
invariant in terms of $\varphi$. The name for this condition is {\em
equivariance}, and things have been arranged so that it is expressed
easily in terms of the differential action of $R_g$ and matrix
operations: 
$${R_g}^*\varphi = g^{-1}\varphi g\qquad\hbox{\rm for all}\ g\in
Gl(n,\R).$$

\subsection{The Basic Vector Fields} \index{basic vector fields}
The universal dual cobasis is nonzero on
any nonvertical vector, so that if we restrict it to the horizontal
subspace of a connection it gives an isomorphism: $\omega : H(b)\to {\bf
R}^n$. If we invert this map and vary $b$, we get the {\em basic vector
fields} of the connection. In particular, using the standard basis of
$\R^n$, we get the basic vector fields $E_i$; they are the
horizontal vector fields such that $\omega^i(E_j) = \delta_j^i$.

\subsection{The Parallelizability of BM} \index{parallelizability!bundle of basis}Since a connection always exists, we
now know that $BM$ is parallelizable; specifically, $\{E_i^j, E_i\}$ is
a parallelization.
\begin{thm}[Existence of Connections (again)] There exists a connection
on $BM$.\index{connection!existence}
\end{thm}

Locally we have that $BM$ is defined to be a product manifold. We can
define a connection locally by taking the horizontal subspace to be the
summand of the tangent bundle given by the product structure,
complementary to the tangent spaces of $Gl(n,\R)$. In turn this
will give us some local connection forms which satisfy the equivariance
condition. \index{equivariant form}Then we combine these local connection forms by using a
partition of unity for the covering of $M$ by the projection of their
domains. (This is no different than the previous proof of existence of
a connection.)

\subsection{Relation Between the Two Definitions of Connection} If we are
given a connection on $BM$ and a local basis $B:U\to BM$, then we can
pull back the connection form on $BM$ by $B$ to get a matrix of $1$-forms
on $U$. This pullback $B^*\varphi$ will then be the connection form of
a connection on $U$ in the sense of covariant derivatives. It requires
some routine checking to see that these local connection forms all fit
together, as $B$ varies, to make a global connection on $M$ in the
sense of covariant derivatives.\index{connection}\index{covariant derivative}

Conversely, if we are given a connection on $M$ in the sense of
covariant derivatives, then we can define $\varphi$ on the image of a
local basis $B$ by identifying it with the local connection form under
the diffeomorphism. Then the extension to the rest of $BM$ above the
domain of $B$ is forced on us by the equivariance and the fact that
$\varphi$ is already specified on the vertical spaces.
The geometric meaning of the relation between the local connection form
and the form on $BM$ is clear: on the image of $B$ the form $\varphi$
measures the failure of $B$ to be horizontal; by the differential
equation for parallel translation the local form measures the failure
to be parallel. But ``parallel on $M$'' and ``horizontal on $BM$'' are
synonymous.
A form on $BM$ is {\em horizontal} \index{form!horizontal}
if it gives $0$ whenever any vertical
vector is taken as one of its arguments. This means that it can be
expressed in terms of the $\omega^i$ with real-valued functions as
coefficients. A form $\eta$  on $BM$ with values in $\R^n$ is {\em
equivariant} \index{form!equvariant}
if ${R_g}^*\eta = g^{-1}\eta$. A form $\psi$ on $BM$ with
values in $gl(n,\R)$ is {\em equivariant} if ${R_g}^*\psi = g^{-1}\psi
g$. The significance of these definitions is that the horizontal
equivariant forms on $BM$ correspond to tensorial forms on $M$: $\eta$
corresponds to a tangent-vector valued form on $M$, $\psi$ corresponds
to a form on $M$ whose values are linear transformations of the tangent
space. The rules for making these correspondences are rather obvious:
evaluate $\eta$ or $\psi$ on lifts to a basis $b$ of the vectors on $M$
and use the result as coefficients with respect to that basis for the
result we desire on $M$. The horizontal condition makes this
independent of the choice of lifts; the equivariance makes it
independent of the choice of $b$.

For example, the form $\omega$ corresponds to the $1$-form on $M$ with
tangent-vector values which assigns a vector $x$ to itself.
\begin{thm}[The Structural Equations] If $\varphi$ is a connection form
on $BM$, then there is a horizontal equivariant $\R^n$-valued
2-form $\Omega$ and a horizontal equivariant $gl(n,\R)$-valued
2-form $\Phi$ such that\index{first structural equation}
$$d\omega = -\varphi\wedge\omega + \Omega,$$
and\index{second structural equation}
$$d\varphi = -\varphi\wedge\varphi + \Phi.$$
\end{thm}\index{curvature!2-form}

The structural equations have already been proved in the form of
pullbacks of the terms of the equations by a local basis. This shows
how the forms $\Omega$ and $\Phi$ are given on the image of a local
basis. The fact that these local forms yield the tensors $T$ and $R$
\index{torsion}\index{curvature}
which live independently of the local basis can be interpreted as
establishing the equivariance properties of $\Omega$ and $\Phi$ since they are horizontal.
It is also not difficult to prove the structural equations directly
from the specified equivariance of $\varphi$. 
If we restrict the structural equations to vertical vectors, or one
vertical and one horizontal vector, we get information that has nothing
to do with connections. The first one tells how $Gl(n,{\bf, R})$
operates on $\R^n$. The second one is more interesting: it is the
equations of Maurer-Cartan for $Gl(n,\R)$, which are essentially
its Lie algebra structure in its dual packaging.\index{Maurer-Cartan equations}

\subsection{The Dual Formulation} The dual to taking exterior derivatives of
$1$-forms is essentially the operation of bracketing vector fields. If we
bracket two fundamental fields or a fundamental and a basic field, we
obtain nothing new, only a repeat of the Lie algebra structure and its
action on Euclidean space. The brackets that actually convey
information about the connection are the brackets of basic vector
fields. By using the exterior derivative formula $d\alpha(X,Y) =
X\alpha(Y) - Y\alpha(X) - \alpha([X,Y])$, we see that the first
structural equation tells what the horizontal part of a basic bracket
is and the second tells us what the vertical part is. Most of the terms
are $0$:
$$d\omega(E_i,E_j) = E_i\omega(E_j) - E_j\omega(E_i) -
\omega([E_i,E_j])$$
$$\qquad = -\omega([E_i,E_j])$$
$$\qquad = \Omega(E_i,E_j).$$
Similarly,
$$-\varphi([E_i,E_j]) = \Phi(E_i, E_j).$$
We can immediately get some important geometrical information about a
connection. The condition for the subbundle to be integrable is that
the brackets of its vector fields again be within the subbundle. The
basic fields are a local basis for $H$, so the condition for $H$ to be
integrable is just that curvature be $0$. This means that locally there
are horizontal submanifolds, which are local bases with a very special
property: whenever we parallel translate around a small loop the result
is the identity transformation; or, parallel translation is locally
independent of path. The fact that setting curvature to $0$ gives this
local independence of path is not very obvious from the covariant
derivative viewpoint of connections.
If we go one step further and impose both curvature and torsion equal $0$,
then the result is also easy to interpret from basic manifold theory
applied to the fields on $BM$. Indeed, when a set of independent vector
fields has all brackets vanishing, there are coordinates so that these
fields are coordinate vector fields. When these are the $E_i$ of a
connection on $BM$ the coordinates correspond to coordinates on the
leaves of $H$, which get transferred down to coordinates on $M$ such
that the coordinate vector fields are parallel along every curve. The
geometry is exactly the same as the usual geometry of Euclidean space,
at least locally.

\subsection{Geodesics} \index{geodesic!of connection}
A {\em geodesic} of a connection is a curve for which
the velocity field is parallel. Hence, a geodesic is also called an
{\em autoparallel curve}. \index{autoparallel curve}
Notice that the parametrization of the curve
is significant, since a reparametrization of a curve can stretch or
shrink the velocity by different factors at different points, which
clearly destroys parallelism. (There is a trivial noncase: the constant
curves are formally geodesics. But then a reparametrization does
nothing.) 

If we are given a point $p$ and a vector $x$ at $p$, we can take a
basis $b=(x_1,\ldots,x_n)$ so that $x = x_1$. Then the integral curve
of $E_1$ starting at $b$ is a horizontal curve, so represents a
parallel field of bases along its projection $\gamma$ to $M$. Moreover,
the velocity field of $\gamma$ is the projection of $E_1$ at the points
of the integral curve; but the projection of $E_1(b)$ always gives the
first entry of $b$. We conclude that $\gamma$ has parallel velocity
field. The steps of this argument can be reversed, so that the
geodesics of $M$ are exactly the projections of integral curves of
$E_1$. Any other basic field could be used instead of $E_1$.

Geodesics do not have to go on forever, since the field $E_i$ may not
be complete. If $E_i$ is complete, so that all geodesics are extendible
to all of $\R$ as geodesics, then we say that $M$ is {\em geodesically
complete}.\index{complete!geodesically}

\begin{thm}[Local Existence and Uniqueness of Geodesics] 
\index{geodesic!existence} For every
$p\in M$ and $x\in M_p$ there is a geodesic $\gamma$ such that
$\gamma'(0)= x$. Two such geodesics coincide in a neighborhood about
$0\in \R$. There is a maximal such geodesic, defined on an open
interval, so that every other is a restriction of this maximal one.
\end{thm}

This theorem is an immediate consequence of the same sort of
statements about vector fields.

\subsection{The Interpretation of Torsion and Curvature in Terms of
Geodesics} Recall the geometric interpretation of brackets: if we move
successively along the integral curves of $X, Y, -X, -Y$ by equal
parameter amounts $t$, we get an endpoint curve which returns to the
origin up to order $t^2$, but gives the bracket in its second order term.
We apply this to the vector fields $E_1, E_2$ on $BM$. The meaning of
the construction of the ``small parallelogram'' on $BM$ is that we
follow some geodesics below on $M$, carrying along a second vector by
parallel translation to tell us what geodesic we should continue on
when we have reached the prescribed parameter distance $t$. If we were
to do this in Euclidean space, the parallelogram would always close up,
but here the amount it fails to close up is of order $t^2$ and is
measured by the horizontal part of the tangent to the endpoint curve in
$BM$ above. But we have seen that the horizontal part of that bracket
is given by $-\Omega(E_1,E_2)$ relative to the chosen basis. When we
eliminate the dependence on the basis we conclude the following:
\begin{thm}[The Gap of a Geodesic Parallelogram] A geodesic
parallelogram generated by vectors $x,y$ with parameter side-lengths
$t$ has an endgap equal to $-T(x,y)t^2$ up to terms of order $t^3$.
\end{thm}\index{geodesic!parallelogram}

The other part of the gap of the bracket parallelogram on $BM$ is the
vertical part. What that represents geometrically is that failure of
parallel translation around the geodesic parallelogram below to bring
us back to the identity. That failure is what curvature measures, up to
terms cubic in $t$. We can't quite make sense of this as it stands,
because the parallel translation in question is not quite around a
loop; however, if we close off the gap left due to torsion in any
non-roundabout way, then the discrepancies among the various ways of
closing up, as parallel translation is affected, are of higher order in
$t$. That is the interpretation we place on the following theorem.

\begin{thm}[The Holonomy of a Geodesic Parallelogram] Parallel
translation around a geodesic parallelogram generated by vectors $x,y$
with parameter side-lengths $t$ has second order approximation
$I+R_{xy}\cdot t^2$.
\end{thm}\index{holonomy}

It seems to me that a conventional choice of sign of the curvature
operators to make the ``+'' in the above theorem turn into a ``$-$'' is
in poor taste. The only other guides for which sign should be chosen
seem to be merely historic.
The word ``holonomy'' is used in connection theory to describe the
failure of parallel translation to be trivial around loops. By chaining
one loop after another we get the product of their holonomy
transformations, so that it makes sense to talk about a holonomy group
as a measure of how much the connection structure fails to be like
Euclidean geometry.\index{holonomy!group}

\vspace{\smallskipamount}

\begin{prob}\label{prob10} Decomposition of a Connection into Geodesics and
Torsion. Prove that if two connections have the same geodesics and
torsion they are the same. Furthermore, the geodesics and torsion can
be specified independently.\end{prob}
\vspace{\smallskipamount}

In regard to the meaning of the last statement, we intend that the
torsion can be any tangent-vector valued $2$-form; the specification of
what families of curves on a manifold can be the geodesics of some
connection has been spelled out in an article by W. Ambrose,
R.S.~Palais, and I.M.~Singer, {\em Sprays}, Anais da Academia Brasileira
de Ciencias, vol. 32, 1960. But for the problem you are required only
to show that the geodesics of a connection to be specified can be taken
to be the same as some given connection.\index{sprays}\index{Ambrose, W.}
\index{Palais, R.S.}\index{Singer, I.M.}

\vspace{\smallskipamount}
\begin{prob}\label{prob11} Holonomy of a Loop. Prove the following more
general and precise version of the Theorem on Holonomy of Geodesic
Parallelograms.
\textsl{Let $h:[0,1]\times[0,\tau]\to M$ be a smooth homotopy of the constant
loop $p=h(0,v)$ to a loop $\gamma(v)=h(1,v)$ with fixed ends, so that
$h(u,0)= h(u,\tau)=p$. Fix a basis $b=(p,x_1,\ldots,x_n)$ and let
$g(u)\in Gl(n,\R)$ be the matrix which gives parallel translation
$bg(u)$ of $b$ around the loop $v\to h(u,v)$. Let $X=\frac{\partial h}{\partial u}, Y= \frac{\partial h}{\partial v}$ and let $\bar
h:[0,1]\times[0,\tau]\to BM$ be the lift of $h$ given by lifting each
loop $v\to h(u,v)$ horizontally with initial point $b$. In particular,
$\bar h(u,\tau) = bg(u)$. Prove that}
$$\int_0^1 g(u)^{-1}g'(u)\,du = \int_0^1\,\int_0^\tau \bar h(u,v)^{-1}
\circ R_{XY}\circ\bar h(u,v)\,dv\,du.$$
The meaning of the integrand on the right is as follows:
We interpret a basis $b' = (q,y_1,\ldots,y_n)$ to be the linear
isomorphism $\R^n\to M_q$ given by $(a^1,\ldots,a^n)\to\sum_i
a^iy_i$. Thus, for each $u,v$, we have a linear map
$$\bar h(u,v)^{-1}\circ R_{X(u,v)\,Y(u,v)}\circ\bar h(u,v):
\R^n\to M_{h(u,v)}\to M_{h(u,v)}\to\R^n.$$
As a matrix this can be integrated entry-by-entry. Hint: Pullback the
second structural equation via $\bar h$ and apply Stokes' theorem on
the rectangle.\end{prob}

\vspace{\smallskipamount}
\begin{prob}\label{prob12} The General Curvature Zero Case. 
Suppose that $M$
is connected and that we have a connection $H$ on $BM$ for which $\Phi
= 0$, that is, $H$ is completely integrable. Let $\tilde M$ be a leaf
of $H$, that is, a maximal connected integrable submanifold. Show that
the restriction of $\pi$ to $\tilde M$ is a covering map. Moreover, if
we choose a base point $b\in \tilde M$, then we can get a homomorphism of
the fundamental group $\pi_1(M)\to Gl(n,\R)$ as follows: for a
loop based at $\pi(b)$ we lift the loop into $\tilde M$, necessarily
horizontally, getting a curve in $\tilde M$ from $b$ to $bg$. Then $g$
depends only on the homotopy class of the loop.\end{prob}\index{holomony}

\vspace{\smallskipamount}

A connection with curvature zero is called {\em flat} and the
homomorphism of Problem \ref{prob12} is called the 
{\em holonomy map} of that
flat connection.

\subsection{Development of Curves into the Tangent Space} 
\index{development of curve} Let
$\gamma:[0,1]\to M$ and let $D$ be a connection on $M$. We let
$\bar\gamma:[0,1]\to BM$ be a parallel basis field along $\gamma$
starting at $b=\bar\gamma(0)$ and express the velocity of $\gamma$ in
terms of this parallel basis, getting a curve of velocity components
$\beta(t)=\bar\gamma(t)^{-1}\gamma'(t)\in\R^n$. Then we let
$\sigma(t)= b\int_0^t\beta(u)\,du$. Thus, the velocity field of
$\sigma$ in the space $M_{\gamma(0)}$ bears the same relation to
Euclidean parallel translates of $b$ as does the velocity field of
$\gamma$ to the parallel translates of $b$ given by the connection. We
call $\sigma$ the {\em development of $\gamma$ into} $M_{\gamma(0)}$.

\vspace{\smallskipamount}
\begin{prob}\label{prob13} Show that the development is independent of the
choice of initial basis $b$.\end{prob}

\vspace{\smallskipamount}

\begin{prob}\label{prob14} Show that the lift $\bar\gamma$ of $\gamma$ in
the definition of the development is the integral curve of the
time-dependent vector field $\sum \beta^i(t)E_i$ on $BM$ starting at
$b$.\end{prob}

\subsection{Reverse Developments and Completeness} Starting with a curve
$\sigma$ in $M_p$ such that $\sigma(0) = 0 $, we choose a basis $b$
and let $\beta=b^{-1}\sigma'$.  By Problem \ref{prob14}, 
we can then get a curve
$\gamma$ in $M$ whose development is $\sigma$, at least locally. We
call $\gamma$ the {\em reverse development} of $\sigma$. Since the
vector field $\sum\beta^i(t)E_i$ need not be complete, it is not
generally true that every curve in $M_p$ can be reversely developed
over its entire domain.

\vspace{\smallskipamount}

We say that $(M,D)$ is {\em development-complete} at $p$ if every
smooth curve $\sigma$ in $M_p$ such that $\sigma(0) = 0$ has a reverse 
development over its entire domain.\index{completeness!development-}

\vspace{\smallskipamount}

\begin{prob}\label{prob15} Suppose that $M$ is connected. Show that 
the condition that $(M,D)$ be development-complete at $p$ is
independent of the choice of $p$.\end{prob}

\vspace{\smallskipamount}

\begin{prob}\label{prob16} Show that the development of a geodesic is a ray
with linear parametrization, and hence, if $M$ is development complete,
then $M$ is geodesically complete.\end{prob}\index{complete!geodesically}

\vspace{\smallskipamount}
\index{connection!of parallelization}
\begin{prob}\label{prob17} For the connection of a parallelization
$(X_1,\ldots , X_n)$ show that the geodesics are integral curves of
constant linear combinations $\sum a^iX_i$.\end{prob}
\index{parallelization!connection of}\index{parallelization!geodesics of connection}

\vspace{\smallskipamount}

\begin{prob}\label{prob18} For the connection of a parallelization 
$B = (X_1,\ldots,X_n)$ show that the reverse development of $\sigma$
for which $\beta = b^{-1}\sigma'$ is an integral curve of
$B\beta$.\end{prob}

\vspace{\smallskipamount}

\begin{prob}\label{prob19} Show that if $f > 0$ grows fast enough along a
curve $\gamma$ in $\R^2$, then the development of $\gamma$ with
respect to the parallelization $(f\frac{\partial}{\partial x},
f\frac{\partial}{\partial y})$ is bounded. Hence there is no reverse
development having unbounded continuation.\end{prob}

\vspace{\smallskipamount}

\begin{prob}\label{prob20} Show that the geodesics of the parallelization of
Problem \ref{prob19} are straight lines except for parametrization, and on
regions where $f=1$ even the parametrization is standard. Then by
taking $\gamma$ to be an unbounded curve with a neighborhood $U$ such
that any straight line meets $U$ at most in a bounded set, and $f$ a
function which is $1$ outside $U$ and grows rapidly along $\gamma$, it is
possible to get a connection (of the parallelization) which is
geodesically complete but not development-complete.\end{prob}
\index{complete!geodesically}

\begin{rem} Parallel translation along a curve $\gamma$ is
independent of parametrization.
\end{rem}
\begin{rem} The only reparametrizations of a nonconstant geodesic which are
again geodesics are the affine reparametrizations: $\sigma(s) =
\gamma(as+b).$ 
\end{rem}

A curve which can be reparametrized to become a geodesic is called a
{\em pregeodesic.}\index{pregeodesic}\index{geodesic!pre-}

\vspace{\smallskipamount}

\begin{prob}\label{prob21} Show that a regular curve $\gamma$ is a
pregeodesic if and only if $D_{\gamma'}\gamma' = f \gamma'$ for
some real-valued function $f$ of the parameter.\end{prob}

\subsection{The Exponential Map of a Connection} 
\index{exponential map} For $p \in M,\ x \in M_p$
let $\gamma_{p,x}$ be the geodesic starting at $p$ with initial
velocity $\gamma_{p,x}'(0) = x$. The {\em exponential map at p} is
defined by 
$$exp_p:U\to M\qquad {\rm by}\qquad exp_px = \gamma_{p,x}(1),$$
where $U$ is the subset of $x\in M_p$ for which $\gamma_{p,x}(1)$ is
defined.  
\begin{prop} The domain $U$ of $exp_p$ is an open
star-shaped subset of $M_p$. Exponential maps are smooth.
\end{prop}

It is called the exponential map because it generalizes the matrix
exponential map, and also the exponential map of a Lie group. These
are obtained when the connection is taken to be the connection of the
parallelization by a basis of the Lie algebra (the left-invariant
vector fields or the right invariant vector fields; both give the same
geodesics through the identity, namely, the one-parameter subgroups).
More specifically, the multiplicative group of the complex numbers is
a two-dimensional real Lie group for which the ordinary complex
exponential map coincides with the one given by the invariant (under
multiplication) connection. The geodesics are concentric circles, open
rays, and loxodromes (exponential spirals).

\subsection{Normal Coordinates} \index{normal coordinates}
Since $M_p$ is a vector space, the tangent
space $(M_p)_0$ is canonically identified with $M_p$ itself. Using
this identification, it is easily seen that the tangent map of $exp_p$
at $0$ may be considered to be the identity map. In particular, by the
inverse function theorem, there is a neighborhood $V$ of $p$ on which
the inverse of $exp_p$ is a diffeomorphism. Referring $M_p$ to a basis
$b$ gives us an isomorphism $b^{-1}$ to $\R^n$, and the
composition gives a {\em normal coordinate map} at $p$:
$$b^{-1}\circ {exp_p}^{-1}: V \to \R^n.$$

For normal coordinates it is clear that the coordinate rays starting
at the origin of $\R^n$ correspond to the geodesic rays starting
at $p$.
\begin{prob}\label{prob22} Suppose that torsion is $0$ and that
$(x^i)$ are normal coordinates at $p$. For any $x \in M_p$ show that
$D_x\frac{\partial}{\partial x^i} = 0$, and hence the operation of
covariant differentiation of vector fields with respect to vectors at
$p$ reduces to operating on the components of the vector fields by the
vectors at $p$.\end{prob}

\subsection{Parallel Translation and Covariant Derivatives of Other Tensors}
\index{parallel translation!tensors fields}\index{covariant derivative!tensors}
We have so far only been concerned with parallel translation and
covariant derivatives of vectors and vector fields. For tensors of
other types we simply reduce to the vector field case: a tensor field
is {\em parallel along a curve} $\gamma$ if the components of the
tensor field are constant with respect to a choice of parallel basis
field along $\gamma$. We calculate $D_xA$, where $A$ is a tensor field
and $x \in TM$ by taking a curve $\gamma$ with $\gamma (0) = x$,
referring $A$ to a parallel basis along $\gamma$, and differentiating
components at $t=0$. These definitions are independent of the choice
of basis.

\setcounter{section}{3}

\section{The Riemannian Connection}

\vspace{\medskipamount}\index{connection!metric}\index{connection!Riemannian}
\index{connection!Levi-Civita}
\subsection{Metric Connections} We now return to the study of
semi-Riemannian metrics, and in particular, Riemannian metrics. If $g$
is such a metric, then we say that a connection is a {\em metric
connection} if parallel translation along any curve preserves inner
products with respect to $g$. It is easy to see that there are several
equivalent ways of expressing that same condition:

\vspace{\medskipamount}

Equivalent to a connection being metric are:

\begin{enumerate}
\item The parallel translation of a frame is always a frame.
\item The tensor field $g$ is parallel along every curve.
\item $D_xg = 0$ for all tangent vectors $x$.
\item $x(g(Y,Z)) = g(D_xY, Z) + g(Y, D_xZ)$ for all tangent vectors $x$
and all vector fields $Y$ and $Z$.
\item Let $FM$ be the {\em frame bundle} of $g$, consisting of all frames
at all points of $M$. It is easily shown that $FM$ is a submanifold of
$BM$ of dimension $n(n+1)/2$. The condition equivalent to a connection
$H$ being metric is that at points of $FM$ the horizontal subspaces
are contained in $TFM$.\index{bundle of frames}
\end{enumerate}

\subsection{Orthogonal Groups} \index{orthogonal group}
\index{frame}\index{frame!bundle}The frame bundle is invariant under the
action of the orthogonal group with the corresponding index $\nu$.
This is the group of linear transformations which leaves invariant the
standard bilinear form on $\R^n$ of that index:
$$g_\nu(x,y) = \sum_{i=1}^{n-\nu}x^iy^i - \sum_{i=n-\nu+1}^nx^iy^i.$$
Thus, $A \in O(n,\nu)$ if and only if $g_\nu(Ax,Ay) = g_\nu(x,y)$ for
all $x, y \in \R^n$. In case $\nu = 0$ this reduces to the
familiar condition for orthogonality: $A^T\cdot A = I$. For other
indices the matrix transpose should be replaced by the
adjoint with respect to $g_\nu$, which we will denote by $A'$. The
corresponding Lie algebra consists of the matrices which are
skew-adjoint:\index{Lie algebra!orthogonal group}
$$so(n,\nu) = \{A: A' = -A\}.$$
No matter what $\nu$ is, the dimension of $O(n,\nu)$ is $n(n-1)/2$,
and since $FM$ is locally a product of open sets (the domains of local
frames!) in $M$ times $O(n,\nu)$, $FM$ has dimension $n(n+1)/2$.

\subsection{Existence of Metric Connections} In general, the definition and
existence of a connection on a principal bundle (this means that the
fiber is a Lie group acted on the right by a model fiber) can be
carried out by imitating the case of $BM$. 
\index{bundle of frames}However, for $FM$ we can
obtain a connection by restricting a connection on $BM$ and then
retaining only the skew-adjoint part. Thus, if $\varphi$ is a
connection form on $BM$, then for $x \in TFM$ we let
$$\varphi _a(x) = {1 \over 2 } (\varphi (x) - \varphi (x)').$$
For the Riemannian case this means that we decompose the matrix
$\varphi(x)$ into its symmetric and skew-symmetric parts and discard
the symmetric part.  Since the decomposition into these parts is
invariant under the action of the orthogonal group by conjugation
(similarity transform), it follows that the form $\varphi_a$ on $FM$
satisfies the equivariance condition:
$${R_g}^*\varphi_a = g^{-1}\varphi_a g,$$
for all $g \in O(n,\nu)$. On the vertical subspaces of $TFM$ the
values of $\varphi$ were already skew-adjoint, so that $\varphi_a$ is
an isomorphism of each vertical space onto $so(n,\nu)$. The number of
independent entries of $\varphi_a$ is $n(n-1)/2$. Consequently, the
annihilated subspaces are a complement to the vertical ones, and by
the equivariance condition, they are invariant under ${R_g}_*$. That
is, we have a connection on $FM$.

\vspace{\medskipamount}

What this means in terms of covariant derivatives on $M$ is that we
can start with any connection $D$; then for any frame field $(E_i)$
with respect to the metric $g$ we have $D_xE_i =\sum_j\varphi_i^j(x)E_j$,
defining the local connection forms
$\varphi_i^j$. Then the local connection forms of a metric connection
are obtained by taking the skew-adjoint part of $\varphi$.

\subsection{The Levi-Civita Connection}

\begin{thm}[The Fundamental Theorem of Riemannian Geometry]
For a semi-Riemannian metric $g$ there exists a
unique metric connection with torsion 0.
\end{thm}
\index{Fundamental Theorem of Riemannian Geometry}
In fact, the metric connections are parametrized bijectively by their
torsions. To see this it is necessary to make the correspondence
between the torsion form $\Omega$ and the form $\tau = \varphi -
\varphi_0$, where $\varphi$ is an arbitrary metric connection form and
$\varphi_0$ will be the one with torsion $0$. According to the first
structural equation (which pulls back to $FM$ unchanged in appearance)
we would have
$$d\omega = -\varphi\wedge\omega + \Omega =
-\varphi_0\wedge\omega,$$
or
$$\Omega = \tau\wedge\omega.$$
It is clear that $\tau$ determines $\Omega$, so the problem reduces to
showing that $\Omega$ determines $\tau$. The condition that the values
of the connection forms, and hence, of $\tau$, are skew-adjoint with
respect to the standard bilinear form $g_\nu$ must be used.

\vspace{\medskipamount}

The trick used to determine $\tau$ in terms of $\Omega$ is to
alternately apply the skew-adjointness:
$g_\nu(\tau(x)\omega(y),\omega(z)) =
-g_\nu(\omega(y),\tau(x)\omega(z))$
and first-structural equation relation: $\Omega(x,y) =
\tau(x)\omega(y) - \tau(y)\omega(x)$ three times:
$$g_\nu(\tau(x)\omega(y),\omega(z)) =
-g_\nu(\omega(y),\tau(x)\omega(z))$$
$$\qquad = -g_\nu(\omega(y), \Omega(x,z)) - g_\nu(\omega(y),
\tau(z)\omega(x))$$
$$\qquad = -g_\nu(\omega(y), \Omega(x,z)) + g_\nu(\tau(z)\omega(y),
\omega(x))$$
$$\qquad = -g_\nu(\omega(y), \Omega(x,z)) + g_\nu(\Omega(z,y),
\omega(x)) + g_\nu(\tau(y)\omega(z), \omega(x))$$
$$\qquad = -g_\nu(\omega(y), \Omega(x,z)) + g_\nu(\Omega(z,y),
\omega(x)) - g_\nu(\omega(z), \tau(y)\omega(x))$$
$$\qquad = -g_\nu(\omega(y), \Omega(x,z)) + g_\nu(\Omega(z,y),
\omega(x)) - g_\nu(\omega(z), \Omega(y,x)) - g_\nu(\omega(z),
\tau(x)\omega(y)).$$ 
By symmetry of $g_\nu$, the expression we began with and last term are
the same, so
$$2g_\nu(\tau(x)\omega(y),\omega(z)) = -g_\nu(\omega(y), \Omega(x,z))
+ g_\nu(\Omega(z,y),\omega(x)) - g_\nu(\omega(z), \Omega(y,x)).$$
This establishes the unique determination of $\tau$ by $\Omega$ since
$g_\nu$ is nondegenerate and the tangent vectors $y,z$ to $FM$ can be
chosen freely to give all possible values for $\omega(y), \omega(z)$.
We can use this formula in both directions: we can start with
$\varphi$ and determine $\tau$ and hence $\varphi_0$; or we can start
with $\varphi_0$ and a given torsion $\Omega$ and determine $\varphi$.

\vspace{\bigskipamount}

There are several variants on the trick used to determine the
Levi-Civita connection. Levi-Civita used it to determine the
Christoffel symbols for a local coordinate basis. There is a
basis-free version of it known as the {\em Koszul formula} for covariant
derivatives:\index{Christoffel symbols}
$$2g(D_XY,Z) = Xg(Y,Z) + Yg(X,Z) - Zg(x,Y) + g([X,Y],Z) + g([Z,X],Y)
+ g(X,[Z,Y]).$$\index{Koszul formula}
If we restrict attention to $X, Y, Z$ chosen from a local frame, then
the first three terms of the Koszul formula vanish; for coordinate
basis fields the last three vanish and we recapture Levi-Civita's
formula.

\vspace{\medskipamount}
\index{coframe}
Frequently the most efficient way to calculate the Levi-Civita
connection is to use a local coframe $(\omega^i)$ and the first
structural equation with $\Omega$ assumed to be zero and $\varphi$
forced to be skew-adjoint. The fundamental theorem tells us that the
information contained in the first structural equation is enough to
determine $\varphi$. Often the equations can be manipulated so as to
apply Cartan's Lemma on differential forms:\index{Cartan's Lemma}
$$ {\rm If}\ \sum \theta^i\omega^i = 0$$
and the $\omega^i$ are linearly independent $1$-forms, then the $1$-forms
$\theta^i$ must be expressible in terms of the $\omega^i$ with a
symmetric matrix of coefficients.

\subsection{Isometries} \index{isometry}
An {\em isometry} between metric spaces is a mapping
which preserves the distance function and is $1$-$1$ onto. In particular,
it is a homeomorphism between the underlying topological spaces. If it
is merely $1$-$1$, then it is called an {\em isometric imbedding}. We use
the same words for the mappings which preserve a semi-Riemannian
metric:

\vspace{\medskipamount}
An {\em isometry} $F : M \to N$ from a semi-Riemannian manifold $M$
with metric tensor $g$ onto a semi-Riemannian manifold $N$ with metric
tensor $h$ is a diffeomorphism such that for all tangents $v,w \in
M_p$, for all $p \in M$, we have $h(F_*v,F_*w) = g(v,w)$.

\vspace{\medskipamount}\index{imbedding!isometric}
An {\em isometric imbedding} is a map satisfying the same condition
relating the tangent map and the metrics, but requiring only that it
be a differentiable imbedding; this does not mean that the topology
has to be the one induced by the map, only that the map be $1$-$1$ and
regular. Finally, for an {\em isometric immersion} we drop the
requirement that it be $1$-$1$.

\vspace{\medskipamount}
Since covariant tensor fields (those of type $(0,s)$) can be pulled
back by tangent maps in the same way that differential forms are, we
can also write the condition for isometric immersions as: $F^*h = g$.

\vspace{\medskipamount}
We have seen that there are auxiliary structures uniquely determined
by a semi-Riemannian metric or a Riemannian metric. Thus, the
Levi-Civita connection is uniquely determined by the metric tensor
$g$, and in the Riemannian case, lengths of curves and Riemannian
distance are determined by the metric tensor $g$. Moreover, the
curvature tensor is uniquely determined by the Levi-Civita connection.
These additional structures are naturally carried from one manifold to
another by a diffeomorphism. Thus, it is obvious that these auxiliary
structures are preserved by isometries.

\vspace{\medskipamount}
In particular, geodesics of the Levi-Civita connection are carried to
geodesics by an isometry; this includes their distinguished
parametrizations. Immediately we have that isometries commute with
exponential maps:\index{exponential map}
$$F \circ exp_p = exp_{Fp} \circ {F_*}_p.$$ 

\begin{thm} On a connected semi-Riemannian manifold an
isometry  is determined by its value and its tangent map at one point.
The group of isometries is imbedded in $FM$.
\end{thm}\index{bundle of frames}

\begin{cor} The group of isometries of a semi-Riemannian
manifold into itself is a Lie group. The dimension of the group of
isometries is at most $n(n+1)/2$, where $n$ is the dimension of the
original manifold.
\end{cor}

\subsection{Induced semi-Riemannian metrics} If we have a differentiable
immersion $F:M \to N$ and $N$ has a semi-Riemannian metric $h$, then
we get an induced symmetric tensor field of type $(0,2)$ on $M$,
$F^*h$. In the Riemannian case $F^*h$ is always positive definite,
hence a Riemannian metric on $M$; in the semi-Riemannian case it could
even happen that $F^*h$ is degenerate, or even if we assume that
doesn't happen, on different connected components of $M$, $F^*h$ could
have different indices. When $F^*h$ is indeed a semi-Riemannian
metric, we say that it is the metric on $M$ {\em induced} by $F$. Of
course, then $F$ becomes an isometric immersion (or imbedding or
isometry, depending on what other set-theoretic properties it has).

\begin{examp}[Euclidean and semi-Euclidean spaces] \index{semi-Euclidean spaces}
When we consider ${\bf R}^n$ with its standard coordinate 
vector fields $X_i = {\partial \over \partial x^i}$, we have first of all a parallelization, and
hence the connection of that parallelization. It serves to give the
usual identification of each tangent space of $\R^n$ with ${\bf R}^n$ itself.
In turn, the standard inner product $g_\nu$ of index
$\nu$ can be considered as defined on each tangent space, so that we
have a semi-Riemannian structure of index $\nu$, called the
semi-Euclidean space of index $\nu$. We denote this by ${\bf
R}^n_\nu$. When $\nu = 0$ it is Euclidean space. When $\nu = 1$ it is
called {\em Minkowski} space (although there are other things, Finsler
manifolds, called Minkowski space).\index{Minkowski space}

\vspace{\medskipamount}
The dual $1$-forms of the parallelization $X_i$ are the coordinate
differentials $\omega^i = dx^i$. When we put them into a column
$\omega$ and take exterior derivative we get $d\omega = 0$. Setting
$\varphi = 0$ clearly gives us the connection forms of the
parallelization; but parallel translation also clearly preserves
$g_\nu$, and torsion is obviously $0$. By the fundamental theorem it
follows that this same connection is also the Levi-Civita connection
for all of these semi-Riemannian metrics.

It is obvious that all of the translations $T_a:x \to x+a$ are
isometries of $g_\nu$. They form a subgroup of dimension $n$ of the
isometry group. Almost as obvious (compute the tangent map!), for any
orthogonal transformation $A \in O(n,\nu)$ the inner product $g_\nu$,
viewed as a semi-Riemannian metric, is preserved by $A$. This gives
another subgroup of isometries, of dimension $n(n-1)/2$. Together the
products of these two kinds of isometries form the full isometry group
of $\R^n_\nu$:

\vspace{\medskipamount}
The {\em semi-Euclidean motion group}: $\{T_a \circ A:x \to Ax + a : a
\in \R^n , A \in O(n,\nu)\}$.\index{semi-Euclidean motion group}
The motion group is transitive on $\R^n$, and at each point, the
tangent maps of isometries which fix that point are transitive on the
frames at that point. Thus, the induced group on the bundle of frames
is transitive; in fact, if we fix a base point of $F\R^n_\nu$,
then the motion group becomes identified with $F\R^n_\nu$ with
the base point as the identity of the group.
\end{examp}
\begin{examp}[Round spheres] \index{spheres!constant curvature}
The points at distance $R$ from the origin in
$E^n = R^n_0$ form the $n-1$-dimensional sphere $S^{n-1}$ of radius
$R$. The induced Riemannian metric has $O(n)$ as a group of
isometries, and it easy to check that it is transitive on points and
transitive on frames at (some conveniently chosen) point. Thus, we can
identify $FS^{n-1}$ with $O(n)$; the bundle projection can be taken to
be the map which takes an orthogonal matrix to $R \cdot$(first
column). Since the isometry group is transitive on frames, and
isometries preserve the curvature tensor of the Levi-Civita
connection, the curvature tensor has the same components with respect
to every frame. Along with symmetries shared by every Riemannian
curvature, the invariance under change of frame is enough to determine
the curvature tensor up to a scalar multiple. Certainly this is enough
excuse to say that $S^{n-1}$ has ``constant curvature''. However, we
shall amplify the meaning of constancy of curvature when we discuss
sectional curvature.
\end{examp}
\begin{prob}\label{prob23} Assume the symmetries of
the general Riemannian curvature tensor with respect to a frame:
$$R^i_{jhk} = -R^j_{ihk} = - R^i_{jkh} = R^h_{kij}$$
and
$$R^i_{jhk} + R^i_{hkj} + R^i_{kjh} = 0.$$\index{curvature!symetries}
\textsl{For a curvature tensor which has the same components with respect 
to every frame, discover what these components must be, up to a scalar 
multiple.}\end{prob}\index{curvature!constant}

\vspace{\medskipamount}\index{Maurer-Cartan equations!orthogonal group}
In the case $R = 1$, the unit sphere, the structural equations of the
Levi-Civita connection on the bundle of frames $O(n)$ are just the
Maurer-Cartan equations of $O(n)$. One has to separate the
left-invariant $1$-forms on the group into those which correspond to the
universal coframe and those which correspond to the connection forms.
This gives a way of calculating the curvature components for problem
23, except for the multiple.\index{coframe!universal}

\begin{examp}[Other quadric hypersurfaces] Analogously to what we
did with
round spheres, we consider quadric surfaces $\{x: g_\nu(x,x) = R\}$.
For nonzero $R$ this always gives a submanifold for which the induced
semi-Riemannian metric is nondegenerate and of constant index. That
new index is $\nu - 1$ if $R <0$ and $\nu$ if $R > 0$. In any case
this gives us examples of metrics having a maximal group of isometries
and ``constant curvature'' in the sense that the components of the
curvature tensor are the same relative to any frame.

\vspace{\medskipamount}
Especially important is the case $\nu = 1, R < 0$, for then we get a
Riemannian manifold ``dual'' to the round sphere, with constant
curvature (which we will call negative curvature). There are two
connected components; retaining only the upper component, we get the
{\em quadric surface model of hyperbolic $n-1$-space}.

\vspace{\medskipamount}
There is an elementary argument to show that the geodesics of any of
the quadric hypersurfaces are the intersections of the hypersurface
with planes through the origin of $\R^n$. We use the fact that
the isometries are so plentiful and they are induced by linear
transformations of the surrounding space, which take planes into
planes. Moreover, for any such plane, the isometries which leave it
invariant are transitive on the intersection with the quadric surface.
Thus, the parallel field generated by a tangent to that intersection
must be carried into another parallel field tangent to that
intersection, which can only be itself or its negative. In particular,
the geodesics of a round sphere are the great circles.
\end{examp}
\begin{prob}\label{prob24} Describe examples of submanifolds of a
semi-Riemannian manifold such that the metric induces tensor
fields on the submanifolds which are degenerate, or are nondegenerate
but not of constant index.\end{prob}

\begin{rem}[Geodesics of quadric hypersurface] The idea of using the
plethora of isometries to show that the geodesics of a quadric
hypersurface $g_\nu(x,x) = R$ are the intersections with planes
through the origin is valid, but the argument given in the last
paragraph does not work. Here is a correct argument. We show that
there is an isometry leaving the intersection of the plane and
hypersurface pointwise fixed, but taking every tangent vector
perpendicular to the intersection into its negative. We must assume
that the tangent vectors $v$ to the intersection are not null vectors
(for which $g_\nu(v,v) = 0$). The desired isometry is described in
terms of a frame for the vector space $\R^n$ with the bilinear
form $g_\nu$. The first frame vector is a normal to the hypersurface
$\grad g_\nu$ (as a quadratic form $g_\nu$ can be considered to be a
real-valued function on $\R^n$, and the gradient is taken
relative to the inner product $g_\nu$) at some point of the
intersection. The second frame member is to be tangent to the
intersection at that same point. Then the frame is filled out in any
way. The isometry then takes this frame into the frame having the same
first two members (so that it fixes the plane) and the remaining
members of the frame replaced by their negatives.

\vspace{\medskipamount}
The existence of such an isometry forces the intersection curve to be
a geodesic. The parallel field along that intersection generated by
the tangent to the geodesic at the base point, i.e., the second frame
member, must be carried to a parallel field along the (fixed)
intersection curve by the isometry. Since one vector, at the base
point, of the parallel field is fixed, the whole field is fixed. But
the only vectors fixed by the isometry are tangent to the
intersection. Hence the (unit) tangent field along the intersection is
parallel, making that curve a geodesic.

\vspace{\medskipamount}\index{quadric hypersurface}
In case the induced metric on the quadric hypersurface has nonzero
index, there will be null vectors, and the proof given will not apply
to planes through the origin tangent to those null vectors. But the
result is still true, since the limit of geodesics is still a
geodesic, and the excluded planes can be obtained as limits of the
others.
\end{rem}
\subsection{Infinitesimal isometries--Killing fields} \index{Killing field}
A one-parameter\index{derivative!Lie}
subgroup in the isometry group of a semi-Riemannian manifold is, in
particular, a one-parameter group of diffeomorphisms of the manifold.
Hence, it is the flow of a complete vector field. More generally there
will be vector fields whose local flows will be isometries of the open
sets on which they are defined. These are called {\em infinitesimal
isometries} or {\em Killing fields}. We give an equation which
describes Killing fields, derived by using the idea of a Lie
derivative. When we differentiate tensors carried along by a flow with
respect to the flow parameter, we get the Lie derivative of the tensor
field with respect to the vector field generating the flow. When the
flow consists of local isometries and the tensor field is the metric
tensor $g$ itself, the tensors being differentiated are constant, so
the derivative is $0$ Hence, we have: 

\begin{prop} If $J$ is a Killing field, then $L_Jg = 0$.
Consequently, a Killing field is characterized by the fact that the
linear map on each tangent space given by $v \to D_vJ$ is
skew-adjoint with respect to $g$.
\end{prop}
We derive the consequence of $L_Jg = 0$, the {\em equation of a
Killing field}, namely, $g(D_xJ,y) = - g(x,D_yJ)$, as follows:
$$0 = (L_Jg)(X,Y) = Jg(X,Y) - g(L_JX,Y) - g(X,L_JY)$$
$$\qquad = g(D_JX,Y) + g(X,D_Y) - g([J,X],Y) - g(X,[J,Y])$$
$$\qquad = g(D_JX,Y) + g(X,D_JY) - g(D_JX - D_XJ,Y) - g(X,D_JY-D_YJ)$$
$$\qquad = g(D_XJ,Y) + g(X,D_YJ).$$
In terms of a local coordinate basis, applied with $x,y$ chosen from
all pairs of coordinate vector fields, the equation of a Killing field
becomes a system of linear first order partial differential equations
for the coefficients of $J$. Since the group of isometries has
dimension at most $n(n+1)/2$, we know {\em a priori} that there are
at most that many linearly independent solutions for $J$. In the case
of Euclidean space we already know all the isometries, and hence can
calculate the Killing fields from that knowledge too, but it is an
interesting exercise to calculate them as solutions of a system of
PDE.

The Killing fields are named after Wilhelm Killing (1847-1923),
who discovered the above equations.

\vspace{\smallskipamount}
\begin{prob}\label{prob25} For ${\bf E}^n$, calculate the Killing
fields by solving the system of PDE's.\end{prob}

\vspace{\smallskipamount}
There is a simple, but useful, relation between geodesics and Killing
fields, generalizing a theorem of Clairaut on surfaces of revolution.
\index{conservation of energy}

\begin{lem}[Conservation Lemma] Let $\gamma$ be a geodesic and $J$
a Killing field. Then $g(\gamma', J) = c$, a constant along $\gamma$.
If $\gamma$ has unit speed, then $c$ is a lower bound on the length
$|J|$ of $J$ along $\gamma$, and hence, $\gamma$ lies in the region
where $|J| \ge c$. If $\gamma$ is perpendicular to $J$ at one point,
then it is perpendicular to $J$ along its extent.
\end{lem}

\begin{cor}[Clairaut's Theorem] If $\gamma$ is a geodesic on 
a surface of revolution, $r$ is the distance from the axis, and
$\varphi$ is the angle $\gamma$ makes with the parallels, then $r
\cos\varphi = r_0$ is constant along $\gamma$. Hence, $\gamma$ can
never pass inside the ``barrier'' $r = r_0$.
\end{cor}\index{Clairaut's Theorem}\index{barrier}

Of course, the proposition can be interpreted as saying that the level
hypersurfaces of $|J|$ form barriers to geodesics even in the general
case. For the surface of revolution we take $J = {\partial \over
\partial \theta}$, where $\theta$ is the angular cylindrical
coordinate about the axis of revolution in space; $J$ is a Killing
field of Euclidean space and its restriction to any surface of
revolution is tangent to that surface.

\vspace{\medskipamount}
Clairaut's theorem is a very powerful tool for analyzing the
qualitative behavior of geodesics on a surface of revolution. We
develop this theme in the following problems, in which we suppose that
the profile curve is expressed parametrically in terms of its
arclength $u$ by giving $r$ and $z$ as functions of $u$: $r = f(u), z
= h(u)$.\index{geodesic!surface of revolution}

\vspace{\smallskipamount}
\begin{prob}\label{prob26} If $u_0$ is a critical point of $f$,
then the parallel corresponding to this value of $u$ is a geodesic. If
$u_0$ is not a critical point of $f$, then on any geodesic tangent to
the corresponding parallel $r$ has a nondegenerate local minimum at
the point of tangency.\end{prob}

\vspace{\medskipamount}
\begin{prob}\label{prob27} The meridians
$\theta = \,constant$ are geodesics. On the other geodesics, $\theta$
is strictly monotonic, and $u$ is strictly monotonic on arcs which
contain no barrier point.\end{prob}\index{meridian}

\vspace{\smallskipamount}
\begin{examp} On the usual donut torus $r$ is not monotonic between
barriers for most geodesics.
\end{examp}

\vspace{\smallskipamount}
\begin{prob}\label{prob28} Classify the geodesics according to whether
$u$ is periodic or not along the geodesic, and, if not, according to how
it behaves relative to its barrier parallels.\end{prob}

\vspace{\smallskipamount}
\begin{prob}\label{prob29} Suppose that we have a
geodesic which traverses the gap between its barriers, neither of
which are geodesics. Show that the angular change $\Delta\theta$ in
$\theta$ between successive collisions with barriers is a continuous
function of the geodesic for nearby geodesics (obtained, say, by
varying the angle $\varphi$ at a point through which we assume they
all pass). Hence, either $\Delta\theta$ is constant or there are
nearby geodesics which fill up the barrier strip densely and others
which are periodic (closed) with arbitrarily long periods. The case of
$\Delta\theta$ constant can actually occur. How?\end{prob}

\section{Calculus of variations}\index{calculus of variations}\index{energy}
\index{Euler equations}
An important technique for discovering
the extremal properties of functionals (usually looking for minimums
of length, energy, area, etc.) is the calculus of variations. A
putative extremum is presumed to be within a family, whereupon the
derivative of the functional with respect to the parameter(s) of the
family must be zero. Using the arbitrariness of the choice of family,
we then obtain the {\em Euler equations} for extremals of the
functional. These are usually differential equations which the mapping
representing the extremal must satisfy. It is also called the {\em
first variation condition}; ``first'' refers to ``first derivate with
respect to the parameter''. Once the condition for the first variation
to be zero is satisfied, the analogue of the second derivative test
for a minimum is often employed; hence we must calculate the ``second
variation'', similar to the Hessian of a function at a critical point.
\index{Hessian}
Thus, the second variation is a quadratic form on the
(infinite-dimensional) tangent space to the space of all mappings in
question. The condition for a nondegenerate minimum is then that the
second variation be positive definite. More subtle results are
obtained by determining the {\em index} of the second variation, that
is, the maximal dimension of a subspace on which the second variation
is negative definite. If the Euler equations are elliptic, then the
index will usually be finite. We won't get into this subtler analysis
very much; it is known as ``Morse theory'', named after Marston Morse,
who developed it for the length functional on the space of paths with
remarkable success.\index{Morse theory}\index{Morse, M.}

\vspace{\medskipamount}\index{rectangles - smooth}\index{longitudinal curve}
\index{transverse curve}
\subsection{Variations of Curves--Smooth Rectangles}. For a given curve
$\gamma: [a,b] \to M$ a {\em variation} of $\gamma$ is a one-parameter
family of curves such that $\gamma$ is the curve obtained by taking
the parameter value 0. Formally this is a mapping $Q: [a,b]\times [0,
\epsilon] \to M$, smooth as a function of two variables, such that
$\gamma(s) = Q(s,0)$. We call $Q$ a {\em smooth rectangle} with base
curve $\gamma$. The curves we get by fixing the second variable $t$
and varying $s$ are called the {\em longitudinal} curves of $Q$; the
curves obtained by fixing the first variable $s$ are called the {\em
transverse} curves of $Q$. The velocity fields of the longitudinal
curves are united in the {\em longitudinal vector field}, which is
formally a vector field on the map $Q$ [see Bishop and Goldberg,
\S5.7],\index{Goldberg, S.I.} and can be denoted either $\partial Q \over \partial s$ or
$Q_* ({\partial \over \partial s})$. Similarly, we have the {\em
transverse vector field} $\partial Q \over \partial t$ or
$Q_*({\partial \over \partial t})$. The {\em variation vector field}
is the vector field on $\gamma$ obtained by restricting $\partial Q
\over \partial t$ to the points $(s,0)$. The first variation of length
or energy depends only on the variation vector field; and when the
base curve is critical with respect to these functionals, and we
restrict to variations which satisfy reasonable end conditions, then
the second variation also depends only on the variation vector field.
(This is analogous to the situation where a smooth function has a
critical point, whereupon the second derivative becomes tensorial.)

\subsection{Existence of Smooth Rectangles, given the Variation Field}. If we
are given a vector field $V$ on a curve $\gamma$, then we can obtain a
rectangle having $V$ as its variation field by using the exponential
map of some connection:
$$Q(s,t) = \exp _{\gamma(s)}(tV(s)).$$
\index{rectangles - smooth!existence}

\subsection{Length and Energy} If we have a Riemannian metric, we have
already defined the length of a curve $\gamma$ as the integral
$|\gamma|$ of its speed. The {\em energy} of $\gamma$ is 
$$E(\gamma) = \int_a^b g(\gamma_*,\gamma_*)\,ds.$$
By Schwartz inequality applied to the functions $f=1$ and
$\sqrt{g(\gamma_*,\gamma_*)}$ on the interval $[a,b]$ we obtain 
$$|\gamma|^2 \le (b-a)E(\gamma).$$

The condition for equality is that $g(\gamma_*,\gamma_*)$ be
proportional to $f=1$ in the sense of integration, i.e., at all but a
set of measure zero.  Thus, we are allowed to apply this condition on
piecewise smooth curves, so the conclusion is that we have equality
for piecewise smooth curves if and only if the speed is constant.
Hence, if we reparametrize curves so that they have constant speed,
which doesn't change their lengths, then the energy functional becomes
practically the same as the length functional for the purposes of
calculus of variations. There is a technical advantage gained from
using energy instead of length because the formulas for derivatives
are simplified--similar to what happens in calculus when you choose to
differentiate the square root of a function implicitly.

\begin{prop} If M is a Riemannian manifold, then a curve has 
minimum energy if and only if the curve has minimum length and is
parametrized by constant speed. (The comparison is among smooth curves
connecting the same two points, parametrized on the same interval
$[a,b]$.)\end{prop}\index{critical energy}\index{critical length}

\index{Lorentz manifold}\index{time-like}\index{space-like}\index{elapsed time}
In semi-Riemannian manifolds the concept of length does not have a
meaning, so that the significant functional for the calculus of
variations of curves is energy. However, in a Lorentz manifold of
index $n-1$ the curves $\gamma$ for which $g(\gamma',\gamma') > 0$,
which are called time-like curves, have a special significance: they
represent paths of ``events'' which an object could experience as time
passes. For these curves the usual expression for length is called the
``elapsed time'', measuring the amount that a clock would change as it
moved with the object. In a Lorentzian vector space the direction of
the Cauchy-Schwartz and triangle inequalities, restricted to time-like
vectors, is reversed. Thus, in Lorentz geometry we find that the
time-like geodesics are the curves which locally {\em maximize} energy
among nearby time-like curves. If one person moves on a geodesic while
a second person starts out at the same time and place, accelerates
away, and then steers backs to join the first person, the first person
will age more than the second! A discussion of the inequalities is
given in O'Neill, Semi-Riemannian Geometry, p 144.\index{O'Neill, B.}
\index{triangle inequality!reversed}

\index{first variation!arclength}\index{first variation!energy}
\begin{thm}\label{thm:firstvar}[First Variation of Energy and Arclength]
Let $Q$ be a smooth 
rectangle in a semi-Riemannian manifold, with base curve $\gamma$ and
variation field $V$. Denote the longitudinal curves by $Q_t:s \to
Q(s,t)$. Then the first variation of energy is given by
$${dE(Q_t) \over dt}(0) = 2\left[g(V(b),\gamma'(b)) -
g(V(a),\gamma'(a)) - \int_a^b
g(V(s),D_{\gamma'(s)}\gamma')\,ds\right].$$
If $g(\gamma', \gamma') = C^2$ is  positive, then the first variation
of arclength also makes sense and is given by
$${dL(Q_t) \over dt}(0) = {dE(Q_t) \over dt}(0)/2C.$$
\end{thm}
\begin{rem}\index{coframe!universal}
Covariant derivatives of vector fields along curves are defined in
terms of parallel frames along curves. This amounts to lifting the
curves to the bundle of frames, pulling back the universal coframe
$\omega$ and connection form $\varphi$, then using the usual formula
($\omega(D_XY) = (X+\varphi(X))\omega(Y))$. For an alternative approach
which does not use bundles, but instead develops the idea of
connections on maps and their pullbacks, see Bishop and Goldberg,
Chapter 5. In particular, the fact that the torsion 0 exchange rule
$D_{\partial \over \partial t}{\partial Q \over \partial s} =
D_{\partial \over \partial s}{\partial Q \over \partial t}$ has
meaningful terms and is true follows from the invariance of exterior
calculus under pullbacks.\index{connection!on map}
\end{rem}

\index{energy-critical curve}
We define a curve to be {\em energy-critical} or {\em length-critical}
within spaces of curves by the requirement that for all variations $Q$
of the curve in the space the first variation of energy or length is
$0$. By elementary calculus it then follows that energy-minimal and
length-minimal curves in the space are also critical. The space of
curves used can be the smooth curves connecting two fixed points and
parametrized on a fixed interval $[a,b]$. More generally, we can let
the ends of the curves vary on submanifolds. The generality of the
domain interval $[a,b]$ is convenient because it allows us to apply
results to subintervals immediately: if a curve is critical or minimal
for certain endpoint conditions on the interval $[a,b]$, then its
restriction to $[a',b'] \subset [a,b]$ is critical or minimal for
fixed endpoint variations over $[a',b']$. Thus, after analyzing the
fixed endpoint case we can easily discover what additional condition
is needed for the variable endpoint case.

\begin{thm} If
$\gamma$ is energy-critical (length-critical) for smooth curves from
$p$ to $q$, then $\gamma$ is a geodesic (pregeodesic).
\end{thm}\index{geodesic}\index{pregeodesic}

There is a subtle mistake which should be avoided at this point in the
development: we cannot immediately assert from the theorem that
minimizing curves are geodesics. There are two gaps in the argument:
minimizers may not exist, or if they do they may not be smooth. For
the case at hand, dealing with the length of curves, we have taken
care of the existence problem using the Arzel\`a-Ascoli Theorem.
\index{Arzel\`a-Ascoli Theorem}
However, the example of the taxicab metric shows that we are still
required to establish smoothness. Moreover, the failure to take care
of these gaps in other contexts has generated some famous mistakes:
recall the 19{\em th} century dispute over the Dirichlet principle;
and in our era, the Yamabe ``theorem'' turned into the Yamabe {\em
problem} precisely because the convergence to a critical map failed,
yielding a nonsmooth object which was a generalized function, not a
smooth map.\index{Arzel\`a-Ascoli Theorem}\index{taxicab metric}
\index{Dirichlet principle}\index{Yamabe problem}

\index{Gauss's Lemma}
For the length functional on curves we handle the difficulty by
imbedding the geodesic in a field of geodesics and applying Gauss's
Lemma. For the special case of Euclidean space Problem \ref{prob4} and its hint
shows how this works by using as the field of geodesics a field of
parallel lines.

\begin{lem}[Gauss's Lemma] Let $Q$ be a variation of a
geodesic $\gamma$ such that all the longitudinal curves $Q_t$ are
geodesics with the same speed and having variation field $V$. Then
$g(\gamma', V)$ is constant. In particular, if $V$ is orthogonal to
$\gamma$ at one point, it is orthogonal everywhere.
\end{lem}
\begin{rem} We have already seen a special case of Gauss's Lemma,
namely, the proposition on Killing fields which we specialized to get
Clairaut's Theorem. Given a Killing field $J$ and a geodesic segment
$\gamma$, we generate a variation $Q$ by applying the flow of $J$ to
$\gamma$; this $Q$ satisfies the hypotheses of Gauss's Lemma. The fact
that the variations of geodesics generated by the flow of a Killing
field always yields geodesics for the longitudinal curves shows that a
Killing field, restricted to any geodesic, is always a Jacobi field,
defined as follows.
\end{rem}
\begin{defn}[Jacobi fields] If $\gamma$ is a geodesic, a {\em
Jacobi field} along $\gamma$ is a vector field $J$ on $\gamma$ such
that on each subinterval of $\gamma$, $J$ is the variation field of a
variation for which the longitudinal curves are geodesics.
\end{defn}\index{Jacobi field}

The longitudinal curves in this definition do not all have to have the
same speed, so that the Jacobi fields dealt with by Gauss's Lemma are
a little special. It is rather trivial to analyze the Jacobi fields
which come from variations which don't move the base curve out of its
trace, instead simply sliding and stretching the geodesic along
itself. We have already stated how a geodesic can be reparametrized to
remain a geodesic, and that's all there is to it. Thus, a Jacobi field
which is tangent to the geodesic it sits on has the form $J(s) =
(as+b)\gamma'(s)$ for some constants $a$ and $b$. We also have a clear
idea of how freely we can vary geodesics in general: we can move the
initial point to any neighboring point and the initial velocity to any
neighboring velocity. Thus, the geodesic variations correspond to a
neighborhood of the initial velocity in the tangent bundle, and the
Jacobi fields, in turn, correspond to the tangents to the tangent
bundle at that point.
\begin{prop} The Jacobi fields along
a given geodesic form a space of vector fields of dimension $2n$.
\end{prop}

\begin{prob}\label{prob30} (Rather trivial) Show that a vector field $J$
on $M$ is a Killing field if and only if along every geodesic $J$ 
has constant inner product with the geodesic's velocity.\end{prob}

\vspace{\smallskipamount}
In fact, we shall soon calculate that the Jacobi fields satisfy a
homogeneous linear second-order differential equation, so that the
space of them is actually a vector space of dimension $2n$. The
tangential ones form a subspace of dimension $2$, and Gauss's Lemma
says, in effect, that the ones orthogonal to the base at every point
form a subspace of dimension $2n-2$, complementary to the tangential
ones.

\begin{prop} In a normal coordinate neighborhood of
$p$ the images under $\exp_p$ of the hypersurfaces $g(v,v) = c$ form a
family of hypersurfaces orthogonal to the radial geodesics from $p$.
The tangent map ${\exp_p}_*$ takes vectors orthogonal to the rays from
the origin of $M_p$ to vectors orthogonal to the corresponding
geodesic from $p$.
\end{prop}

Note that both Gauss's Lemma and the above Proposition hold in the
semi-Riemannian case. In the Riemannian case the local minimizing
property of geodesics now follows by using as a field of geodesics the
radial geodesics of the same speed starting from the initial point and
applying the property given about that field in the Proposition.  The
local maximizing property of time-like geodesics in a Lorentz manifold
of index $n-1$ is done in just the same way, using the fact that
removing the orthogonal component of the tangent to a nearby curve
then {\em increases} rather than decreases the elapsed time.

\begin{thm} Let $M$ be a Riemannian manifold, $p \in M$, and
$R$ the radius of a normal coordinate ball $B$ at $p$. Then for any $q
\in B$ the radial geodesic segment from $p$ to $q$ is the shortest
smooth curve from $p$ to $q$.
\end{thm}

The ``radius'' referred to in the theorem is the distance measured in
terms of the Euclidean formula with the normal coordinates. The value
of that radius function at $q$ is the length of the radial geodesic.
Hence we have

\begin{cor} If the ball of radius $R$ is
normal, then the distance function $q \to d(p,q)$ is smooth on
$B -\{p\}$, and its square is smooth at $p$ as well.
\end{cor}
\begin{thm}[The Jacobi equation] Let $J$ be a Jacobi field along a 
geodesic $\gamma$. We denote covariant derivatives with respect to
$\gamma'$ by a prime. Then
$$J'' + R_{\gamma' J}\gamma' = 0.$$
\end{thm}\index{Jacobi equation}

The proof proceeds in a straightforward fashion by applying the first
and second structural equations to a smooth rectangle generating the
Jacobi field.

\begin{examp} In Euclidean space or in $\R_\nu^n$ the
connection is flat, $R = 0$. Thus, the Jacobi equation is $J'' = 0$.
But covariant differentiation is just differentiation of components
with respect to the standard Cartesian coordinates. Hence the Jacobi
fields are linear fields, $J(s) = sA + B$, where s is a linear
parameter on a straight line.
\end{examp}

\begin{examp} In a sphere of radius $1$ the curvature tensor is given
for a unit speed geodesic $\gamma$ by
$$R_{\gamma'J}\gamma' = 
\begin{cases}
J,& \text{if $J \perp \gamma'$};\\
       0,& \text{if $J$ is tangent to $\gamma$}.
\end{cases}
$$
Hence the Jacobi differential equation splits into uncoupled second
order differential equations for the components of $J$ with respect to
a parallel frame along $\gamma$: $J_i'' + J_i = 0$ for the components
orthogonal to $\gamma$, and $J_n'' = 0$ for the component tangent to
$\gamma$.
\end{examp}\index{Jacobi equation!of sphere}

\begin{examp} For hyperbolic geometry the curvature is opposite in sign from that on the sphere: $J_i'' - J_i = 0$, if $i<n$, $J_n'' = 0$.
\end{examp}
\begin{examp} On a surface the curvature operator is expressed in
terms of one component, the Gaussian curvature $K$. If $J$ is
orthogonal to unit vector $\gamma'$ we have
$R_{\gamma'J}\gamma' = KJ$. Thus, the Jacobi
equation for a Jacobi field orthogonal to a geodesic is $J'' + KJ = 0$,
which is regarded as a scalar differential equation for the one
component of $J$.
\end{examp}

\section{Riemannian curvature}

\subsection{Curvature Symmetries} \textsl{The matrix of $R_{XY}$ is given
by the value $- \Phi(X,Y)$ of a $2$-form, so that it is skew-symmetric in
$X,Y$: ({\em skew-symmetry in arguments})
\begin{equation}\label{eq:1}
R_{XY} = - R_{YX}.
\end{equation}\index{curvature!symmetries}
The matrices $- \Phi$ have values in $so(n,\nu)$, so on $M_p$ the
{\em operators are skew-adjoint with respect to $g$}:
\begin{equation}\label{eq:2}
g(R_{XY}Z,W) = - g(Z,R_{XY}W).
\end{equation}
If we take the exterior derivative of the first structural equation
and substitute, using both structural equations, for $d\omega$ and
$d\varphi$, we get
$$0 = d^2\omega = - d\varphi\wedge\omega + \varphi\wedge d\omega$$
$$\qquad =\varphi\wedge\varphi\wedge\omega - \Phi\wedge\omega
-\varphi\wedge\varphi\wedge\omega = - \Phi\wedge\omega.$$
Evaluating this on $X,Y,Z$, using the shuffle permutation rule for
exterior products, gives
\begin{equation}\label{eq:3}
R_{XY}Z + R_{YZ}X + R_{ZX}Y = 0.
\end{equation}}
We call this the {\em cyclic curvature symmetry}. Many references call
it the ``first Bianchi identity'' (or even worse, ``Bianchi's first
identity''), but that name is solely due to its formal resemblance to
the Bianchi identity (see below). It was known to Christoffel and
Lipschitz in 1871 (when Bianchi was 13) and probably to Riemann in
1854.

\vspace{\medskipamount}
A fourth identity, {\em bivector symmetry} or {\em symmetry in pairs},
is a consequence of (\ref{eq:1}), (\ref{eq:2}), (\ref{eq:3}):
\begin{equation}\label{eq:4}
g(R_{XY}Z,W) = g(R_{ZW}X,Y).
\end{equation}
The name ``bivector symmetry'' comes from a standard identification of
$\bigwedge^2M_p$ with $so(M_p,g)$, the skew-adjoint endomorphisms
of $M_p$. We extend $g$ to bivectors by the usual determinant method:
$$g(X\wedge Y, Z\wedge W) = g(X,Z)g(Y,W) - g(X,W)g(Y,Z).$$
Then if $A:M_p \to M_p$ is skew-adjoint, we can interpret $A$ as a
$(1,1)$ tensor, then raise the covariant index to get a skew-symmetric
$(2,0)$ tensor, also labeled $A$; this amounts to
$$g(A,X\wedge Y) = g(AX,Y).$$
Thus, $R:X\wedge Y \to R_{XY}$ is interpreted as a linear map 
$\bigwedge^2M_p \to so(M_p,g) \approx \bigwedge^2M_p$, for which
(\ref{eq:4}) yields
\begin{equation}\label{eq:4'}
g(R_{XY},Z\wedge W) = g(X\wedge Y, R_{ZW}).
\end{equation}
In this way we have an interpretation of $R$ as a self-adjoint linear
map of $\bigwedge^2M_p$.

\begin{prob}\label{prob31} Let $S:M_p \to M_p$ be
a symmetric (with respect to $g$) linear transformation. Extend $S$ to
act on the Grassmann algebra over $M_p$ as a homomorphism, which we
denote on $\bigwedge^2M_p$ by $S\wedge S$. Show that $R = S\wedge S$,
turned into a $(1,1)$ tensor-valued $2$-form by equation (\ref{eq:4'}) satisfies
all the identities of a curvature tensor.\end{prob}

\subsection{Covariant Differentials} \index{covariant differential}
If $T$ is a tensor field of type
$(r,s)$, then we define a tensor field $DT$ of type $(r, s+1)$ by
letting $DT(..., X) = D_XT(...)$. For example, the Riemannian Hessian
of a function $f$ is $Ddf$, given explicitly by $Ddf(X,Y) = XYf
-(D_XY)f$. At a critical point of $f$ this coincides with the natural
Hessian, and it is always symmetric.\index{Hessian}

\subsection{Exterior Covariant Derivatives} \index{covariant derivative!exterior}
If a tensor field $T$ is
skew-symmetric in some of its vector arguments, say the last $t$ of
the $s$ arguments, then after forming $DT$ we can skew-symmetrize on
the last $t+1$ vector arguments, obtaining a tensor field which we
denote $dT$. This is called the {\em exterior covariant derivative}
of $T$ viewed as an $(r,s-t)$ tensor valued $t$-form. The second
exterior derivative satisfies an identity $d^2T = R\wedge\cdot T$,
which requires an explanation, so although it is not $0$, it does not
depend on derivatives of components of $T$, only the pointwise value
of $T$ and the curvature of the space.

\vspace{\smallskipamount}
\begin{prob}\label{prob32} (a) Explain how the explicit expression for $Ddf$ 
comes from a product rule for the ``product'' $(df,X) \to df(X)$. (b) Find
the corresponding explicit expression for $(DR(X,Y,Z))W$, taking as
conventional: $X,Y$ are the $2$-form arguments of $R$ which appear in
$R_{XY}$, $W$ is the vector on which $R_{XY}$ operates, and $Z$ is the
additional argument for $D$.\end{prob}

\begin{rem} In the semi-Riemannian case, keeping track of the signs
$g(e_i,e_i) = \epsilon_i$ is a source of considerable irritation. In
what follows we do not generally sum on a repeated index when it is
attached only to $\epsilon$ and one other letter; sometimes we use the
sum convention in other settings, sometimes we stick in sum signs.
Usually it should be clear from context whether a sum is intended.
\end{rem}
\begin{thm}[The Bianchi Identity] \index{Bianchi identity}The exterior covariant derivative of 
the curvature $2$-form is $0$. The formulation $dR = 0$ can be expanded
to give the usual expressions as follows. The covariant differential
$DR$ is already skew-symmetric in the $2$-form arguments it inherited
from $R$, so that in order to skew-symmetrize on those two and the
additional one we only need to throw in the cyclic permutation of the
three. Thus, we get the equation
$$DR(X,Y,Z) + DR(Y,Z,X) + DR(Z,X,Y) = 0,$$
for all vector fields $X, Y, Z$. In terms of components with respect
to a basis (not necessarily a frame), we take the first two indices of
$R$ to be the indices of its matrix as a linear operator, so that one
is up the other down, then the next two are its indices as a $2$-form,
hence subscripts. Taking the covariant differential adds another
subscript, which is customarily separated from the others by a
semicolon ``;'' or a solidus ``$|$''. Thus, in index notation the
Bianchi identity is written
$${R^i}_{jhk|l} +{R^i}_{jkl|h} +{R^i}_{jlh|k} = 0.$$
\end{thm}

The usual application of the Bianchi identity is to prove Schur's
theorem that a semi-Riemannian manifold with pointwise constant
curvature and dimension at least $3$ has constant curvature. We state it
precisely and then combine the proof with a proof of the Bianchi
identity.\index{curvature!constant}

\begin{thm}[Schur's Theorem] Let $M$ be a connected
semi-Riemannian manifold of dimension $>2$ 
such that there is a function $K:M \to \R $ such that the local
expression for the curvature forms is $\Phi_j^i =
\epsilon_jK\omega^i\wedge\omega^j$. Then $K$ is constant.
\end{thm}

The local coframe expression given is what we call
``pointwise constant'' curvature.\index{coframe}

\vspace{\smallskipamount}

\begin{proof} The hypothesis about local coframe expression passes
over to a claim that the global expression for the curvature forms on
$FM$ looks the same, $\Phi_j^i = \epsilon_jK\omega^i\wedge\omega^j$,
where now $K$ is the pullback to $FM$ of the former $K$. Thus, $K$ is a
real-valued function which is constant on fibers, so that to show it
is constant we only need to show that its derivatives in the
horizontal directions are $0$, i.e., $E_kK = 0$.

\vspace{\smallskipamount}
In general we have
$$d\Phi_j^i = d(d\varphi_j^i + \varphi_k^i \wedge \varphi_j^k)$$
$$\qquad = d\varphi_k^i \wedge \varphi_j^k - \varphi_k^i \wedge
d\varphi_j^k.$$

Now the general procedure for calculating the exterior covariant
derivative of a tensor-valued form is to pass to the corresponding
equivariant form on $FM$, take exterior derivative, and then take the
horizontal part. We note that every term of $d\Phi_j^i$ has vertical
factors, so the horizontal part is $0$. This proves the Bianchi identity.

Now we continue the calculation supposing that the curvature forms
have the pointwise constant curvature expression.
$$d(K\omega^i\wedge \omega^j) = dK\wedge \omega^i \wedge\omega^j +
K(-\varphi_k^i \wedge\omega^k \wedge\omega^j + \omega^i \wedge
\varphi_k^j \wedge\omega^k)$$
$$\qquad = (E_kK)\omega^k \wedge \omega^i \wedge \omega^j +
K(-\varphi_k^i \wedge\omega^k \wedge\omega^j + \omega^i \wedge
\varphi_k^j \wedge\omega^k).$$
Now we see that in order for this expression to have no horizontal
component the first term must vanish for all $i$ and $j$. If $n > 2$,
we can choose $i$ and $j$ different from any given $k$, so that we
must have $E_kK = 0$ for all $k$, as required.
\end{proof}

\subsection{Sectional Curvature} \index{sectional curvature}
\index{curvature!sectional}The curvature tensor is the major invariant
of Riemannian geometry; it entirely determines the local geometry in a
sense spelled out precisely by the Cartan Local Isometry Theorem.
\index{Cartan Local Isometry Theorem}
However, it is too complicated to use directly in the formulation of
local hypotheses which have significant geometric consequences. Thus,
it is important to repackage the information it conveys in a more
tractable form, the sectional curvature function. In two-dimensional
spaces this reduces to a function on points, since there is only one
$2$-plane section at each point, namely, the tangent plane; the
sectional curvature is then just the Gaussian curvature, $K =
R_{1212}$, which is a component of $R$ with respect to a frame. In
higher dimensions the sectional curvature is still a real-valued
function, but the domain consists of all $2$-dimensional subspaces of
all the tangent spaces, which we call {\em sections}.\index{Gaussian curvature}

\vspace{\medskipamount}

Let $\sigma$ be a section and let $(v,w)$ be a frame for $\sigma$.
Then the sectional curvature of $\sigma$ is
\begin{equation}\label{eq:dagger}
K(\sigma) = g(R_{vw}v,w).
\end{equation}
In the semi-Riemannian case this must be modified slightly: it is not
defined if the section $\sigma$ is degenerate, so there is no frame
for it; but even in the case where the metric is indefinite on
$\sigma$ we change the sign, so that the result will conform to the
more general formula for $K(\sigma)$ in terms of an arbitrary basis
$(x,y)$ of $\sigma$. Then we can write $v = ax + by, w = cx + dy$, and
we get, using the symmetries of $R$
\begin{equation}\label{eq:ddagger}
K(\sigma) = (ad-bc)^2g(R_{xy}x,y).
\end{equation}
A straightforward calculation, starting with the inverse expressions
for $x$ and $y$ in terms of $v$ and $w$, gives
\begin{equation}\label{eq:areal}
g(x,x)g(y,y) - g(x,y)^2 = \pm (ad-bc)^{-2}.
\end{equation}
The sign is positive if $g$ is definite (positive or negative) on
$\sigma$, and negative if $g$ is indefinite. In the Riemannian case
the geometric meaning of the expression (\ref{eq:areal}) is that it is the square
of the area of the parallelogram with edges $x,y$. If we take another
frame for $(x,y)$, then the calculation shows that the result for
$K(\sigma)$ is the same whether we use (\ref{eq:dagger}) or
(\ref{eq:ddagger}), showing that we have really defined $K$.
Following the prescription given for the sign, the general formula
given for sectional curvature in terms of an arbitrary basis is thus
$$K(\sigma) = \frac{g(R_{xy}x,y)}{g(x,x)g(y,y) - g(x,y)^2}.$$
For sections in the direction of pairs of frame vectors, $\sigma_{ij}$
spanned by $e_i, e_j$, we introduce the signs $\epsilon_i =
g(e_i,e_i)$ and get expressions for sectional curvature in terms of
curvature components:\index{curvature!constant}
$$K_{ij} = K(\sigma_{ij}) = g(R_{e_ie_j}e_i,e_j)\epsilon_i\epsilon_j
   = \epsilon_i\epsilon_j R_{ijij} = \epsilon_iR_{iji}^j = R_{ij}^{ij}.$$
\begin{rem} If the curvature at $p \in M$ is constant in the sense
of having the same components with respect to every frame, then as a
map $R:\bigwedge^2M_p \to \bigwedge^2M_p$ it must be a constant
multiple $K\cdot I$ of the identity map. ({\em A priori} this claim is
true for possibly different multiples depending on the signature of
the section, but one checks that it works in general by making some
Gallilean boost change of frames.) Hence the sectional curvature is
that same constant $K$ for all sections.\index{Gallilean boost}
\end{rem}

\index{curvature!space of tensors}
\subsection{The Space of Pointwise Curvature Tensors} 
Let $\mathcal{R}$ denote
the subspace of $(1,3)$ tensors over a semi-Euclidean vector space $V$
of dimension $n$ satisfying the curvature symmetries (\ref{eq:1}), 
(\ref{eq:2}), (\ref{eq:3}),
and hence (\ref{eq:4}). Let $W = \bigwedge^2V$. 
By lowering the first index we
identify $\mathcal{R}$ with the symmetric tensors of degree $2$ over $W$
satisfying (\ref{eq:3}). The dimension of $W$ is $N = n(n-1)/2$, so that
$\dim(\mathcal{S}^2(W)) = N(N+1)/2$; here $\mathcal{S}$ denotes the space of
symmetric tensors. In terms of a basis we get an independent linear
restriction from the cyclic symmetry (\ref{eq:3}) for each choice of $4$ distinct
indices $i<j<h<k$. Hence,
$$\dim\mathcal{R} = N(N+1)/2 - \binom{n}{4} = n^2(n^2 - 1)/12.$$

\subsection{The Ricci Tensor} \index{Ricci tensor}
For $R \in \mathcal{R}$, for all $v, w \in V$, we
consider the linear map $A_{vw}:V \to V, x \to R_{vx}w$. The trace
gives us a bilinear form $tr A_{vw} = Ric(v,w)$ on $V\times V$ called
the Ricci tensor of $R$. In terms of components $R_{ihk}^j$ with
respect to a basis the components of $Ric$ are the contraction
$Ric_{ih} = \sum_j R_{ijh}^j$. With respect to a frame $e_i$ we write
$R_{ijhk} = g(R_{e_ie_j}e_h, e_k)$, which makes the superscript index
correspond to the {\em fourth} index of the covariant form. To take
care of the indefinite case we use the signs $\epsilon_i$, whereupon
$R_{ijhk} = \epsilon_kR_{ijh}^k$, hence $Ric_{ih} = \sum_j \epsilon_j
R_{ijhj} = \sum_j \epsilon_j R_{hjij} = Ric_{hi}$. That is, $Ric$ is a
symmetric bilinear form.

\subsection{Ricci Curvature} \index{Ricci curvature}\index{curvature!Ricci}
The Ricci tensor is determined by the
corresponding quadratic form $v \to Ric(v,v)$. For a unit vector $v$,
$Ric(v,v)g(v,v)$ is called the {\em Ricci curvature} of $v$. We can then
take $v$ to be a frame member $v = e_i$, so that
$$Ric(v,v)g(v,v) = \epsilon_i \sum_j \epsilon_jR_{ijij} =
\sum_{j\ne i}K_{ij},$$
the sum of sectional curvatures of frame sections containing $v$. When
the normal space to $v$ is definite with respect to $g$ we could
average the sectional curvatures of planes containing $v$ over the
whole normal unit sphere, obtaining a standard multiple (depending
only on dimension) of $Ric(v,v)$.

\vspace{\medskipamount}
The geometric content of the Ricci curvature is that it measures the
acceleration of volume contraction with respect to the flow along
certain fields of geodesics. We formulate this precisely as follows.

\begin{defn} The {\em divergence} of a vector field $V$ is
a concept which depends only on an unsigned volume element, that is, a
{\em density} $\mu$. If $\Omega$ is an $n$-form such that locally $\mu
= |\Omega|$, then $div \,V$ is defined by taking the Lie derivative of
$\Omega$:
$$L_V\Omega = (div \, V)\Omega.$$
\end{defn}
For any given unit vector $v$ we define a canonical unit vector field
extension $V$, by forcing $V$ to satisfy
\begin{enumerate}
\item The integral curves of $V$ are unit speed geodesics.
\item If $v^\perp$ is the hyperplane in $M_p$
orthogonal to $v$, then $V$ is orthogonal to the hypersurface
$\exp_pv^\perp$.
\end{enumerate}
\begin{thm}[Theorem on Volume and Ricci Curvature] \index{divergence}
The divergence of the canonical
extension $V$ of $v$ is $0$ at $p$; its derivative in the direction of
$v$ is
$$v\,div\,V = - Ric(v,v).$$
\end{thm}\index{curvature!Ricci}

\begin{prob}\label{prob33}[The Divergence Theorem.] Suppose that $M$ is
compact and orientable. Prove that for any smooth vector field $V$
$$\int_M (div \,V)\Omega = 0.$$\end{prob}\index{orientable}

\subsection{Scalar Curvature} \index{scalar curvature}\index{curvature!scalar}
If we contract the curvature a second time, we get the {\em scalar curvature}
$$S = \sum_{i,j}R_{ij}^{ij} = \sum_{i\ne j} K(\sigma_{ij}) = tr\,Ric.$$
It is twice the sum of all the frame sectional curvatures, independent
of the choice of frame. In Riemannian geometry it gives the discrepancy
of the measure of a sphere or ball from the Euclidean value:
$$\mu_{n-1}(S(p,r)) = r^{n-1}\Omega_{n-1}(1-c_nS(p)r^2 + \cdots),$$
$$\mu_n(B(p,r)) = \int_0^r \mu_{n-1}(S(p,t))\,dt = 
\frac{1}{n}r^n\Omega_{n-1}(1-\frac{n}{n+2}c_nS(p)r^2+ \cdots),$$
where $\Omega_{n-1}$ is the $n-1$-dimensional measure of the
Euclidean unit sphere in $E^n$ and $c_n$ is another constant
depending only on $n$. The argument to prove these approximations is
based on the fact that $S$ is the only linear scalar invariant of $R$,
as well as expressions for the second-order terms of the metric in
normal coordinates which we develop later. (Cartan, Le\c cons sur la
G\'eom\'etrie des Espaces de Riemann, discusses this, pp 255-256.)
\index{Cartan, E.}

\subsection{Decomposition of $\mathcal{R}$} \index{curvature!decomposition of tensors}
Tensor spaces over an inner product
space are naturally inner product spaces, with the induced action of
$O(n,\nu)$ leaving the inner product invariant. That is, if $e_i$ is a
frame of $V$, then we get a frame for the tensor space by forming all
the products of the $e_i$. When the tensor space involves the
Grassmann algebra over $V$, then it is customary to use the
determinant inner product, so that for $e_i\wedge e_j = e_i\otimes e_j
- e_j\otimes e_i$ (in conformity with the shuffle-permutation
definition of $\wedge$), we have $g(e_i\wedge e_j, e_i\wedge e_j) =
g(e_i,e_i)g(e_j,e_j) - g(e_i,e_j)^2$, not twice that.

\vspace{\medskipamount}

The operations of raising and lowering indices and contractions are
equivariant under the action of $O(n,\nu)$. Thus, the kernel of the
Ricci contraction $C_{Ric}:\mathcal{R} \to \mathcal{S}^2$, $R \to Ric$, is a
subspace $ker\,C_{Ric} = \mathcal{W} \subset \mathcal{R}$ invariant under
$O(n,\nu)$. The orthogonal projection of $\mathcal{R}$ onto $\mathcal{W}$ gives
us the {\em Weyl conformal curvature tensor}. It is easy to show that
$C_{Ric}$ is {\em onto}, so that we have an orthogonal direct sum
$\mathcal{R} = \mathcal{W} \oplus \mathcal{S}^2$.\index{curvature!Weyl tensor}
\index{Weyl tensor}\index{curvature!conformal tensor}

\vspace{\medskipamount}
\index{curvature!Ricci}
The second contraction to get scalar curvature, $tr:\mathcal{S}^2 \to 
\R$ splits the Ricci curvature summand further into a
constant-curvature tensor and a trace-free Ricci tensor in $\mathcal{S}_0^2
= ker\,tr$. Thus, we always have an orthogonal direct sum
decomposition
$$\mathcal{R} = \mathcal{W} \oplus \R \oplus \mathcal{S}_0^2.$$
This decomposition is irreducible under $SO(n,\nu)$ except when $n =
4$. Then there is a Hodge $\ast$-operator $\ast:\bigwedge^2V \to
\bigwedge^2V$ satisfying $\ast\circ\ast = I$, which has a natural
\index{Hodge star operator}
isometric extension to $\mathcal{R}$, and consequently gives an a further
splitting of $\mathcal{R}$ into the $+1$ and $-1$ eigenspaces of $\ast$,
invariant under the orientation-preserving maps $SO(n,\nu)$. Curvature
tensors which are eigenvalues of $\ast$ are called self-dual and
anti-self-dual; they have become very important in recent years
because the Yang-Mills extremals are just the connections with
self-dual or anti-self-dual curvature tensor, and there were surprising
relations with the multitude of differentiable structures on
$4$-manifolds.\index{Yang-Mills extremals}

\vspace{\smallskipamount}
\begin{prob}\label{prob34} (a) Calculate the dimensions of the summands in 
the splitting of $\mathcal{R}$. (b) For self-adjoint linear transformations
$A,B:M_p \to M_p$ we have seen in Problem \ref{prob31} that
$A\wedge A \in \mathcal{R}$, hence $A \wedge B + B \wedge A = (A+B) \wedge (A+B) - A \wedge A -
B \wedge B \in \mathcal{R}$. Taking $B = I$, the identity transformation,
show that
$$C_{Ric}(A \wedge I + I \wedge A) = (n-2)A + (tr \,A)I.$$
(c) Hence,
$$C_{Ric}(\frac{1}{n-2}(A \wedge I + I \wedge A - \frac{1}{n-1}(tr\,A)I
\wedge I)) = A,$$
and
$$W^\perp(A) = \frac{1}{n-2}(A \wedge I + I \wedge A -
\frac{1}{n-1}(tr\,A)I \wedge I)$$
gives a monomorphism $W^\perp:\mathcal{S}^2 \to \mathcal{R}$ equivariant
under $O(n,\nu)$.\end{prob}

\vspace{\smallskipamount}

From the problem we conclude that the Weyl conformal curvature tensor
is given by
$$W(R) = R - W^\perp(Ric).$$

\subsection{Normal Coordinate Taylor Series} It was Riemann who invented
Riemannian geometry and defined normal coordinates, calculated the
second-order expressions for the metric coefficients. Maybe that's how
he discovered the Riemann tensor (i.e., the curvature tensor); but
anyway he knew that those second-order coefficients are given in terms
of curvature components.\index{Riemann tensor}\index{Riemann, G.B.}
\index{normal coordinates!Taylor expansion}

\begin{thm}[Riemann's Normal Coordinate Expansion of Metric
Coefficients]
The second-order Taylor expansion of the metric in terms of normal
coordinates $x^i$ is give in terms of components of $R$ as follows:\par
$$g_{ij} = g(\frac{\partial}{\partial x^i},\frac{\partial}{\partial x^j})
= \delta_{ij} - \frac{1}{3}\sum_{h,k}R_{ihjk}x^hx^k + \cdots$$
\end{thm}

An interesting feature is how sparse these quadratic parts are, as
well as the fact that there are no linear terms. The components of $R$
are supposed to be evaluated at the origin, with respect to the frame
which defines the normal coordinates. There is probably a version for
semi-Riemannian geometry, but as it stands the formula given is only
for the Riemannian case. We can now continue by obtaining the
Christoffel symbols of the normal coordinate vector field basis,
stopping at the linear terms. It is quite easy to use the Koszul
formula for this, since the inverse of the matrix $g_{ij}$ which is
needed is just the identity matrix, up to second-order terms.

\subsection{The Christoffel Symbols} \index{Christoffel symbols}
The Christoffel symbols for the
coordinate vector field local basis have as the linear terms of their
Taylor expansions
$$\Gamma_{jk}^i = \frac{1}{3}(R_{jikh} + R_{kijh})x^h + \cdots.$$

On a normal coordinate neighborhood we define a canonical frame field
$(E_i)$ as follows. At the origin $p$ the frame coincides with the
coordinate basis $\partial_i = \frac{\partial}{\partial x^i}$. Then we
obtain $E_i$ at other points by parallel translation along the radial
geodesics from $p$. Thus, $D_{x^i\partial_i}E_j = 0$. These canonical
frames are vital in the proof of the Cartan Local Isometry Theorem.
\index{Cartan Local Isometry Theorem}

\vspace{\smallskipamount}
\begin{prob}\label{prob35} Show that the second order Taylor expansion of 
the canonical frame is given by
$$E_i = \sum_h(\delta_i^h + \frac{1}{6}\sum_{j,k}\,R_{ijhk}x^jx^k +
\cdots)\,\partial_h,$$
the coframe by\index{coframe}
$$\omega^i = \sum_h(\delta_h^i - \frac{1}{6}\sum_{j,k}\,R_{ijhk}x^jx^k
+ \cdots)\,dx^h.$$
Moreover, the connection forms of this frame field are
$$\varphi_j^i = \frac{1}{4}\sum_{h,k}\,R_{ijhk}(x^hdx^k - x^kdx^h) +
\cdots .$$\end{prob}

\vspace{\smallskipamount}
\begin{prob}\label{prob36} Explain why for each $h,k$, along every radial
geodesic $x^h\partial_k - x^k\partial_h$ is a  field.\end{prob}

\vspace{\smallskipamount}
We can prove Riemann's theorem by doing the $2$-dimensional case first.
Let $ds^2 = E\,dx^2 + 2F\,dx\,dy + G\,dy^2$ be the metric in terms of
Riemannian normal coordinates $x,y$. Then the result to be proved is
that the second order Taylor expansions are
$$E = 1 - \frac{1}{3}Ky^2 + \cdots,$$
$$F = \frac{1}{3}Kxy + \cdots,$$
$$G = 1 - \frac{1}{3}Kx^2 + \cdots,$$
where $K$ is the Gaussian curvature at the origin $p$. Once we have
the 2-dimensional case, we can prove the general case by a multitude
of polarizations; that is, we apply the $2$-dimensional case to surfaces
obtained by exponentiating tangent planes spanned by $(\partial_i,
\partial_j)$, $(\partial_i, \partial_j + \partial_k)$, and
$(\partial_i + \partial_h, \partial_j + \partial_k)$. To make this
approach work we need to know that the Gaussian curvature of these
surfaces is the sectional curvature of their tangent planes, and that
the normal coordinates of these surfaces are the restrictions to the
surfaces of the appropriate normal coordinates of the space.

\vspace{\medskipamount}
In the Riemannian case, the fact that a geodesic of the ambient space
which lies in a submanifold is a geodesic of the induced metric on
that submanifold follows immediately from the characterization of
geodesics as locally length-minimizing curves. However, we also would
like to have the results in the semi-Riemannian case, so a more
computational proof that geodesics are inherited is in order. In fact,
we need to know how the Levi-Civita connections are related.

\begin{thm}[Theorem on the Connection of an Isometric Imbedding]
\index{connection!isometric imbedding}
Let $N \subset M$ be a submanifold such that the induced metric
from $M$ on $N$ is nondegenerate. (This is automatic in the
Riemannian case.) Then the Levi-Civita connection of $N$ is given
by projecting the covariant derivatives for the Levi-Civita connection
of $M$ orthogonally onto the tangent planes of $N$. That is, if $\Delta$
is the connection of $N$, $D$ is the connection of $M$, and
$\Pi:M_p \to N_p$ is the orthogonal projection for each $p \in N$,
then for vector fields $X,Y$ on $N$ 
$$\Delta_XY = \Pi D_XY.$$
\end{thm}
The theorem can be proved by a straightforward verification of the
axioms and characteristic properties for $\Delta$ (connection axioms,
torsion $0$, metric). It can also be done by using {\em adapted frame
fields}: \index{adapted frame field}
these are frames of $M$ at points of $N$ such that the first
$q = \dim\,N$ of the fields are a frame for $N$. Locally we can
always extend adapted frame fields to a frame field of $M$ in a
neighborhood. Then the local coframe $\omega^i$, when restricted to
$N$, satisfies
$$\omega^1, \cdots, \omega^q$$
is a coframe of $N$, and
$$\omega^{q+1} = 0, \cdots, \omega^n = 0$$
on $N$. Now consider the connection forms $\varphi_j^i$ of $\omega^i$.
The first $q \times q$ block is skew-adjoint with respect to the
semi-Euclidean metric induced on the tangent spaces of $N$ by
$(\omega^1,\cdots,\omega^q)$. The restriction of the first structural
equation of $M$ to $N$ then gives
$$\omega^i = - \sum_{j=1}^q\,\varphi_j^i \wedge \omega^j,$$
which shows that the $q \times q$ block is that connection form of $N$
and that torsion is $0$. Then we have that orthogonal projection of $Z
\in M_p$ to $N_p$ is the vector with components $\omega^i(Z),\ i =
1,\dots,q$, and hence
$$\omega^i(\Delta_XY) = X\omega^i(Y) +
\sum_{j=1}^q\,\varphi_j^i(X)\omega^j(Y)$$
$$\qquad = \omega^i(D_XY),\ i = 1,\dots,q.$$
This verifies the formula $\Delta_XY = \Pi D_XY$.
\begin{cor} A geodesic of $N$ is a curve $\gamma$ in $N$
such that the $M$-acceleration $D_{\gamma'}\gamma'$ is always
orthogonal to $N$.
\end{cor}

This suggests a numerical algorithm for approximate geodesics on a
surface $N$ in $3$-space. Given a velocity tangent to $N$, move a small
distance on $N$ in the direction of that velocity. Then rotate the
velocity in the normal plane to $N$ at the new point so that the
velocity will be tangent to $N$ again. Then repeat the procedure. It
is not hard to show that as the ``small distance'' goes to $0$, the
sequence of points generated converges to a geodesic.

\begin{cor}[Riemann's Theorem--Pointwise
Realizability of Curvature Tensors] 
\index{curvature!pointwise realization}If we are given components
$R_{ijhk}$ satisfying the symmetries (\ref{eq:1}), (\ref{eq:2}), (\ref{eq:3}), then there exists
a metric on a neighborhood of $O \in \R ^n$ such that these
$R_{ijhk}$ are the components of the curvature tensor at the origin.
\end{cor}
\begin{cor}[to the Proof of Riemann's Theorem] \index{Riemann's Theorem}
Sectional curvatures determine the curvature tensor.
\end{cor}

Of course, the fact that the sectional curvatures determine the
curvature tensor can be proved easily directly. See Bishop and
Crittenden, Corollary 2, p 164, and the explicit polarization formula
of Problem 2, p 165.\index{polarization}\index{Crittenden, R.J.}

\vspace{\medskipamount}

We give an outline of the proof of the $2$-dimensional case of Riemann's
theorem. We know that the vector field $x\partial_x + y\partial_y$ is
radial, so that the normal coordinate parametrization, which is a
numerical realization of the exponential map, preserves its
square-length $x^2+y^2$. Thus, we obtain an identity:
\begin{equation}\label{eq:aa}
x^2E + 2xyF + y^2G = x^2 + y^2.
\end{equation}

By Problem \ref{prob36} we know that $y\partial_x - x\partial_y$ is a normal
Jacobi field along every radial geodesic. We do not need the Jacobi
equation, but the fact that the two vector fields are orthogonal
gives:
\begin{equation}\label{eq:c}
xy(E-G) + (y^2-x^2)F = 0.
\end{equation}
The classical formulas for the Christoffel symbols, which can be
calculated from the Koszul formulas by plugging in coordinate vector
fields give us equations expressing the first derivatives of $E, F, G$
in terms of linear expressions in those Christoffel symbols. (The
coefficients of those linear equations are each $E, F$, or $G$.) But
we have seen that the Christoffel symbols all vanish at the origin.
Hence, $E, F, G$ have linear coefficients $0$; of course, they have
constant terms $1$, $0$, $1$, respectively. Now match up the fourth order
terms in (\ref{eq:aa}) and (\ref{eq:c}). The result is that at the origin all of the
second derivatives vanish except three, which are themselves equal as
follows:
$$E_{yy} = -2F_{xy} = G_{xx}.$$
Finally, by the classical equations for the curvature tensor in terms
of the Christoffel symbols we calculate that the Gaussian curvature at
the origin is $K = -3E_{yy}/2$.

\vspace{\smallskipamount}
\begin{prob}\label{prob37} For the indefinite case, for which $E(0,0) = -1$
and $G(0,0) = 1$, show that the proof outlined still goes through with
only some sign changes.\end{prob}

\section{Conjugate Points}\index{conjugate point}
Let $p \in M$ and let $\gamma$ be a geodesic
through $p$. A point $q$ on $\gamma$, $q \ne p$, is a {\em conjugate
point of p along} $\gamma$ if there exists a Jacobi field $J \ne 0$ on
$\gamma$ such that $J$ vanishes at both $p$ and $q$. We could
equivalently define a conjugate point $q$ of $p$ to be a singular
value of $\exp_p$, for if we had $v\in (M_p)_x$, $v \ne 0$, such that
${\exp_p}_*v = 0$, then $J(s) = {\exp_p}_*sv$ defines a Jacobi field
along the geodesic $s \to \exp_psx$ such that $J(0) = 0$ and $J(1) =
0$. (To interpret $sv \in (M_p)_{sx}$ we use the canonical isomorphism
of $M_p$ with its tangent spaces.)

\vspace{\smallskipamount}

Both viewpoints of conjugate points are important. The definition as
given shows that the relation ``is a conjugate point of'' is
symmetric, and on the other hand, by applying Sard's Theorem to
$\exp_p$ we obtain that the set of all conjugate points of $p$
forms a set of measure $0$. \index{Sard's Theorem}\index{Sard, Arthur}
Moreover, the singular value
property gives us the following theorem, which is half of the property
of local nonminimization beyond a first conjugate point $q$ of $p$
along $\gamma$. In a more general setting, such as a locally compact
intrinsic metric space, we would take this geometrically significant
\index{metric space!intrinsic}property as the {\em definition} of a conjugate point.
\begin{thm}[Local Nonminimization Implies Conjugate Point]
Let $\gamma([0,s])$ be a nonselfintersecting geodesic segment such
that for every neighborhood $U$ of $\gamma([0,s])$ there is $s' > s$
such that there is a shorter curve in $U$ from $\gamma(0)$ to
$\gamma(s')$ than $\gamma([0,s'])$. Then $\gamma([0,s])$ has a
conjugate point of $\gamma(0)$.
\end{thm}
\begin{proof} Suppose there are no conjugate points of $p = \gamma(0)$
on $\gamma([0,s])$. Then $\exp_p$ is a diffeomorphism of some
neighborhood $V$ of $s\gamma'(0)$ onto a neighborhood of $\gamma(s)$.
Moreover, we may assume that $V$ is so small that $\exp_p$ is also a
diffeomorphism on the ``cone'' $\cup\{tV : 0<t<1\} = W$. We can also
take a central normal ball $B \subset M_p$ so small that $\exp_pB$
intersects $\gamma([0,s])$ only in an initial radial segment of
$\gamma$. Let $X = B \cup W$. Then $\exp_p: X \to \exp_pX = U$ is a
diffeomorphism onto a neighborhood of $\gamma([0,s])$ which is thereby
filled with a field of radial geodesics uniquely connecting points to
$p$. Now we use the proof of the minimization theorem, page 25, to
show that if $\gamma(s') \in U$, then $\gamma([0,s'])$ is the shortest
curve in $U$ from $p$ to $\gamma(s')$.
\end{proof}
\begin{rem} To turn the local nonminimizing property into a
characterization of first conjugate points it is necessary to free
ourselves of the restriction that $\gamma([0,s])$ be
nonselfintersecting. We do this by taking the neighborhoods $U$ in
question to be a neighborhood in the space of rectifiable curves
starting at $p$. To topologize the space of curves we parametrize them
with constant speed on $[0,1]$ and use the topology of uniform
convergence. We construct $W$ as we did above, and take $B$ to be any
central normal ball, so that for some $\epsilon > 0$, $\exp_p$ is a
diffeomorphism of some neighborhood of $t\gamma'(0)$ within $X$ onto
the $\epsilon$-ball centered at $\gamma(t)$ for every $t, 0 < t < 1$.
Then any rectifiable curve which is uniformly $\epsilon$-close to
$\gamma$ can be lifted uniquely to a curve in $X$. This means that we
can define a unique radial geodesic to each point of the curve, so
that we can carry out the comparison of lengths as in the proof of
Theorem \ref{thm:firstvar}.
\end{rem}

\index{second variation}
\subsection{Second Variation} We have seen that when we vary a geodesic to
nearby curves with the same endpoints, then the first derivative of
energy or arclength is $0$. Moreover, if there is no conjugate point,
then the critical value of energy or length actually is a local
minimum. As in calculus, to gain information about what is happening
at a critical point we need to look at second derivatives. We will
only do the simple case for which the endpoints are fixed; the more
general case in which we study the distance to a submanifold would
require some notions that we have not defined, {\em focal points} and
\index{focal point} the {\em second fundamental form of a submanifold}. 
\index{second fundamental form}The details of the
more general case can be found in Bishop \& Crittenden or other
references.

\vspace{\smallskipamount}

We assume that the geodesic base curve and the longitudinal curves of
the variations we consider are normalized in their parametrization, so
\index{parametrization!normalized}
that they have constant speed and are parametrized on $[0,1]$.
Intuitively it is clear that we don't lose anything by such a
restriction. It makes energy equal to the square of the length among
the curves under consideration, so that there is no difference which
we choose to calculate with. Moreover, the restriction is fitting to
the conclusion we want to draw: by showing that some second derivative
is negative we conclude that the value is not a minimum. Thus, it is
some particular special variations that we want to discover from a
calculation of second variation, and great generality is not required
as long as what we do points to the special variations needed.

\vspace{\medskipamount}
So suppose that $Q$ is a smooth rectangle with base geodesic and
constant end transversal curves, and longitudinal parameter $s, 0 \le
s \le 1$. Let $X$ be the longitudinal vector field and $Y$ the
transverse vector field. We will use such facts as $D_XY = D_YX$ and
$R_{XY} = -D_XD_Y + D_YD_X$ without going through some formal
justification. Our starting point is the formula for first variation
of energy:
$$\frac{dE}{dt} = 2 \int_0^1 g(D_XY,X)\,ds.$$
Differentiating with respect to $t$ again:
$$\frac{d^2E}{dt^2} = 2 \int_0^1 [g(D_YD_XY,X) + g(D_XY,D_YX)]ds$$
$$\qquad = 2 \int_0^1[g(D_XD_YY +R_{XY}Y,X) + g(D_XY,D_XY)]ds$$
$$\qquad = 2 \int_0^1[\frac{\partial}{\partial s}g(D_YY,X)
- g(D_YY,D_XX) - g(R_{XY}X,Y) + g(D_XY,D_XY)]ds.$$
Now we evaluate at $t=0$, so that $D_XX = 0$ because the base curve is
a geodesic, and the term which can be integrated drops out because the
end transverse curves are constant: $D_YY = 0$, at $s = 0,1$. For the
term $g(D_XY,D_XY)$ we have an alternative form
$\frac{\partial}{\partial s}g(D_XY,Y) - g(D_XD_XY,Y)$, and again the direct integration
of the derivative drops out because of the fixed end condition. We
also simplify the notation, writing $Y' = D_XY$ and $Y'' = D_XD_XY$
when $t=0$. Thus, we get {\em Synge's formula for second variation}:
\index{Synge's formula!2nd variation}\index{Synge, J.L.}
$$\frac{d^2E}{dt^2}(0) = 2 \int_0^1 [g(Y',Y') - g(R_{XY}X,Y)]ds$$
$$\qquad = -2 \int_0^1 g(Y''+R_{XY}X,Y)ds.$$
There are several interesting features to note. The result only
depends on the variation vector field $Y$ along the base, not on other
properties of $Q$. This is an illustration of the rule that when a
first derivative vanishes, then the second derivative becomes
tensorial. The formula is a quadratic form in $Y$, so that it makes
sense to consider the corresponding bilinear form, which is called the
{\em index form for variations of the base curve}. The formulas for
the index form are rather transparent; just replace one of the $Y$'s
in each term by a $Z$ to get $I(Y,Z)$.\index{index form}

\vspace{\smallskipamount}
From the second expression we see that if we can take $Y$ to be a
Jacobi field, so that the two ends of the base geodesic are conjugate
points of each other, then the second variation vanishes. In fact,
$I(Y,Z) = 0$ for every $Z$, which tells us that the Jacobi field is in
the nullspace of the index form. \index{nullspace!index form}
Conversely, if $I(Y,Z)=0$ for every
$Z$, then we must have $Y''+R_{XY}X = 0$, that is $Y$ is a Jacobi field.
Hence we have identified the nullspace of $I$: it consists of Jacobi
fields which vanish at both ends.

\begin{thm}[Nonminimization Beyond Conjugate Point] If a
geodesic segment $\gamma$ has an interior conjugate point, then it is
not minimal.\index{nonminimization!beyond conjugate point}
\end{thm}

\begin{proof} Make the parametrization be such that $\gamma(1)$ is a
conjugate point of $\gamma(0)$, and $\gamma$ extends beyond to some
$\gamma(s), s>1$. Using the nontrivial Jacobi field which vanishes at
$\gamma(0)$ and $\gamma(1)$ as a variation field, we obtain a
variation of that first part of $\gamma$ for which the change in
arclength is $O(t^3)$. These varying curves form angles at $\gamma(1)$
with $\gamma$ looking backwards proportional to $t$, up to higher
order. If we cut across the obtuse angles formed, the saving in length
is on the order of $t^2$, so that saving is greater in magnitude than
the gain in length $O(t^3)$ that we have from the variation. Hence
there must be shorter curves nearby $\gamma$ connecting $\gamma(0)$
and $\gamma(s)$.
\end{proof}

The claim in the proof that ``the saving in length is on the order of
$t^2$'' requires a proof, for which we introduce some second variations
of vector fields which are not $0$ at the ends. For a vector field $Y$
along $\gamma$ we still define an index form
\begin{equation}\label{eq:*}
I(Y) = 2 \int_0^1[g(Y',Y') - g(R_{XY}X,Y)]\,ds.
\end{equation}
To interpret $I(Y)$ as a second derivative of the energy of some
rectangle we note that in leading up to (\ref{eq:*}) when we had a fixed-end
variation, we integrated terms
$\frac{\partial}{\partial s}g(X,D_YY)$. Thus, the weaker condition
$D_YY =0$ at $s = 0$ and $1$
would also suffice to derive (\ref{eq:*}). For any given $Y$ this condition
can be attained by letting the initial and final transverse curves be
geodesics. In fact, that was how we showed the existence of a
variation attached to a given variation field anyway.

\vspace{\smallskipamount}
Moreover, we also want to interpret (\ref{eq:*}) as a second derivative in
the case where $Y$ is only continuous and piecewise $C^1$. We can do
so if we require that at the finite number of places where $Y' = D_XY$
does not exist we have $D_YY = 0$; in particular, it is again enough
to make the transverse curves at those places be geodesics.

\begin{thm}[The Basic Inequality] Suppose that there are no
\index{basic inequality}
conjugate points of $\gamma(0)$ on $\gamma((0,1])$. Let $V$ be a
piecewise $C^1$ vector field on $\gamma$ orthogonal to $\gamma$ such
that $V(0) = 0$. Let $Y$ be the unique Jacobi field such that $Y(0) =
0$ and $Y(1) = V(1)$. Then $I(V) \ge I(Y)$ and equality occurs only if
$V = Y$.
\end{thm}
\begin{proof} The inequality $I(V) \ge I(Y)$ is actually an immediate
consequence of the minimization of energy by geodesics: if we take a
rectangle $Q$ such that $V = \frac{\partial Q}{\partial t}$ and the
final transversal $t \to Q(1,t)$ is a geodesic, then we can define
another rectangle $\hat Q$ by making $\hat Q_t:s \to \hat Q(s,t)$ be
the geodesic from $\gamma(0)$ to $Q(1,t)$. Then $E(\hat Q_t) \le
E(Q_t)$, so also $I(Y) \le I(V)$.

\vspace{\smallskipamount}
To show that the case of equality requires $V=Y$ we make a calculation
involving a basis of $Y_1,\ldots,Y_{n-1}$ of the Jacobi fields which
are $0$ at $s = 0$ and perpendicular to $\gamma$ at $s = 1$. The we can
write $V = \sum f_iY_i$, where the coefficients $f_i$ are piecewise
$C^1$. We let
$$V' = \sum f_i'Y_i + \sum f_iY_i' = W + Z.$$

\begin{lem}\label{lem:1} If $Y$ and $Z$ are Jacobi fields along a
geodesic $\gamma$, then $g(Y',Z) - g(Y,Z')$ is constant. If $Y$ and
$Z$ vanish at the same point, then the constant is $0$, so that $g(Y',Z)
= g(Y,Z')$.
\end{lem}

\begin{proof} Let $X = \gamma'$. Differentiating the difference we have
$$g(Y'', Z) + g(Y',Z') - g(Y',Z') - g(Y,Z'') = -g(R_{XY}X,Z) + 0 +
g(Y,R_{XZ}X) = 0,$$
by the symmetry in pairs of $R$.
\end{proof}

Now we have
\begin{equation}\label{eq:a}
g(V, \sum f_j'Y_j') = \sum f_if_j'g(Y_i,Y_j') 
\end{equation}
$$\qquad = \sum f_if_j'g(Y_i',Y_j),\qquad\text{by Lemma \ref{lem:1},}$$
$$\qquad = g(W,Z).$$
\begin{equation}\label{eq:b}
g(V,Z)' = g(V',Z) + g(V,Z') 
\end{equation}
$$\qquad = g(W,Z) + g(Z,Z) + g(V, \sum f_i'Y_i') + g(V, \sum f_iY_i'')$$
$$\qquad = g(W,Z) + g(Z,Z) + g(W,Z) - g(V, \sum f_iR_{XY_i}X)$$
$$\qquad = 2 g(W,Z) + g(Z,Z) - g(V,R_{XV}X).$$
The integrand in $I(V)$ is thus
$$g(V',V') - g(R_{XV}X,V) = g(W+Z,W+Z) - g(R_{XV}V,V)$$
$$\qquad\qquad = g(W,W) + g(V,Z)', \qquad\text{by (\ref{eq:b}),}$$
Note that $V$ and $Z$ are continuous, piecewise $C^1$, so that we can
integrate to get
\begin{equation}\label{eq:d}
I(V) = \int_0^1g(W,W)\,ds + g(V(1),Z(1)) - g(V(0),Z(0)) 
\end{equation}
$$\qquad= \int_0^1g(W,W)\,ds + g(V(1),Z(1)).$$
We make the same calculation with $Y = \sum f_i(1)Y_i$, wherein the
field corresponding to $W$ is $0$ and $V(1) = Y(1)$,
$Z(1) = \sum f_i(1)Y_i'(1)$ are the same. Thus,
$$I(V) - I(Y) = \int_0^1g(W,W)\,ds \ge 0,$$
and the condition for equality is that the piecewise continuous field
$W = 0$, hence the $f_i$ are constant and $V=Y$. This completes the
proof of the Basic Inequality.
\end{proof}

Now we can give a precise version of the proof that geodesics do not minimize beyond conjugate points.

\vspace{\smallskipamount}
Let $q = \gamma(s_2)$ be a conjugate point of $p = \gamma(0)$, and let
$Y$ be a nonzero Jacobi field vanishing at $p$ and $q$. If we extend
$Y$ by a segment of 0-vectors beyond $q$, we will still have $I(Y) =
0$. Let $\epsilon$ be a positive number such that there is no
conjugate point of $\gamma(s_2-\epsilon)$ on
$\gamma((s_2-\epsilon,s_2+\epsilon])$. Reparametrize $\gamma$ (and
scale $\epsilon, s_2$ correspondingly) so that $s_2+\epsilon = 1,
s_2-\epsilon = s_1$. We will indicate vector fields along $\gamma$ by a
triple designating what they are on each subinterval $[0,s_1],
[s_1,s_2], [s_2,1]$. For example, $(Y,Y,0)$ denotes the vector field
which is $Y$ on $[0,s_1]$, $Y$ on $[s_1,s_2]$, and $0$ on $[s_2,1]$.
 As long as we force the tranverse curves at the
break points $s_1,s_2$ to be geodesics, the summing of index forms for
subintervals will represent a sum of second derivatives of energies
which gives the second derivative of the whole. Let $W$ be the Jacobi
field such that $W(s_1) = Y(s_1)$ and $W(1) = 0$. Then we satisfy the
condition for the basic inequality on the interval $[s_1,1]$, so that
$$0 = I((Y,Y,0)) > I((Y,W,W)).$$
Since there is a vector field $(Y,W,W)$ having negative second
variation, the longitudinal curves of any rectangle fitting it are
shorter than $\gamma([0,1])$ on the order of $t^2$. This completes
the proof of the nonminimization beyond conjugate points.

\begin{rem} In Problem \ref{prob30} the condition that the field be a Jacobi
field along every geodesic should be derived from the fact that it is
a Killing field independently of the other parts of the problem.
\index{Killing field}
\end{rem}

\vspace{\medskipamount}
\begin{examp} The following kind of metric comes up in classical
mechanics because the space of positions of a rigid body rotating
about a fixed point is $SO(3)$, and the kinetic energy of the rigid
body can be viewed as a Riemannian metric having the left-invariance
considered. A general theorem of mechanics of a freely moving
conservative system \index{conservative system}says that the free, 
unforced motions are geodesics
of the kinetic energy metric. See Bishop \& Goldberg, chapter 6.

On $SO(3)$ and $S^3$ we have parallelizations by left invariant
vector fields $X_1, X_2, X_3$ satisfying 
$$[X_1,X_2] = X_3,\ [X_2,X_3] = X_1,\ [X_3,X_1] = X_2.$$
The dual basis $\omega^1, \omega^2, \omega^3$ satisfies the
Maurer-Cartan equations\index{Maurer-Cartan equations}
$$d\omega^1 = -\omega^2\wedge\omega^3, d\omega^2 =
-\omega^3\wedge\omega^1, d\omega^3 = -\omega^1\wedge\omega^2.$$
We introduce a left-invariant metric for which the $X_i$ are
orthogonal and have constant lengths $a,b,c$:
$$ds^2 = a^2(\omega^2)^2 + b^2(\omega^2)^2 + c^2(\omega^3)^2$$
$$\qquad = (\eta^1)^2 + (\eta^2)^2 + (\eta^3)^2.$$
Hence $d\eta^1 = -a\omega^2\wedge\omega^3 =
-\frac{a}{bc}\eta^2\wedge\eta^3 = -\varphi_2^1\wedge\eta^2 -
\varphi_3^1\wedge\eta^3$, where the $\varphi_j^i$ are the Levi-Civita
connection forms for our metric. By Cartan's Lemma, $\varphi_2^1$ and
$\varphi_3^1$ \index{Cartan's Lemma}are linear combinations of 
$\eta^2$ and $\eta^3$, hence
have no $\eta^1$ term. By cyclic permutations we conclude that each
connection form is a multiple of a single basis element:
$$\varphi_2^1 = -C\eta^3, \varphi_3^2 = -A\eta^1, \varphi_1^3 =
-B\eta^3.$$
Then the first structural equations give us what $A,B,C$ are, using a
little linear algebra:
$$A = \frac{b^2+c^2-a^2}{2abc},$$
and the cyclic permutations. We continue by calculating the curvature
by using the second structural equations.
$$\Phi_2^1 = (AC+BC-AB)\eta^1\wedge\eta^2 =
K_{12}\eta^1\wedge\eta^2,$$
and again cyclicly. Thus, the curvature as a symmetric operator on
bivectors is diagonal with respect to the assumed basis. We reduce the
numerator to a sum and difference of squares to see whether the
curvatures can be negative:
$$K_{12} = \frac{3(u-v)^2 + (u+v)^2 - (3w-u-v)^2}{12uvw},$$
where $u=a^2,v=b^2,w=c^2$.
\end{examp}
\subsection{Loops and Closed Geodesics} [This material is from Bishop \&
Crittenden, Section 11.7, with little change.] 
\index{Crittenden, R.J.}A {\em closed
geodesic}, or {\em geodesic loop} is a geodesic segment for which the
initial and final points coincide. A {\em smooth} closed geodesic, or
{\em periodic} geodesic is a geodesic loop for which the initial and
final tangents coincide.\index{geodesic!closed}\index{geodesic!loop}

In a compact Riemannian manifold we can get convergent subsequences of
a family of constant speed curves uniformly bounded in length, using
the Arzel\`a-Ascoli theorem. \index{Arzel\`a-Ascoli Theorem}
If we have a continuous loop with base
point $p$, we can obtain a rectifiable curve in the same homotopy class
by replacing uniformly confined subsegments (which exist by uniform
continuity) by unique minimal geodesic segments. Then we can
reparametrize the resulting piecewise smooth curve by its
constant-speed representative (all parametrized on $[0,1]$). Applying
the Arzel\`a-Ascoli Theorem to a sequence of such homotopic
constant-speed curves with length descending to the infimum length
produces a loop of {\em minimum} length in the homotopy class. It is
obviously a geodesic loop. More generally, the same procedure works in
a {\em complete} Riemannian manifold. (Completeness is next on the agenda.)
\begin{prop} In a compact Riemannian manifold each pointed
homotopy class of loops contains a minimal geodesic loop.
\end{prop}\index{homotopy class}

If we allow the base point to float during homotopies, then we get
{\em free} homotopy classes of loops. If we connect a loop back and
forth to a base point by a curve, then the resulting pointed loops, as
we vary the connecting curve, give conjugate elements of the
fundamental group. In this way we associate uniquely to a free
homotopy class of loops a conjugate class in the fundamental group.
\index{fundamental group}\index{conjugate homotopy class}

\vspace{\smallskipamount}
Using the same trick, in a compact manifold we can extract a periodic
\index{geodesic!periodic}\index{homotopy class!free}
geodesic representative of a free homotopy class of loops. However, in
the complete, but noncompact, case the trick does not necessarily work
because the length-decreasing sequence can go off to infinity. For
pointed loops that does not happen because closed bounded sets are
compact, and everything takes place in a closed ball about the base
point.
\begin{thm}[Synge's Theorem on Simple Connectedness] If $M$ is
compact, orientable, even-dimensional, and has positive sectional
curvatures, then $M$ is simply connected.\index{Synge's Theorem}
\index{simply connected}\index{curvature!positive}
\end{thm}
\begin{proof} The idea is to use second variation to show that a
nontrivial periodic geodesic cannot be minimal in its homotopy class.
By parallel translation once around such a periodic geodesic, we get a
map of the normal space to the geodesic at the initial point into
itself. Because that map has determinant $1$ (Why?) and the normal space
is odd-dimensional, there must be an eigenvalue equal to $1$, hence a
fixed vector $v$. This fixed vector $v$ generates a parallel field
$V$, joining up smoothly at the ends since $v$ is fixed by parallel
translation. Thus, the second variation of an attached rectangle is
$$I(V) = \int_0^1 [g(V',V') - g(R_{XV}X, V)]\,ds = -\int_0^1
g(R_{XV}X, V)\,ds < 0,$$
so there are nearby homotopic shorter curves.
\end{proof}

\vspace{\medskipamount}
We give some variations on the same theme as problems.

\vspace{\medskipamount}
\begin{prob}\label{prob38} Let $M$ be compact, even-dimensional,
nonorientable, and have positive curvature. Show that the fundamental
group of $M$ is $Z_2$.\end{prob}\index{fundamental group}

\vspace{\medskipamount}
\begin{prob}\label{prob39} Let $M$ be compact, odd-dimensional, and have
positive curvature. Show that $M$ is orientable.\end{prob}
\index{orientable}

\vspace{\medskipamount}\index{Klingenberg, W.}
\index{Klingenberg's Theorem}
A method of Klingenberg gives us either geodesic loops or conjugate
points. The hypothesis can be weakened to completeness and an
assumption that there is a cut point. The definition of cut points and
the facts used about them will be given after we discuss completeness.

\begin{thm}[Klingenberg's Theorem] Let $M$ be compact, $p \in M$
and $m$ a point of the cut locus of $p$ which is nearest to $p$. If
$m$ is not a conjugate point of $p$, then there is a unique geodesic
loop based at $p$ through $m$ and having both segments to $m$
minimal.\index{geodesic loop}
\end{thm}
\begin{proof}. If $m$ is not a conjugate point, then there are at least
two minimal segments from $p$ to $m$. We show that there are just two
and that they match smoothly at $m$. Let $\gamma$ and $\sigma$ be any
two. By matching smoothly at $m$ we mean that $\gamma'(1) = -
\sigma'(1)$. Otherwise, there will be local distance functions $f$ and
$g$, giving the distance from $p$ in neighborhoods of $\gamma$ and
$\sigma$, respectively. In a neighborhood of $m$, the equation $f=g$
defines a smooth hypersurface not perpendicular to $\gamma$ or
$\sigma$. (The proof given in B \& C has an error at this point.) In a
direction on this hypersurface making acute angles with both geodesics
there are points which are nearer to $p$ which can be reached by distinct
geodesics, one near $\gamma$, one near $\sigma$. That is, there are
cut points of $p$ closer than $m$.\index{cut point}\index{cut locus}
\end{proof}
\begin{cor} \label{cor:cutpoint}Let $M$ be compact and let $(p,m)$ be a pair
which realizes the minimum distance from a point to its cut locus. Then
either $p$ and $m$ are conjugate to each other, or there is a unique
periodic geodesic through $p$ and $m$ such that both segments are
minimal.\index{minimal locus|see{cut locus}}
\end{cor}

\begin{cor} Let $M$ be compact, even-dimensional,
orientable, with positive curvature, and let $p, m$ be as in corollary
\ref{cor:cutpoint}. Then $p$ and $m$ are conjugate.
\end{cor}\index{orientable}

\section{Completeness}\index{completeness}\index{Cauchy sequence}
\index{metric space!complete}
In a topological metric space $X$, a sequence
$(x_n)$ is a {\em Cauchy sequence} if for every $\epsilon > 0$ there
is $n_0$ such that for all $n > n_0,\ m > n_0$, we have $d(x_n, x_m) <
\epsilon$. A convergent sequence is a Cauchy sequence. If every Cauchy
sequence is convergent, then $X$ is a {\em complete} metric space. If
some subsequence of a Cauchy sequence converges, then the sequence
itself converges. Hence a compact metric space is complete.

\vspace{\medskipamount}
\index{completion}
A noncomplete metric space $X$ has an essentially unique {\em
completion}, a complete metric space in which $X$ is isometrically
imbedded and no unnecessary identifications are made among the points
added to make the space complete or among those and the original
points of $X$. There is a standard way of constructing the completion:
start with the set of all Cauchy sequences; put an equivalence
relation on that set, requiring two sequences to be equivalent if the
sequence which results from interleaving them is still Cauchy; the
metric is extended to the set of equivalence classes by taking the
limit of distances $d(x_n, y_n)$; we imbed $X$ isometrically in this
set of equivalence classes as the classes represented by constant
sequences.

\vspace{\medskipamount}

The completion of a Riemannian manifold does not have to be a
manifold. For example, we can start with the Euclidean plane, make it
noncomplete by puncturing it; then we can take the universal simply
connected covering manifold $X$. As a manifold $X$ is just
diffeomorphic to the plane. However, as a metric space there is only
one equivalence class of Cauchy sequences which does not converge: the
projection of a representative converges to the point we removed. In
terms of Riemann surface theory $X$ is the Riemann surface of the
complex logarithm function; that is, the locally defined log can be
defined as a single-valued complex-analytic function on $X$. So we
call $X$ the {\em logarithmic covering} of the punctured plane.
\index{logarithmic covering}

\vspace{\medskipamount}
\begin{prob}\label{prob40} Discover a topological property of the 
completion of the logarithmic covering which shows that it is not a
manifold.\end{prob}

\vspace{\medskipamount}\index{metric space!intrinsic}
Recall that a metric space is {\em intrinsic} if the distance between
two points is the infimum of lengths of curves between the points. We
say that a metric space is a {\em geodesic} metric space if it
intrinsic and for every pair of points there is a curve between the
points whose length realizes the distance between the points. A space
is {\em locally geodesic} is every point has a neighborhood in which
distances between pairs from that neighborhood are realized by curve
lengths. Thus, we have seen that a Riemannian manifold is a locally
geodesic space.\index{geodesic space}\index{geodesic space!locally}

\vspace{\smallskipamount}

If we parametrize a geodesic in a Riemannian manifold by the arc
length measured from some point on it, then it becomes an isometric
immersion from an interval to the Riemannian manifold. We have
realized the geodesics as the projections of the integral curves of
the vector field $E_1$ on $FM$. In particular we will have {\em
maximal} geodesics, \index{geodesic!maximal}
ones which cannot be extend as geodesics to a
larger interval; the domain is always an open interval. If that
interval is not all of $\R $, then there will be a Cauchy sequence
in the interval converging to a finite end. The image distances are no
farther apart, hence also a Cauchy sequence. Thus, if the Riemannian
manifold is complete, the limit point of that sequence can be used to
extend the geodesic to one more point. We have called this ability to
always extend a geodesic to a finite end of its domain {\em geodesic
completeness}. We have therefore proved:\index{completeness!geodesic}
\begin{prop} A complete Riemannian manifold is
geodesically complete. That is, maximal geodesics are defined on
the whole real line.
\end{prop}
The converse of this theorem is known as the Hopf-Rinow Theorem,
proved in its original form by H. Hopf and W. Rinow (On the concept of
complete differential-geometric surfaces, Comment. Math. Helv. \textbf{3}
(1931), 209-225.) \index{Hopf, H.}\index{Rinow, W.}
\index{Hopf-Rinow Theorem}\index{de Rham, G.}
It was recognized later, by de Rham, that it was
important to weaken the hypothesis a little more: if we assume that
just those geodesics extending from a single point can be extended
infinitely, then the Riemannian space is complete. It was that form
which has become accepted as the classical form of the Hopf-Rinow
Theorem. Strangely, an even better version was proved in 1935: S.
Cohn-Vossen, Existenz Kurzester Wege. Doklady SSSR 8 (1935), 339-342.
\index{Cohn-Vossen, S.}
The Cohn-Vossen version is applicable to locally geodesic spaces which
have maximal shortest paths of finite length; the de Rham improvement,
requiring the extendibility of only those geodesics which originate
from some single point, is also valid for Cohn-Vossen's version
\begin{thm}[The Hopf-Rinow-CohnVossen Theorem]
 \index{Hopf-Rinow-CohnVossen Theorem}If $X$ is a locally
geodesic space such that there is a point $p$ for which every maximal
geodesic through $p$ is defined on a closed interval, then $X$ is a
complete metric space.
\end{thm}

Of course, for locally compact spaces, the assumption that the space
is locally geodesic can be derived from the metric being intrinsic.

\vspace{\medskipamount}

In the same context, of locally compact intrinsic complete metric
spaces, it then follows that the space is globally geodesic.

\vspace{\medskipamount}
The modern form of the H-R-CV theorem specifies several equivalent
conditions and a consequence. This facilitates the proof of the main
implication, which was stated as the H-R-CV theorem above.
\begin{thm}[The Hopf-Rinow-CohnVossen Theorem] 
\index{Hopf-Rinow-CohnVossen Theorem}In a locally compact
intrinsic metric space $M$, the following are equivalent:
\begin{description}
\item[(i)] Every halfopen minimizing geodesic from a fixed base point
extends to a closed interval.
\item[(ii)] Bounded closed subsets are compact. (This is often said:
$M$ is finite--compact.)
\item[(iii)] $M$ is complete.
\item[(iv)] Every halfopen geodesic extends to a closed interval.
\end{description}

Any of these implies: $M$ is a geodesic space (i.e., any two points
may be joined by a shortest curve).
\end{thm}
\begin{proof} We start with an outline of proof. We establish a cycle $(i) \Rightarrow (ii)
\Rightarrow (iii) \Rightarrow (iv) \Rightarrow (i)$, and then the
final assertion is clear from $(ii)$.

\vspace{\smallskipamount}
The implication $(ii) \Rightarrow (iii)$ is true in any metric space
and we have already noted the truth of $(iii) \Rightarrow (iv)$ above;
$(iv) \Rightarrow (i)$ is trivial. Thus, the more difficult part, $(i)
\Rightarrow (ii)$ is left.

\vspace{\medskipamount}
So let us assume $(i)$ with base point $p$ and define
$$E_r = \{q: d(p,q) \le r \ \text{and there is a minimizing geodesic
from $p$ to $q$}\},$$
$$\bar B_r = \ \text{the closed metric ball of radius $r$ with center
$p$},$$
$$A = \{ r : E_r = \bar B_r \}.$$
The idea is to show that $A$ is both open and closed in $[0,\infty)$.

\vspace{\smallskipamount}
(In the Riemannian case $E_r$ is the image under $\exp_p$ of the
intersection of the domain of $\exp_p$ with the closed ball in $M_p$
of radius $r$. The proof of the fact that $E_r$ is compact then
follows by using the continuity of $\exp_p$. From there the pattern of
proof was given by deRham, copied into Bishop \& Crittenden.)
\index{Crittenden, R.J.}

\vspace{\medskipamount}
Now we give the details of the proof of the H-R-CV theorem.

\vspace{\smallskipamount}
(1) $E_r$ is compact for all $r$.

\vspace{\smallskipamount}
The set $I$ of $r$ for which $E_r$ is compact is an interval: for if
$s < r$ and $E_r$ is compact, let $(q_n)$ be a sequence in $E_s$. A
sequence of minimizing segments from $p$ to $q_n$ lies in the compact
set $E_r$, so by the Arzel\`a-Ascoli Theorem has a convergent
subsequence.\index{Arzel\`a-Ascoli Theorem} The limit curve is necessarily a minimizing geodesic, and
its endpoint is the limit of the corresponding subsequence of $(q_n)$
and is in $E_s$.

\vspace{\smallskipamount}
The interval $I$ is open: for if $E_r$ is compact, then by local
compactness we can cover it by a finite number of open sets having
compact closure. The union of these closures will contain all points
within distance $\epsilon > 0$ of $E_r$, so it will be a compact set
$K$ containing $E_{r+\epsilon}$. Repeating the above argument with
$E_r$ replaced by $K$ and $E_s$ by $E_{r+\epsilon}$ shows that the
latter is compact.

\vspace{\smallskipamount}
Suppose that $I = [0,r)$ where $r$ is finite. Let $(q_n)$ be a
sequence in $E_r$ and $\sigma_n$ a minimizing geodesic from $p$ to
$q_n$. Using the compactness of $E_s$, for $s<r$, we construct
successive subsequences $\sigma_n^m$ for each $m$ which converge on
$[0,s_m]$ where $s_m \to r$. Then the diagonal subsequence
$\sigma_n^n$ converges to a minimizing halfopen geodesic segment on
$[0,r)$. By hypothesis, this halfopen geodesic has an extension to
$r$, which provides a limit point of $(q_n)$ and a minimizing geodesic
to it from $p$. Thus, $E_r$ is compact, contradicting the assumption
that $r \notin I$. Hence, $I = [0,\infty)$, that is, all $E_r$ are
compact.

\vspace{\medskipamount}
(2) The set $A$ of all $r$ for which $E_r = \bar B_r$ is an interval.

\vspace{\smallskipamount}
Indeed, if every point at distance $r$ or less can be reached by a
minimizing geodesic from $p$, and if $s<r$, then certainly every
point of distance $s$ or less can be so reached.

\vspace{\medskipamount}
(3) $A$ is closed. If $r$ is a limit point of elements $s$ in $A$,
then $q \in \bar B_r$ is certainly a limit of points
$q_s \in \bar B_s \subset E_r$. Since $E_r$ is compact, $q \in E_r$.

\vspace{\medskipamount}
(4) $A$ is open. If $E_r = \bar B_r$, then we can cover $\bar B_r$ by
a larger compact set including some $\bar B_{r+2\epsilon}$, so the
latter is compact. For $q \in \bar B_{r+\epsilon}$ a sequence of
curves from $p$ to $q$ with lengths converging to $d(p,q)$ will
eventually be in $\bar B_{r+2\epsilon}$. By the Arzel\`a-Ascoli Theorem
there will be a subsequence converging to a minimizing segment from
$p$ to $q$, so that $q \in E_{r+\epsilon}$. Hence, $E_{r+\epsilon} =
\bar B_{r+\epsilon}$, showing that $A$ is open.
\end{proof}

\subsection{Cut points} \index{cut point}
In this section we consider Riemannian manifolds.
The notion of a cut point in a more general space, such as a complete
locally compact interior metric space, or even a complete Riemannian
manifold with boundary, is difficult to formulate so as to retain the
nice properties that cut points have in complete Riemannian manifolds.
Also, they don't make much sense in noncomplete Riemannian manifolds.

\vspace{\smallskipamount}
If $\gamma$ is a geodesic, $p = \gamma(0)$, then $q = \gamma(s)$ is a
{\em cut point} (or {\em minimum point) of p along $\gamma$} if
$\gamma$ realizes the distance $d(p,q)$, but for every $r>s$, $\gamma$
does not realize the distance $d(p,\gamma(r)$.

\begin{thm} If $q$ is a cut point of $p$, then either there are
two minimizing segments from $p$ to $q$ or $q$ is a first conjugate
point of $p$ along a minimizing segment.
\end{thm}
\begin{proof}. Let $q = \gamma(s)$ and suppose that $s_n > s$ is a
sequence converging to $s$ and that $\sigma_n$ is a curve from $p$ to
$\gamma(s_n)$ such that the length of $\sigma_n$ is less than the
length of the segment of $\gamma$ from $p$ to $\gamma(s_n)$. Since we
have assumed that the space is complete, we may as well suppose that
$\sigma_n$ is a minimizing segment. We take a convergent subsequence,
the limit of which will be a minimizing segment from $p$ to $q$. If
the limit is $\gamma$, then $\exp_p$ is not one-to-one in a
neighborhood of $s\gamma'(0)$; this shows that $q$ is a conjugate
point of $p$ along $\gamma$.
\end{proof}
\begin{cor} The relation ``is a cut point of'' is symmetric.
\end{cor}

This follows from the fact that both relations ``has two or more
minimizing geodesic to'' and ``is a first conjugate point of'' are
symmetric. The symmetry of the latter follows from the
characterization of conjugate points by Jacobi fields.

\vspace{\medskipamount}

It is not hard to show that the distance to a cut point from $p$ along
the geodesic with direction $v = \gamma'(0)$ is a continuous function
of $v$ to the extended positive reals. The manifold is compact if and
only if there is a cut point in every direction. If we remove the cut
locus of $p$ from the manifold, then the remaining subset is
diffeomorphic to a star-shaped open set in the tangent space via
$\exp_p$. In particular, the noncut locus is homeomorphic to an open
$n$-ball. This shows that all the nontrivial topology of the manifold
is conveyed by the cut locus and how it is glued onto that ball. The
cut locus is the union of some conjugate points, which have measure $0$
by Sard's Theorem, \index{Sard's Theorem}
and subsets of hypersurfaces obtained by equating
``local distance functions from $p$'', the localization being in
neighborhoods of minimizing geodesics to the cut points. Thus, the cut
locus has measure $0$.  Other than that it can be very messy; H. Gluck and
M. Singer have constructed examples for which the cut locus of some point
is nontriangulable.\index{Gluck, H.}\index{Singer, M.}
\index{cut locus!nontriangulable}

\section{Curvature and topology}\index{curvature!nonpositive}
\subsection{Hadamard manifolds} A complete Riemannian manifold with
nonpositive curvature is called an {\em Hadamard manifold}.
\index{Hadamard manifold} The
two-dimensional case was studied by Hadamard, and then his results
were extended to all dimensions by Cartan. There is a further
generalization to manifolds for which some point has no conjugate
points. We first show that nonpositive curvature implies that there
are no conjugate points. \index{conjugate point}
\begin{thm} If $M$ is a Riemannian manifold with nonpositive
curvature, then there are no conjugate points.
\end{thm}
\begin{proof} We have to show that if $Y$ is a nonzero Jacobi field
such that $Y(0) = 0$, then $Y$ is never $0$ again. Note that $Y'(0) \ne
0$, since $Y(0)$ and $Y'(0)$ are deterministic initial conditions for
the second-order Jacobi differential equation. Now differentiate
$g(Y,Y)$ twice:
$$g(Y,Y)' = 2g(Y,Y'),$$
which is $0$ at time $0$, and
$$g(Y,Y)'' = 2g(Y',Y') + g(Y,-R_{XY}X),$$
which is positive at time $0$.  But we may assume that $Y$ is
perpendicular to the base geodesic, and that the velocity field of that
geodesic is a unit vector field $X$. Then $g(Y,R_{XY}X) = K_{XY}g(Y,Y)
\le 0$. Hence, $g(Y,Y)'' \ge 0$, so that $g(Y,Y)$ only vanishes at
time $0$.
\end{proof}
\begin{thm}[The Hadamard-Cartan Theorem] 
\index{Hadamard-Cartan Theorem}If $M$ is a complete
Riemannian  manifold having a point $p$ such that $p$ has no
conjugate points, then $\exp_p$ is a covering map. If curvature is
nonpositive, then that is true for every $p$ and every fixed-end
homotopy class of curves contains a unique geodesic segment.
If $M$ is simply connected, then it is diffeomorphic to $\R _n$
via $\exp_p$.
\end{thm}
\begin{proof} We are given that $\exp_p$ is a local diffeomorphism
everywhere. From covering space theory it is sufficient to show that
it has the path-lifting property: given a curve $\gamma :[0,1] \to M$
and a point $v$ such that $\exp_p(v) = \gamma(0)$ we must show that
there is a curve $\bar \gamma$ such that $\exp_p\circ\bar\gamma =
\gamma$ and $\bar\gamma(0) = v$.

\vspace{\medskipamount}
We can pullback the metric on $M$ to get a metric $\exp_p^*g$ on
$M_p$. The radial lines from the origin are geodesics in this metric,
so that by the H-R-CV Theorem the new metric on $M_p$ is complete. We
can use the local regularity of $\exp_p$ to lift an open arc of a
curve about any point where we have it lifted already. Thus, it
becomes a matter of extending from a halfopen interval to the extra
point. On a large closed ball in $M_p$, which is compact, there will
be a positive lower bound on the amount distances will be stretched by
$\exp_p$; hence a Cauchy sequence approaching the open end in $M$ will
be lifted to a Cauchy sequence in $M_p$, providing the point to
continue the lift.
\end{proof}
\begin{thm}[Myers' Theorem] \index{Myers, S.}\index{Myers' Theorem}
Let $M$ be a complete Riemannian\index{curvature!positive}
manifold such that there is a positive number $c$ for which
$Ric(v,v) \ge (n-1)c\,g(v,v)$ for all tangent vectors $v$. Then $M$
is compact with diameter at most $\pi/\sqrt c$.
\end{thm}
\begin{proof} We show that along every geodesic there must be a
conjugate point within distance $\pi/\sqrt c$. Let $\gamma$ be a
geodesic parametrized by arc length, and let $(E_i)$ be a parallel 
frame field along $\gamma$ with $E_n = \gamma'$. We
define fields along $\gamma$ which would be Jacobi fields vanishing at
0 and $\pi/\sqrt c$ if $M$ had constant sectional curvature $c$:
$$V_i(s) = \sin\sqrt c s\,E_i(s),$$
for $i = 1,\ldots, n-1$.

Then the index form has value
$$I(V_i) = \int_0^{\pi/\sqrt c} (c\cdot \cos^2 \sqrt c s -
K_{\gamma'V_i}\sin^2 \sqrt c s)\,ds.$$
We are given that $\sum K_{\gamma'V_i} = Ric(\gamma',\gamma')
\ge (n-1)c$, so that if we add the index forms we get
$$\sum_{i=1}^{n-1}I(V_i) \le c\cdot\int\cos 2\sqrt c s \,ds = 0.$$

If there were no conjugate point on the interval in question, then the
basic inequality tells us that the index form would be positive
definite. Hence there is a conjugate point and there can be no point
at distance from $\gamma(0)$ greater than $\pi/\sqrt c$.
\end{proof}
\begin{rem} We have proved a slightly better result than claimed. We
don't have to assume that {\em all} Ricci curvatures are positively
bounded below, but only those for tangents along geodesics radiating
from a single point. Then we still get compactness, but not the
estimate on the diameter.
\end{rem}
\begin{prob}\label{prob41} The {\em radius} \index{radius of Riemannian
manifold}of a Riemannian manifold $M$
is the greatest lower bound of radii of metric balls which cover $M$.
Prove that in general $radius \,\le \, diameter \,\le \, 2\cdot radius$.
Moreover, the upper bound $radius \,\le \pi/\sqrt c$ can be obtained
from completeness and the assumption that the Ricci curvatures of
tangents along geodesics radiating from a single point have the lower
bound assumed in Myers' Theorem.\end{prob}

\vspace{\medskipamount}\index{curvature!Ricci}
When $n=2$ the condition on Ricci curvature reduces to a lower bound
$K\ge c$ on the Gaussian curvature; the conclusion of Myers' Theorem
was known for this case much earlier and this result is called
Bonnet's Theorem.\index{Bonnet's Theorem}
\subsection{Comparison Theorems}. There is an improved method of
doing comparison theorems, refining the technique of using Jacobi
fields as in Bishop \& Crittenden, employing Riccati equations as
well.\index{Riccati equation} A good reference for this approach is

\vspace{\smallskipamount}
J.-H. Eschenburg, Comparison Theorems and Hypersurfaces, Manuscripta
Math. 59(1987), 295-323.\index{Eschenburg, J.-H.}

\vspace{\medskipamount}
One of the starting points of modern comparison theory is the Rauch
Comparison Theorem.\index{Rauch Comparison Theorem}
\index{Rauch, H.} It says that if we have an inequality on sectional
curvatures at corresponding points of two geodesics, then the opposite
inequality holds for corresponding exponentiated tangent vectors. In
effect we compare growth of Jacobi fields when we are given a
curvature comparison and the same initial conditions for the Jacobi
fields.

\vspace{\smallskipamount}\index{Jacobi equation}
The reason that Riccati equations are sometimes more convenient is
that they come close to conveying just the right amount of
information, while the Jacobi equation has too much detail. If we are
concerned with estimating the distance to a conjugate point, we are
interested in whether there is a one-dimensional subspace of Jacobi
fields with zeros at two points; the Jacobi equation determines
individual members of that subspace, while the Riccati equation is
aimed at the subspace itself. For the two-dimensional case the
interest centers on the fields orthogonal to a geodesic, so that the
equations are given in terms of one scalar coefficient $f$. The Jacobi
equation is $f'' + Kf = 0$, where $K$ is the Gaussian curvature 
\index{Gaussian curvature}along
the geodesic. The Riccati equation is $h' + h^2 + K = 0$, where $h$ is
related to $f$ by $h = f'/f$. The distance between conjugate points is
the distance between two singularities of $h$. Those singularities are
all of the same sort, with $h$ approaching $-\infty$ from the left,
$+\infty$ from the right, and asymptotically $h(s)$ behaves like
$\pm(s-c)^{-1}$ at a singularity $c$. For the higher dimensional case we
deal with the vector Jacobi equation, while the Riccati equation is a
matrix Riccati equation whose solutions package all the solutions of
the Jacobi equation.

\vspace{\smallskipamount}
Geometrically the matrix of the Riccati equation represents the second
fundamental forms of a wave front, so the Riccati equation itself
expresses how the relative geometry of those wave fronts evolve as one
moves orthogonally to them. When we radiate from a point the wave
fronts are metric spheres, but the equations have the same form for
wave fronts radiating from any submanifold.

\vspace{\medskipamount}
To describe a {\em wave front},\index{wave front}
 or family of parallel hypersurface all
that is required is a real-valued function having gradient of unit
length everywhere: $f:M\to \R $, such that $g(df,df) =1$. Then the
hypersurfaces are $S_t = \{p:f(p)=t\}$, the level hypersurfaces of
$f$. We let $X = grad\,f$, the metric dual of $df$; that is, the
vector field such that $g(X,Y)= df(Y)$ for all vector fields $Y$. Then
we calculate for any tangent vector $y$:
$$0 = y\,g(X,X) = 2g(D_yX,X).$$
We can take an extension $Y$ of $y$ such that $[X,Y] = 0$ and $g(X,Y)$
is constant, hence $0 = Xg(X,Y) = g(D_XX,Y) + g(X,D_YX)$, so that
$$D_XX = 0.$$
Thus, the integral curves of $X$ are geodesics. The distances between
the hypersurfaces $S_t$ are measured along these geodesics.

\vspace{\medskipamount}
A case of particular importance is $f(q) = d(p,q)$, the distance
function defined on a deleted normal neighborhood of $p$, for which
the wave fronts are the concentric spheres about $p$.\index{concentric
spheres}

\vspace{\medskipamount}
Let $B = DX$, the Hessian \index{Hessian}tensor of $f$. 
Then $X$ is in the nullspace
of $B$, so that we will be mainly concerned with the restriction of
$B$ to the normal space $X^{\perp}$. As defined by $B = DX$, $B$ is a
linear map $Y \to D_YX$, which is the shape operator (Weingarten map)
\index{Weingarten map}\index{second fundamental form}
of the hypersurfaces $S_t$. But we also can view $B$ as the second
fundamental form, the symmetric bilinear form $(Y,Z) \to B(Y,Z) =
g(D_YX,Z)$. Usually it will be the operator version that occurs here.

\vspace{\smallskipamount}
Suppose that $J$ is a vector field orthogonal to $X$ such that
$[X,J]=0$. Then $D_JX = D_XJ = BJ$. Applying $D_X$ again we get
$$D_X(BJ) = (D_XB)J + BD_XJ = (D_XB)J + B^2J$$
$$\qquad = D_XD_JX = D_JD_XX - R_{XJ}X = - R_{XJ}X.$$
Define the symmetric linear operator $R_X$ by $R_XJ = R_{XJ}X$. Since
the vector field $J$ can have arbitrary pointwise values perpendicular
to $X$, we get the following operator Riccati equation for $B$:
$$D_XB + B^2 + R_X = 0.$$\index{Riccati equation!matrix}
Still assuming that $[X,J] = 0$, we get
$$D_XD_XJ = D_XD_JX = D_JD_XX - R_{XJ}X =-R_XJ,$$
which is the Jacobi equation; so such a $J$ must be a Jacobi field
along the integral curves of $X$. We can take $n-1$ such fields $J_1,
\ldots, J_{n-1}$ which are orthogonal to $X$ at some point, and use
them to make a linear isomorphism $\R ^{n-1} \to X^{\perp}$ which
we denote by $J$. The row of derivative fields $D_XJ = J'$ can also be
regarded as such a linear map, so that $\tilde B = J'J^{-1}$ makes
sense as a linear operator on $X^{\perp}$. Writing the definition of
$\tilde B$ as $J' = \tilde BJ$, the fact that $\tilde B$ satisfies the
same Riccati equation is just a repeat of the previous calculation.
Finally, we can rig $J$ and $J'$ at one point so that $\tilde B$
coincides with $B$ at that point, hence everywhere. This establishes
the usual relation between the Riccati equation and the corresponding
second-order linear equation, here the Jacobi equation.

\subsection{Reduction to a Scalar Equation} If we assume that $B$ has a
simple eigenvalue $\lambda$, then locally $\lambda$ will be a smooth
function and will have a smooth unit eigenvector field $U$. We show
that $\lambda$ satisfies a scalar Riccati equation\index{Riccati 
equation!scalar} of the same form,
where the driving operator $R_X$ is replaced by a sectional curvature
function.
$$X\lambda = \lambda' = g(BU,U)' = g(B'U + BU',U) + g(BU,U')$$
$$\qquad = -g(B^2U,U) - g(R_XU,U) = -\lambda^2 - K_{XU}.$$
We have used the symmetry of $B$ and the fact that
$g(BU,U') = \lambda g(U,U') = 0$, since the derivative of a unit
vector is always perpendicular to the unit vector itself.

\vspace{\medskipamount}
We cannot generally assume that the eigenvalues of $B$ will be simple.
However, we can always perturb an initial value of $B$ so that the
perturbed solutions of the Riccati equation will have simple
eigenvalues locally. Then an upper or lower bound on the eigenvalues
derived from the scalar equation can be applied to the eigenvalues of
the matrix equation by taking a limit as the perturbations go to $0$. 
This will serve our purposes even in case $B$ has multiple
eigenvalues.

\subsection{Comparisons for Scalar Riccati Equations} We consider the
Riccati equations of the form $f'= -f^2-H$, where $f,H$ are
real-valued functions of a real variable. In our geometric
applications $H$ will be the sectional curvature of a section tangent
to a geodesic, and $f$ has interpretations as a principal normal
curvature (eigenvalue of the second fundamental forms) of a wave
front, or a connection coefficient for a frame field adapted to the
setting.

\vspace{\medskipamount}
A basic trick in dealing with the scalar Riccati equation is the
change to the corresponding linear homogeneous second order equation.
We let $f = j'/j$, and then easily calculate $j'' = -Hj$. Conversely,
a solution of the second order equation leads to a solution of the
Riccati equation. Of course, $j$ is only determined up to a ratio. The
trick may be viewed as splitting the second order equation into two
first order steps, the Riccati equation and the linear equation
$j' = -fj$. We assume that $H$ is continuous.

\begin{lem} A solution $f$ on $(0,a)$ either extends
continuously to a solution in a neighborhood of 0, or $\lim_{t\to
0+}tf(t) = 1$. In either case $f$ is uniquely determined on $(0,a)$ by
its value $f(0+)$, whether finite or $+\infty$. Similarly, $f$ on
$(-a,0)$ is uniquely extendible to $f:(-a,0] \to [-\infty,+\infty)$.
\end{lem}
\begin{proof}. We must have a Taylor expansion $j(t) = c + bt + O(t^2)$
for the corresponding linear equation solution. If $c\ne 0$, $f(t) =
\frac{b+O(t)}{c+bt + O(t^2)}$ gives us the continuous extension $f(0)
= b/c$. If $c=0$, then $f(t) = \frac{b+O(t)}{bt+O(t^2)}\to +\infty$ as
$t\to 0+$ and $tf(t) \to 1$. The second order equation determines a
solution $j$ such that $j(0) = 0$ up to a constant multiple, so that
$f$ is uniquely determined when $0$ is a singularity.
\end{proof}
\begin{lem} If we determine a unique solution $f_r$ for
$r \ne 0$ by $f_r(r) = 1/r$, then the unique solution singular $f$
at $0$ is given by $f(t) = \lim_{r\to 0+}f_r(t)$.
\end{lem}
\begin{proof}. Take the limit of the corresponding solution $j_r$ for $j$,
which satisfies initial conditions $j_r(r) = r,\,j_r'(r) = 1$.
\end{proof}
\begin{thm}[Driving Function Comparison Theorem] If $f' = -f^2 -H$,
$g' = -g^2 - K$, $f(0) = g(0)$ [which may be $+\infty$], and $H\ge K$,
then $g$ exists on at least as great an interval $[0,a)$ as does $f$,
and $f\le g$ on that interval.
\end{thm}

[Note that $f$ continues until $f(t) \to -\infty$ as $t\to a-$.]

\begin{proof}. Let $h = g-f$. Then $h' + (g+f)h = -K +H \ge 0$. If
$f(0)$ is finite, we multiply both sides of this inequality by
$\exp(\int(f+g)(u)du = k$ to get $(kh)' \ge 0$. Since $h(0) = 0$ and
$k \ge 0$, we conclude that $h(t) = g(t) - f(t) \ge 0$ on $[0,a)$. But
$g$ can only become singular by going to $-\infty$, so that g must
exist on $[0,a)$ too.
\end{proof}
If $f(0) = +\infty$, then we set $f_r(r) = g_r(r) = 1/r$, use the
result just proved, and take a limit as $r\to 0+$.

\begin{thm}[The Sturm Comparison Theorem] If $j'' = -Hj$,
$k'' = -Kk$, $j(0) = k(0) = 0$, $H\ge K$, and these solutions are not
trivial, then the next $0$ of $j$ occurs at or before the next $0$ of
$k$.\index{Sturm Comparison Theorem}
\end{thm}

\begin{thm}[Value Comparison Theorem] If $f' = -f^2-H$, $g' = -g^2-H$,
and $f(0) \le g(0)$, then $f\le g$ on the maximal interval $[0,a)$ on
which $f$ exists.
\end{thm}
\begin{proof} Again let $h = g-f$, so that $h'+(g+f)h = 0$. Clearly
$h\ge 0$ on $[0,a)$.
\end{proof}
\begin{thm}[Rauch Comparison Theorem] \index{Rauch Comparison
Theorem}Let M and N be Riemannian
manifolds, $\gamma$ and $\sigma$ unit speed geodesics in each,
$X = \gamma'$ and $Y = \sigma'$ their unit tangent vector fields.
Suppose that for every pair of vector fields $Z$ and $W$ orthogonal
to $\gamma$ and $\sigma$, respectively, we have an inequality on
sectional curvatures at corresponding points: $K_{XZ} \le K_{YW}$.
Let $J$ and $L$ be nonzero Jacobi fields orthogonal to $\gamma$
and $\sigma$, respectively, such that $J(0)= 0$, $L(0) = 0$, and
$J'(0)$, $L'(0)$ have the same length. Then $g_M(J,J)/g_N(L,L)$ is
nondecreasing for $s>0$; in particular, $J$ is at least as long as
$L$.
\end{thm}
\begin{proof} We start by calculating the logarithmic derivative of
the length of a Jacobi field $J$:
$$(\log g(J,J))' = 2g(J',J)/g(J,J) = g(BJ,J)/g(J,J).$$
For Jacobi fields vanishing at an initial point the operators $B$ are
the shape operators of the spherical wave front\index{wave 
front!spherical} about that point. On
both manifolds these operators have a simple pole $B \sim (1/s)I$ as
their initial conditions. By the driving function comparison theorem
the eigenvalues of the operators for the spheres on $M$ and $N$ are
related oppositely to the relation for curvature. That inequality on
eigenvalues is then passed on to an inequality for the logarithmic
derivatives, and we have supposed that $J$ and $L$ are asymptotically
the same at $s=0$.
\end{proof}

Note that we did not have to assume that the dimensions are the same.
The most common application is a comparison to constant curvature
spaces, which can be stated conveniently as follows.
\index{bounded curvature}\index{curvature!bounded}

\begin{thm}[Constant Curvature Comparison Theorem] Suppose that the
sectional curvatures of $M$ are bounded by constants: $a \le K \le b$.
Let $S(a)$ and $S(b)$ be the simply connected complete Riemannian
manifolds of constant curvatures $a$ and $b$, of the same dimension as
$M$. Let $\exp_a$, $\exp_b$, and $\exp_p$ be exponential maps for
$S(a)$, $S(b)$, and $M$, respectively, each restricted to a normal
neighborhood. Then $\exp_p\circ\exp_a^{-1}$ is length nonincreasing
and $\exp_p\circ\exp_b^{-1}$ is length nondecreasing. (For convenience
we have identified the three tangent spaces by some Euclidean
isometry.)
\end{thm}

To keep the directions of the inequalities correct you should always
bear in mind particular comparisons, say of the Euclidean plane with
the unit sphere: it is easy to visualize that the Euclidean lines
spread apart faster than great circles making the same initial angle.

\vspace{\medskipamount}
An equivalent way of viewing the constant curvature comparison is in
terms of triangles. \index{triangle comparison}
The triangles compared should be sufficiently
small so that they lie in a normal coordinate neighborhood and in the
sphere are uniquely determined by the three side lengths. For a given
triangle in $M$ the {\em comparison triangle} is the triangle in the
constant curvature surface $S(a)$ having the same side lengths. Then
an inequality on curvatures, say, $a \le K$, is conveyed by
inequalities between the angles and corresponding distances across the
two triangles: the angles are smaller and the distance shorter in the
comparison triangle than in the given triangle.

\vspace{\medskipamount}
Alternatively, instead of making the three sides the same, one can
make two sides and the included angle the same in the given triangle
and the comparison triangle, with obvious consequent inequalities
between the other corresponding ``parts'' of the triangles. This called
{\it hinge comparison}\index{hinge comparison}

\vspace{\medskipamount}
Alexandrov \index{Alexandrov, A.D.}\index{Alexandrov space}
has turned these triangle 
comparisons into definitions, for\index{singular space}
geodesic metric spaces, of what it means for the space to have
curvature bounded above or below by a constant. This allows an
extension of many ideas of Riemannian geometry to ``singular'' spaces.
For example, he proves that if curvature is bounded above, then the
angle between two geodesic rays with a common starting point is
well-defined and satisfies many of the usual properties. However, an
angle and its supplementary angle has sum $\ge \pi$, but equality may
fail. The metric completion of the logarithm spiral surface 
\index{logarithmic spiral surface}covering
the punctured Euclidean plane has curvature $\le 0$, but geodesic rays
starting at the singular point can have arbitrarily large angle
between them. Generally, in spaces with curvature bounded above
geodesics may bifurcate (which is an indication of some infinitely
negative curvature), but locally a geodesic segment is uniquely
determined by its ends.

\vspace{\medskipamount}\index{curvature!positive}
The opposite case of spaces with curvature bounded below has also been
studied. Here again angles are meaningful; for the two-dimensional
case, the sum of angles about a point can be at most $2\pi$, and if it
is less, the point is regarded as having positive curvature measure.
Geodesics cannot bifurcate, but local bipoint uniqueness may fail and
indicate positive infinite curvature. Examples of this sort are obtain
by gluing two copies of a convex Euclidean set along their boundaries
(the {\em double} of the set).

\vspace{\medskipamount}
When a locally compact metric space has curvature bounded both above
and below then it is very close to being a manifold. To make it be a
manifold we only have to assume one further very natural property:
geodesics must be locally extendible. With this hypothesis, Nikolaev
\index{Nikolaev, I.G.}
proved that there is a $C^{3,\alpha}$ manifold structure and the
metric is given by a $C^{1,\alpha}$ Riemannian metric. The number
$\alpha$ is a H\"older exponent for the last derivatives, and can be
any number between $0$ and $1$.

\vspace{\medskipamount}
Curvature bounds defined in Alexandrov's way are easily seen to be
inherited by the limits of spaces, for some reasonable notions of such
limits. Thus, Nikolaev's theorem has an important consequence that
some limits of Riemannian manifolds are actually Riemannian manifolds.
There is a general compactness theorem of Gromov which shows
that many such limits exist.\index{Gromov, M.}\index{Gromov
Compactness Theorem}

\vspace{\medskipamount}
Another kind of result which Alexandrov was able to abstract to spaces
with curvature bounded above was the proof that one could compare
certain global triangles, given that local ones could be compared. For
example, he proved a generalization of the Hadamard-Cartan Theorem to
\index{Hadamard-Cartan Theorem}
locally compact complete geodesic metric spaces with curvature bounded
above by $0$.  Recently S. Alexander and I generalized this even more,
using instead the weaker assumption of geodesic convexity and
eliminating the hypothesis of local compactness.\index{Alexander, S.}

\subsection{Volume Comparisons} For a Riemannian manifold we can get
comparisons between the volume of balls and spheres (and more
generally, tubes) and corresponding volumes in constant curvature
spaces founded on curvature inequalities. For {\em lower} bounds on
volume it is hard to do much better than to assume upper bounds on
{\em sectional} curvature and apply the length nondecreasing maps that
we get from the Rauch comparison theorem. The more interesting case is
to get {\em upper} bounds on volumes from the weaker assumption of
lower bounds on {\em Ricci} curvature.\index{Ricci curvature}
\index{curvature!Ricci}

\begin{thm}[Bishop's Volume Comparison Theorem] \index{Bishop's
Volume Comparison Theorem}If
$Ric(X,X)\ge(n-1)K$ for all unit vectors $X = grad(d(p,\cdot))$,
then for each ball $B_p(r)$ and sphere$S_p(r) = \partial B_p(r)$
the Riemannian volume, n-dimensional and (n-1)-dimensional,
respectively, is less than or equal to that for a ball or sphere of
the same radius in a space of constant curvature $K$.
\end{thm}

For $K>0$ the result holds for all $r\le \pi/\sqrt K$; for $K \le 0$
the result holds for all $r$. In either case it is permissible to let
$\exp_p$ be noninjective while its counterpart in the constant
curvature space is injective, since counting parts of volume more than
once enhances the inequality. This refinement is now attributed to
Gromov, but I knew it and thought it was so trivial as to be
unnecessary to say explicitly. However, it turned out to be important
in applications.

\vspace{\medskipamount}
The method of proof is to estimate the Jacobian determinant of the
exponential map, by calculating the logarithmic derivative. That much
is similar to the proof of the Rauch theorem. We express that Jacobian
determinant in terms of the length of an $(n-1)$-vector:
$$J_1\wedge\ldots\wedge J_{n-1} = j\cdot E_1\wedge\ldots\wedge
E_{n-1}.$$
This was done directly in the original proof, given in Bishop \&
Crittenden, Chapter 11. \index{Crittenden, R.J.}
Now the fashion is to use an operator Riccati
equation as intermediary, converting it to a scalar Riccati equation
for $j'/j$ by taking the trace.\index{Riccati equation}

\printindex

\end{document}